\documentclass[11pt, oneside]{amsart}

\usepackage[british,english]{babel} 
\usepackage{graphics,color,pgf}
\usepackage{epsfig}
\usepackage[ansinew]{inputenc}
\usepackage[all]{xy}
\usepackage{hyperref}
\newdir{ >}{!/8pt/@{}*@{>}}
\usepackage{amssymb, amsmath,amsthm, mathtools, amscd}
\usepackage{mathrsfs}
\usepackage{stmaryrd}
\usepackage[margin=1.5in]{geometry}

\usepackage{enumerate}

\theoremstyle{plain}
\newtheorem{teor}{Theorem}[section]
\newtheorem{lem}[teor]{Lemma}
\newtheorem{cor}[teor]{Corollary}
\newtheorem{prop}[teor]{Proposition}

\theoremstyle{definition}
\newtheorem{deft}[teor]{Definition}
\newtheorem{es}[teor]{Example}

\theoremstyle{remark}
\newtheorem{oss}[teor]{Remark}

\newcommand{\R}{\mathbb{R}}
\newcommand{\Z}{\mathbb{Z}}

\newcommand{\bbS}{\mathbb{S}}

\newcommand{\hyp}{\mathbb{H}}
\newcommand{\Vol}{\textup{Vol}}
\newcommand{\uph}{\textup{H}}

\newcommand{\hcb}{\textup{H}_{cb}}

\newcommand{\homs}{\textup{Homeo}^+(\mathbb{S}^1)}
\newcommand{\comp}{\textup{comp}}
\newcommand{\eu}{\textup{eu}}
\newcommand{\po}{\textup{PO}}
\newcommand{\Isom}{\textup{Isom}}

\newcommand{\trans}{\textup{trans}}
\newcommand{\orien}{\mathfrak{o}}

\title[A Matsumoto/Mostow result for hyperbolic cocycles]{A Matsumoto-Mostow result for Zimmer's cocycles of hyperbolic lattices}

\author[]{M. Moraschini}
\address{Fakult\"{a}t f\"{u}r Mathematik, Universit\"{a}t Regensburg, 93040 Regensburg, Germany}
\email{marco.moraschini@ur.de}

\author[]{A. Savini}
\address{Department of Mathematics, University of Bologna, Piazza di Porta San Donato 5, 40126 Bologna, Italy}
\email{alessio.savini5@unibo.it}

\date{\today.\ \copyright{\ M.~Moraschini, A. Savini 2019}.
  The first author was supported by the CRC~1085 \emph{Higher Invariants}
  (Universit\"at Regensburg, funded by the DFG). The second author was partially supported by the project \emph{Geometric and harmonic analysis with applications}, funded by EU Horizon 2020 under the Marie Curie grant agreement No 777822.}

\begin{document}

\begin{abstract}
Following the philosophy behind the theory of maximal representations, we introduce the volume of a Zimmer's cocycle $\Gamma \times X \rightarrow \po^\circ(n, 1)$, where $\Gamma$ is a torsion-free (non-)uniform lattice in $\po^\circ(n, 1)$, with $n \geq 3$, and $X$ is a suitable standard Borel probability $\Gamma$-space. Our numerical invariant extends the volume of representations for (non-)uniform lattices to measurable cocycles and in the uniform setting it agrees with the generalized version of the Euler number of self-couplings. We prove that our volume of cocycles satisfies a Milnor-Wood type inequality in terms of the volume of the manifold $\Gamma \backslash \mathbb{H}^n$. 
Additionally this invariant can be interpreted as a suitable multiplicative constant between bounded cohomology classes. This allows us to define a family of measurable cocycles with vanishing volume. The same interpretation enables us to characterize maximal cocycles for being cohomologous to the cocycle induced by the standard lattice embedding via a measurable map $X \rightarrow \po(n, 1)$ with essentially constant sign.

As a by-product of our rigidity result for the volume of cocycles, we give a different proof of the mapping degree theorem. This allows us to provide a complete characterization of maps homotopic to local isometries between closed hyperbolic manifolds in terms of maximal cocycles.

In dimension $n = 2$, we introduce the notion of Euler number of measurable cocycles associated to a closed surface group and we show that it extends the classic Euler number of representations. Our Euler number agrees with the generalized version of the Euler number of self-couplings up to a multiplicative constant. Imitating the techniques developed in the case of the volume, we show a Milnor-Wood type inequality whose upper bound is given by the modulus of the Euler characteristic of the associated closed surface. This gives an alternative proof of the same result for the generalized version of the Euler number of self-couplings. Finally, using the interpretation of the Euler number as a multiplicative constant between bounded cohomology classes, we characterize maximal cocycles as those which are cohomologous to the one induced by a hyperbolization. 
\end{abstract}

\maketitle


\section{Introduction}

\subsection{Historical background} 

The study of lattices in semisimple Lie groups of non-compact type has many applications both to algebra and geometry. One of the most remarkable properties is given by their rigidity. The investigation of  rigidity of lattices was initiated by Mostow~\cite{mostow68:articolo}, who proved that in dimension $n \geq 3$ any two isomorphic torsion-free uniform lattices $\Gamma_1, \Gamma_2 <\po^\circ(n,1)$ must be conjugated by an element $g$ in  $\po(n,1)$. This phenomenon is known as \emph{Mostow rigidity}. Equivalently, from a geometric point of view, Mostow rigidity implies that any  two closed hyperbolic $n$-manifolds $M_1, M_2$ with isomorphic fundamental groups are isometric. Later this statement has been extended by Prasad~\cite{prasad:articolo} to complete finite-volume hyperbolic manifolds via the study of non-uniform lattices in $\po^\circ(n,1)$, whence the name of \emph{Mostow-Prasad rigidity}.

The study of rigidity properties of lattices has then developed in several different directions~\cite{margulis:super, margulis:ext, zimmer:annals, fisher:morris:whyte, MonShal, kida1, kida2, sauer:articolo, bucher2:articolo}. For instance, Prasad's work \cite{prasad:articolo} covers also non-uniform lattices in rank one Lie groups of non-compact type. Successively Margulis extended Mostow-Prasad rigidity to the higher rank case \cite{margulis:ext}. In fact in the latter setting Margulis proved actually a stronger result, known as \emph{Margulis superrigidity}~\cite{margulis:super}. Margulis' approach was then generalized to the theory of measurable cocycles by Zimmer~\cite{zimmer:annals}. In this wider setting Mostow rigidity can be translated in terms of \emph{tautness} of groups. Recall that a unimodular locally compact second countable group is \emph{taut} if every $(G, G)$-coupling can be trivialized to the tautological coupling. The notion of tautness may be interpreted as generalization of Mostow rigidity in the following sense: any lattice of a taut group satisfies Mostow rigidity~\cite{sauer:articolo}. Note that since any $(G, H)$-coupling $(\Omega, m_\Omega)$  has an associated right measure equivalence cocycle, we may interpret the theory of couplings as an instance of Zimmer's one. Indeed, a $(G, G)$-coupling $(\Omega, m_\Omega)$ is taut if and only if its right measure equivalence cocycle $\alpha_\Omega$ is cohomologous to the cocycle associated to the standard lattice embedding via a suitable measurable map. This correspondence shows once more the relation between tautness and Mostow rigidity. In that direction, it is worth mentioning the results of Monod and Shalom~\cite{MonShal} about cocycle superrigidity and the tautness of certain groups which arise as products, and Kida who proved in \cite{kida1, kida2} the tautness of certain mapping class groups. Bader, Furman and Sauer ~\cite{sauer:articolo} studied the $1$-tautness of $\po(n,1)$ when $n \geq 3$. Here $1$-tautness refers to some additional integrability assumptions on the couplings involved in the definition. If $n \geq 3$, the group $\po(n,1)$ is indeed $1$-taut and hence any of its lattices is $1$-taut relative to its inclusion in $\po(n,1)$ (see~\cite[Proposition~2.9]{sauer:articolo}) . Recently, some other examples of taut groups have arised in the literature~\cite{chifan:kida, bowen:taut}, and some other rigidity results have been obtained~\cite{Austin:GAFA, cantrell, gelander:levit, savini:tautness}.

%
A source of inspiration for this work was the proof of the $1$-tautness of $\po(n,1)$ exposed by Bader, Furman and Sauer \cite{sauer:articolo}. For this reason, we briefly recall their strategy. Working with a torsion-free \emph{uniform} lattice $\Gamma \leq \po^\circ(n, 1)$, Bader, Furman and Sauer define a numerical invariant associated to a $(\Gamma, \Gamma)$-coupling $(\Omega, m_\Omega)$, called Euler number~$\eu(\Omega)$. They show that the Euler number of self-couplings satisfies the following Milnor-Wood type inequality $|\eu(\Omega)| \leq \Vol(\Gamma \backslash \mathbb{H}^n)$. Under an additional integrability hypothesis, the study of maximal $(\Gamma, \Gamma)$-couplings leads to the desired conjugation, whence the $1$-tautness of $\po(n, 1)$.

A crucial step in their proof is the description of the Euler number via a sort of proportionality principle~\cite[Lemma~4.10]{sauer:articolo}. A careful reading of that proof and the involved diagrams shows that one can extend the notion of Euler number to arbitrary measurable cocycles $\Gamma \times X \rightarrow \po^\circ(n, 1)$, where $\Gamma$ is a uniform lattice and $(X, \mu_X)$ is a standard Borel probability $\Gamma$-space. More precisely, one may drop both the restrictions on the probability space and on the target group to obtain results similar to the ones proved by Bader, Furman and Sauer \cite{sauer:articolo} (see Section~\ref{subsec:nuova:uniforme}). We will refer to those results with our new hypothesis as \emph{generalized} Bader-Furman-Sauer's results (see again Section~\ref{subsec:nuova:uniforme}).

In the case of torsion-free \emph{non-uniform} lattices of $\po(n, 1)$, one could study rigidity via maximal representations. This successful approach was initiated by Bucher, Burger and Iozzi~\cite{bucher2:articolo}. Their techniques are based on an accurate  study of the bounded cohomology groups of $\po(n,1)$, still for $n \geq 3$. The authors first show that the volume cocycle $\Vol_n$ defined on $\bbS^{n-1}$ is bounded, alternating and $\po(n,1)$-invariant. Then, following the functorial approach to bounded cohomology developed by Burger and Monod \cite{monod:libro, burger2:articolo}, they show that $\Vol_n$ canonically determines a cohomology class $[\Vol_n]$ which generates $\hcb^n(\po(n,1);\R_\varepsilon) \cong \mathbb{R}$. Here $\R_\varepsilon$ denotes the $\po(n,1)$-module endowed with the sign action (see Section~\ref{sec:cbc} for a precise definition). The crucial idea is that one may consider the pullback of the volume class along any representation $\rho:\Gamma \rightarrow \po^\circ(n,1)$ of a torsion-free non-uniform lattice $\Gamma < \po^\circ(n,1)$. The pairing between the pullback class with the relative fundamental class of $\Gamma \backslash \mathbb{H}^n$ gives rise to a numerical invariant called \emph{volume of the representation} $\rho$. The volume is invariant with respect to the conjugacy by elements of $\po^\circ(n,1)$ and hence it provides a well-defined continuous function on the character variety $X(\Gamma,\po^\circ(n,1))$ with respect to the topology of pointwise convergence. Note that when $n =3$ this numerical invariant agrees with the one introduced by both Dunfield~\cite{dunfield:articolo} and Francaviglia~\cite{franc04:articolo} as the pullback of the volume form along any pseudo-developing map (a proof of the equivalence of these definitions was given by Kim~\cite{kim:articolo}). 

Bucher, Burger and Iozzi \cite{bucher2:articolo} also investigate maximal representations. They introduce a Milnor-Wood inequality showing that any representation $\rho \colon \Gamma \rightarrow \po^\circ(n,1)$ satisfies $\lvert \Vol(\rho)\rvert \leq \Vol(\Gamma \backslash \mathbb{H}^n)$. Hence the study of maximal representations can be translated in terms of rigidity properties of representations. Keeping the dimensional hypothesis $n \geq 3$, they obtain that maximal representations must be conjugated to the standard lattice embedding $\Gamma \rightarrow \po^\circ(n,1)$ via an element in $\po(n,1)$. This is a suitable adaptation of Mostow-Prasad rigidity to the context of representations. In a more general setting, Francaviglia and Klaff~\cite{franc06:articolo} proved some similar rigidity results for their definition of volume of a representation $\Gamma \rightarrow \po(m,1)$, this time assuming $m \geq n \geq 3$ (the rigidity of volume actually holds also at infinity, as proved by Francaviglia and the second author~\cite{savini:articolo} for the real hyperbolic lattices. Moreover, the second author also showed that the rigidity holds for complex and quaternionic lattices~\cite{savini2:articolo}). The interest in the study of volume of representations has recently grown, leading to the development of a rich literature~\cite{Pozzetti, KimKim:proc, Tholozan, farre2, farre1}.

On the other hand, one could naturally ask what happens for lattices in $\textup{PSL}(2,\R) \cong \Isom^+(\mathbb{H}^2)$. It is well-known that Mostow rigidity does not hold in the two-dimensional case (and, whence $\textup{PSL}(2,\R)$ is not taut by itself). For instance, if $\Gamma_g$ is the fundamental group of a closed surface $\Sigma_g$ of genus $g \geq 2$, we know via Teichm\"{u}ller theory that there exists a space of real dimension $6g-6$ of inequivalent discrete and faithful representations of $\Gamma_g$ into $\textup{PSL}(2,\R)$.

However, in dimension $n = 2$, working with uniform lattices, Bader, Furman and Sauer~\cite{sauer:articolo} show that $\textup{PSL}(2,\R)$ is in fact taut \emph{relatively} to the natural embedding of $\textup{PSL}(2,\R)$ into $\homs$. Their strategy is similar to the proof of the $1$-tautness of $\po^\circ(n,1)$, whence via the study of the Euler number of maximal self-couplings. 

By changing perspective and following the ideas of the study of representations, Ghys~\cite{ghys:articolo} noticed that the Euler class $\textup{e} \in \uph^2(\homs;\Z)$ is actually a bounded class and hence it determines a class~$\textup{e}_b\in \uph_b^2(\homs;\Z)$ in the bounded cohomology group. Therefore, given a representation~$\rho \colon \Gamma_g \rightarrow \homs$, one can still pullback the Euler class and define a numerical invariant via the Kronecker product. More precisely, Ghys defined the Euler number of $\rho$ as the number $\eu(\rho)=\langle \rho^*_b(\textup{e}_b),[\Sigma_g] \rangle$, where $\langle \cdot, \cdot \rangle$ denotes the Kronecker product and $[\Sigma_g]$ is the fundamental class of $\Sigma_g$. Since Ghys showed that the pullback class is a total invariant for the semiconjugacy class of $\rho$, the Euler number is constant along the semiconjugacy class of $\rho$.  Recall that $\rho_1,\rho_2:\Gamma_g \rightarrow \homs$ are \emph{semiconjugated} if there exists a degree one monotone map $f \in \homs$ such that $f(\rho_1(\gamma) \xi)=\rho_2(\gamma)f(\xi)$, for every $\gamma \in \Gamma_g$ and $\xi \in \, \mathbb{S}^1$.

One can follow Ghys' approach to study maximal representations. Indeed, after the works of Milnor~\cite{milnor:articolo} and Wood~\cite{wood}, we know that any representation $\rho$ satisfies the key estimate $|\eu(\rho)| \leq |\chi(\Sigma_g)|$. The challenging problem about the characterization of maximal representations was solved by Matsumoto in~\cite{matsumoto:articolo}. Indeed, he proved that the maximal value of $\eu$ detects the semiconjugacy class of a hyperbolization $\pi_0:\Gamma_g \rightarrow \textup{PSL}(2,\R)$. We stress the analogy between this statement and the rigidity results~\cite{bucher2:articolo, sauer:articolo} when the dimension satisfies~$n \geq 3$. Note that this correspondence is not only formal, since one may in fact investigate these problems with similar techniques. Indeed, for instance, Iozzi~\cite{iozzi02:articolo} provided a new proof of Matsumoto's theorem using bounded cohomology.

\subsection{Volume of measurable cocycles}
In this paper, we propose to introduce a numerical invariant of measurable cocycles associated to torsion-free (non-)uniform lattices in $\po^\circ(n,1)$. On the one hand, it can be considered as an extension of the \emph{generalized} Bader-Furman-Sauer's Euler number to the non-uniform case (see Section~\ref{subsec:nuova:uniforme}). On the other hand, our invariant also extends the classical volume of representations to the more general setting of measurable cocycles (both in the uniform and non-uniform cases). Indeed given a representation $\rho \colon \Gamma \rightarrow \po^\circ(n,1)$ there exists a canonical way to define an associated measurable cocycle $\sigma_\rho: \Gamma \times X \rightarrow \po^\circ(n,1)$, for any  standard Borel probability $\Gamma$-space $(X,\mu_X)$. We prove in Proposition~\ref{prop:rep:vol}, that the volume of $\rho$ agrees with our invariant of $\sigma_\rho$. For this reason, despite our invariant is inspired by Bader-Furman-Sauer's Euler number, we prefer to call it \emph{volume of measurable cocycles} associated to a (non-)uniform lattice $\Gamma < \po^\circ(n,1)$. 

Let $n \geq 3$ and let $\Gamma$ be a torsion free (non-)uniform lattice in $\po^\circ(n,1)$ and let $\sigma \colon \Gamma \times X \rightarrow \po^\circ(n,1)$ be a measurable cocycle, where $(X, \mu_X)$ is a standard Borel probability $\Gamma$-space. We study measurable cocycles that admit an essentially unique equivariant boundary map $\phi:\bbS^{n-1} \times X \rightarrow \bbS^{n-1}$ (this is the case of \emph{non-elementary measurable cocycles} as proved by Monod and Shalom \cite[Proposition 3.3]{MonShal0}). We define the \emph{volume of the cocycle} $\sigma$ via the pullback of the volume cocycle $\Vol_n$ along $\phi$ (see Definition~\ref{def:per:intro:volume}). Much more generally we show in Section~\ref{sec:lemma:tecnico} that an equivariant boundary map allows to pullback an essentially bounded $\po(n,1)$-invariant Borel cocycle to obtain suitable classes in the bounded cohomology groups $\textup{H}^\bullet_b(\Gamma;\R)$. 

The first insight into the rigidity of measurable cocycles via our volume is described by Proposition~\ref{gplus:cohomology}, where we show that the volume is in fact invariant along the $\po^\circ(n,1)$-cohomology class. Moreover, following the general philosophy of maximal representations, we prove a Milnor-Wood type inequality in order to get a stronger rigidity result for measurable cocycles. Indeed, we extend not only both the Milnor-Wood type inequalities by Bucher, Burger and Iozzi~\cite{bucher2:articolo} and by Bader, Furman and Sauer~\cite{sauer:articolo}, but we completely characterize maximal cocycles as follows:

\begin{teor}\label{mostow}
Let $\Gamma < \po^\circ(n,1)$ be a torsion-free non-uniform lattice. Let $(X,\mu_X)$ be a standard Borel probability $\Gamma$-space. Let $\sigma:\Gamma \times X \rightarrow \po^\circ(n,1)$ be a cocycle with an essentially unique boundary map $\phi:\bbS^{n-1} \times X \rightarrow \bbS^{n-1}$. Then, we have
\[
|\Vol(\sigma)| \leq \Vol(\Gamma \backslash \hyp^n) \ ,
\]
and equality holds if and only if $\sigma$ is cohomologous to the cocycle associated to the standard lattice embedding $i: \Gamma \rightarrow \po^\circ(n,1)$ via a measurable function $f \colon X \rightarrow \po(n,1)$ with essentially constant sign.  
\end{teor}

It would be nice to describe the family of cocycles having extremal values explicitly. To that end, we introduce the family of \emph{reducible cocycles} and we prove that they have vanishing volume (see Example~\ref{es:reducible:cocycles}). Despite the theorem above is stated for non-uniform lattices, our construction still holds in the uniform case as explained in Remark~\ref{oss:everything:uniform:case} (compare with Section~\ref{subsec:nuova:uniforme}). Working with uniform lattices and using some results available in the literature \cite{bucher2:articolo, sauer:articolo}, we can describe two families of maximal cocycles: the one arising from maximal representations (and this in fact also holds in the non-uniform case) and the one arising from ergodic \emph{integrable} self-couplings (see Corollary~\ref{cor:value:coupling}).

We briefly explain here the idea of the proof of Theorem~\ref{mostow}, which follows the scheme of a theorem by Bader, Furman and Sauer~\cite[Theorem B]{sauer:articolo} (compare also with a result by the second author~\cite[Theorem 1]{savini3:articolo}). A crucial step in the proof is given by the following proposition which expresses the volume of a measurable cocycle as a suitable multiplicative constant between bounded cohomology classes.

\begin{prop}\label{prop:formula:volume}
Let $n \geq 3$ and let $G = \Isom(\mathbb{H}^n)$. Let $\Gamma < G^+$ be a torsion-free non-uniform lattice and let $(X, \mu_X)$ be a standard Borel probability $\Gamma$-space. Let $\sigma \colon \Gamma \times X \rightarrow G^+$ be a measurable cocycle with essentially unique boundary map $\phi \colon \mathbb{S}^{n-1} \times X \rightarrow \mathbb{S}^{n-1}$. If $M = \Gamma \backslash \mathbb{H}^n$, then we have
$$\int_{\Gamma \backslash G} \int_X \varepsilon(\overline{g}^{-1}) \cdot  \Vol_n(\phi( \overline{g} \cdot \xi_0, x),\cdots,\phi(\overline{g} \cdot \xi_n,x))d\mu_X(x) d\mu (\overline{g})=\frac{\Vol(\sigma)}{\Vol(M)}\Vol_n(\xi_0, \cdots ,\xi_n) \ , \nonumber 
$$
for almost every $(\xi_0, \cdots ,\xi_n) \in (\bbS^{n-1})^{n + 1}$, where $\mu$ denotes the normalized probability measure induced on the quotient $\Gamma \backslash G$. Here $\varepsilon(\overline{g}^{-1})$ denotes the sign of the element $\overline{g}^{-1} \in \, \Gamma \backslash G$.
\end{prop}

Using the formula reported above we show that when the cocycle $\sigma$ is positively maximal, that is $\Vol(\sigma)=\Vol(\Gamma \backslash \hyp^n)$, the associated measurable map $\phi_x:\bbS^{n-1} \rightarrow \bbS^{n-1}, \phi_x(\xi):=\phi(\xi,x)$ sends almost every regular ideal tetrahedron to another regular ideal tetrahedron with the same orientation, for almost every $x \in X$. This is enough to prove that $\phi_x$ coincides essentially with an isometry in $\po^\circ(n,1)$ for almost every $x \in X$. Then, we use the previous result to construct a map $f:X \rightarrow \po^\circ(n,1)$ which is measurable by Fisher, Morris and Whyte \cite{fisher:morris:whyte} and thus it realizes the desired conjugation of our cocycle. This strategy works fine in the case of positive maximal volume and can be suitably adapted to the case of negative maximal volume. 

\subsection{Maximal cocycles and local isometries}

Thanks to Theorems \ref{mostow} we have a complete description of maximal Zimmer's cocycles. In this paper we are going to show how the study of maximal cocycles can be suitably used to characterize maps between closed hyperbolic manifolds of dimension $n \geq 3$ which are homotopic to local isometries. This characterization will be related to the well-known mapping degree theorem. A first proof of the mapping degree theorem was given by Kneser in the case of  surfaces~\cite{kneser}. He showed that given a map $f \colon \Sigma_g \rightarrow \Sigma_{g'}$ between two closed surfaces of genus $g$ and $g'$, the following bound on the mapping degree of $f$ holds 
$$
|\deg(f) | \leq \frac{\chi(\Sigma_g)}{\chi(\Sigma_{g'})} \ .
$$
Note that when $\Sigma_g$ and $\Sigma_{g'}$ are both endowed with a hyperbolic structure, by Gauss-Bonnet theorem we can substitute the Euler characteristics appearing in the estimate above with the areas of $\Sigma_g$ and $\Sigma_{g'}$. More generally the mapping degree theorem states that the latter estimate holds in any dimension. Formally, given a map $f \colon M_1 \rightarrow M_2$ between closed hyperbolic manifolds of the same dimension, we have the following bound
$$
|\deg(f)| \leq \frac{\Vol(M_1)}{\Vol(M_2)} \ .
$$
Moreover, Thurston showed that in dimension $n \geq 3$ the equality holds if and only if $f$ is homotopic to a local isometry \cite[Theorem~6.4]{Thurston}. This generalizes  Kneser's theorem~\cite{kneser} to the higher dimensional case.

There exist several different proofs of the mapping degree theorem. For instance Besson, Courtois and Gallot \cite{bcg95,bcg96,bcg98} used the notion of natural map to get a proof in the case of locally symmetric rank one closed manifolds. Their techniques were then extended by Connell and Farb to the higher rank case \cite{connellfarb1,connellfarb2}. 

A different approach via $\ell^1$-homology and simplicial volume was introduced by Thurston~\cite[Theorem~6.2.1 and Theorem~6.4]{Thurston} and Gromov~\cite[Section~0.2]{Grom82}. The simplicial volume is a homotopy invariant of compact manifolds that measures the complexity of such a manifold in terms of the $\ell^1$-norm of (real) singular chains. It can be interpreted as the geometric dual of the theory of bounded cohomology. Their approach is based on the crucial proportionality principle \cite[Theorem~6.2]{Thurston} (see also~\cite{GrMo}), which says that the simplicial volume of a closed hyperbolic manifold is proportional to the Riemannian one up to a multiplicative constant only depending on the dimension (see~\cite{Thurston, Grom82, BePe} for a detailed description on this topic). 

Recently, many results about mapping degree via simplicial volume and bounded cohomology have arised in the literature~\cite{LK:degree, Loh-Sauer, bucher2:articolo, Neof1, Neof2, FM:ideal, derbez:liu:sun:wang}. Among these applications, we are primarily interested in the works by Bucher, Burger and Iozzi \cite{bucher2:articolo} and by Derbez, Liu, Sun and Wang~\cite{derbez:liu:sun:wang}. Indeed, on the one hand Bucher, Burger and Iozzi show that one can easily reprove Thurston's theorem using the volume of representations~\cite[Corollary~1.3]{bucher2:articolo}. However, it seems that their proof is still related to simplicial volume, once they have reproved Gromov and Thurston's proportionality principle. On the other hand, Derbez, Liu, Sun and Wang prove a stronger result~\cite[Proposition~3.1]{derbez:liu:sun:wang}. In  a simplified version, given a map $f \colon M_1 \rightarrow M_2$ between to closed hyperbolic manifold of the same dimension, they define the pullback of a representation $\rho \colon \pi_1(M_1) \rightarrow \po^\circ(n,1)$. Then, the volume of the \emph{pullback representation} turns out to be proportional to the one of $\rho$ up to a multiplicative constant which agrees with the mapping degree of $f$.

In this paper, we prove the following technical result which will allow us to characterize local isometries in terms of maximal cocycles. 

\begin{prop}\label{prop:moltiplicativita}
Let $f \colon M_1 \rightarrow M_2$ be a continuous map with $\deg(f) \neq 0$ between closed hyperbolic $n$-manifolds, with $n \geq 3$. Denote by $\Gamma_1$ and $\Gamma_2$ the fundamental groups of $M_1$ and $M_2$, respectively. Consider a maximal cocycle $\sigma \colon \Gamma_2 \times X \rightarrow \po^\circ(n, 1)$. Then, we have
$$
\Vol(f^*\sigma) = \deg(f) \cdot \Vol(M_2) \ ,
$$
where $f^*\sigma$ denotes the pullback cocycle of $\sigma$ along $f$.
\end{prop}

This result can be interpreted as a reformulation of the result~\cite[Proposition~3.1]{derbez:liu:sun:wang} in the case of measurable cocycles. However, it turns out to encode useful information. First, it allows us to reprove the mapping theorem in dimension $n \geq 3$, as shown in Corollary \ref{cor:mapping:degree}.

More surprisingly, Thurston's strict version of the mapping degree theorem~\cite[Theorem~6.4]{Thurston} will allow us to describe maps homotopic to local isometries as the ones preserving maximal cocycles. More precisely, we are going to show that the pullback of a maximal cocycle is still maximal if and only if the map along which we are performing the pullback is homotopic to a local isometry. This is the content of the following:

\begin{prop}\label{prop:local:iso}
Let $f \colon M_1 \rightarrow M_2$ be a continuous map with $\deg(f) \neq 0$ between closed hyperbolic manifolds of the same dimension $n \geq 3$. Denote by $\Gamma_1$ and $\Gamma_2$ the fundamental groups of $M_1$ and $M_2$, respectively. Let $\sigma \colon \Gamma_2 \times X \rightarrow \po^\circ(n, 1)$ be a maximal cocycle. Then, $f$ is homotopic to a local isometry if and only if $f^*\sigma$ is a maximal cocycle.
\end{prop}

\subsection{Euler number for measurable cocycles}
In the same spirit of the volume invariant, in this paper  we also provide the definition of \emph{Euler number} of measurable cocycles defined in terms of uniform lattices in the spirit of the Euler number of representations. As mentioned in Remark~\ref{oss:nuova:uniforme} (compare with Section~\ref{subsec:nuova:uniforme}), this numerical invariant coincides with the \emph{generalized} version of Bader, Furman and Sauer's Euler number~\cite{sauer:articolo} up to a multiplicative constant when we deal with $(\Gamma, \Gamma)$-couplings. One could also extend our Euler number to the case of non-compact surfaces, that is to non-uniform lattices. However, this situation seems to contain some subtleties and we prefer to postpone it to a forthcoming project.

Let $\pi_0:\Gamma_g \rightarrow \textup{PSL}(2,\R)$ be a hyperbolization and assume that $\Gamma_g$ acts on $\bbS^1$ via $\pi_0$. Consider a measurable cocycle $\sigma:\Gamma_g \times X \rightarrow \homs$ with essentially unique boundary map $\phi:\bbS^1 \times X \rightarrow \bbS^1$ (one can assume that $\sigma$ is \emph{non-elementary} by \cite[Proposition 3.3]{MonShal0}). Here $(X,\mu_X)$ is again a standard Borel probability $\Gamma_g$-space. As in the case of the volume, the existence of the map $\phi$ allows us to pullback the Euler cocycle and evaluate it on the fundamental class. This construction provides our \emph{Euler number} $\eu(\sigma)$ associated to a measurable cocycle~$\sigma$. If we restrict ourselves to cocycles associated to representations, we are able to prove in Proposition~\ref{eu:representations} that our invariant coincides with the classic Euler number of representations. Moreover, our Euler number is a well-defined numerical invariant since it is constant along the $\homs$-cohomology class of measurable cocycles as proved in Proposition~\ref{homs:cohomology}. Our investigation on the relation between our Euler number of measurable cocycles and the Euler number of representations leads to a Milnor-Wood type inequality. Our approach to the study of our Euler number of measurable cocycles is substantially different from the one carried on by Bader, Furman and Sauer~\cite{sauer:articolo}, but we are still able to provide a new proof of such an inequality (compare our Proposition~\ref{Milnor:wood:style} with the \emph{generalized} version of Bader, Furman and Sauer's inequality~\cite[Corollary~4.9]{sauer:articolo}). Finally, we are able to characterize  maximal cocycles by proving an extension of Matsumoto's theorem in this setting:
\begin{teor}\label{matsumoto}
Let $\Sigma_g$ be a closed surface of genus $g \geq 2$ and let $\Gamma_g:=\pi_1(\Sigma_g)$. Let $\pi_0:\Gamma_g \rightarrow \textup{PSL}(2,\R)$ be a hyperbolization and assume that $\Gamma_g$ acts on $\bbS^1$ via $\pi_0$. Let $(X,\mu_X)$ be a standard Borel probability $\Gamma_g$-space. Then for every cocycle $\sigma:\Gamma_g \times X \rightarrow \homs$ with essentially unique boundary map~$\phi:\bbS^1 \times X \rightarrow \bbS^1$, we have

\[
|\eu(\sigma)|\leq |\chi(\Sigma_g)|
\]
and equality holds if and only if $\sigma$ is cohomologous to a cocycle induced by a hyperbolization.
\end{teor}

We briefly mention that our proof follows the lines of a result by Bader, Furman and Sauer~\cite[Theorem C]{sauer:articolo} (compare also with a result by the second author~\cite[Theorem 1]{savini3:articolo}). In a similar way of what happens for the volume invariant, we are able to express the Euler number of a measurable cocycle as a suitable multiplicative constant between bounded cohomology classes.

\begin{prop}\label{prop:proportionality:eulero}
Let $\Sigma_g$ be a closed surface of genus $g \geq 2$ and let $\Gamma_g=\pi_1(\Sigma_g)$. Let $(X,\mu_X)$ be a standard Borel probability $\Gamma_g$-space. Fix any hyperbolization $\pi_0:\Gamma_g \rightarrow \textup{PSL}(2,\R)$ and assume that $\Gamma_g$ acts on $\bbS^1$ via $\pi_0$. Consider a cocycle $\sigma:\Gamma_g \times X \rightarrow \homs$ with essentially unique boundary map $\phi:\bbS^1 \times X \rightarrow \bbS^1$. Then,  
\begin{equation}
\hspace{-10pt} \int_{\pi_0(\Gamma_g)\backslash G} \int_X \orien(\phi(\overline{g} \cdot \xi_0,x),\phi(\overline{g} \cdot \xi_1,x),\phi(\overline{g} \cdot \xi_2,x))d\mu_X(x) d\mu_0(\bar g)=\frac{\eu(\sigma)}{\chi(\Sigma_g)}\orien(\xi_0,\xi_1,\xi_2) \ , \nonumber
\end{equation}
for almost every $(\xi_0,\xi_1,\xi_2) \in (\bbS^1)^3$. Here, we set $G = \textup{PSL}(2,\R)$ for ease of notation and $\mu_0$ is the normalized probability measure induces by the Haar measure of $\textup{PSL}(2,\R)$ on the quotient $\pi_0(\Gamma) \backslash \textup{PSL}(2,\mathbb{R})$. The symbol $\orien$ denotes the orientation cocycle (see Equation~(\ref{eq:orien:cocycle})).
\end{prop}

Using the expression above, we show that when a cocycle $\sigma$ is positively maximal, that is $\eu(\sigma)=|\chi(\Sigma_g)|$, the associated map defined by $\phi_x(\xi):=\phi(\xi,x)$ is order preserving for almost every $x \in X$, where $\xi \in \, \mathbb{S}^1$. Hence, it essentially coincides with an element $f \in \homs$, for almost every $x \in \, X$. This construction provides the desired measurable function, which conjugates the cocycle $\sigma$ with the one induced by the hyperbolization $\pi_0$. The same technique works also for negatively maximal cocycles. 

The main techniques that we develop in this paper could be extended to some other numerical invariants of representations. More precisely, in \cite{moraschini:savini:2} we study in a systematic way the theoretical setting in which one can extend numerical invariants of representations to measurable cocycles and we apply our results to the study of measurable cocycles associated to complex hyperbolic lattices. 


\subsection*{Acknowledgements}

We would like to thank Roman Sauer for the enlightening discussions about the topic of this paper and his interest on it. We also wish to thank Michelle Bucher, Marc Burger, Alessandra Iozzi and Stefano Francaviglia for useful comments and corrections on the paper. We are grateful to Maria Beatrice Pozzetti for the suggestion of Remark 4.8 and for having pointed out a mistake in Corollary 5.13. We thank Clara L{\"o}h for the helpful discussions about Borel measure spaces and her careful reading of the first draft of this paper.

Finally, we are grateful to the anonymous referees for their useful comments and suggestions that we used to improve our paper. 

\subsection*{Plan of the paper}

In Section~\ref{sec:zimmer:cocycle} we recall the general theory of Zimmer's cocycles. We describe the notion of couplings and the concept of boundary maps associated to measurable cocycles. In Section~\ref{sec:bounded:cohomology} we collect all the properties about bounded cohomology that we need in the sequel. In Section~\ref{subsec:continuous:bounded:cohomology} we define the continuous (bounded) cohomology of a group. Then, in Section~\ref{sec:bm} we remind Burger-Monod's functorial approach to bounded cohomology. Section~\ref{sec:cbc} is devoted to the description of the continuous bounded cohomology of $\Isom(\mathbb{H}^n)$. Here it appears the crucial definition of volume cocycle. In Section~\ref{sec:euler:class} it follows a brief exposition of the notion of Euler class and of the definition of the orientation cocycle. Since we will need later to work with relative bounded cohomology, Section~\ref{sec:rel:coom} is dedicated to the study of its properties. Here it is described the fundamental pairing involving Kronecker product. Finally, in Section~\ref{sec:trans} the transfer maps are defined together with their properties.

Section~\ref{sec:lemma:tecnico} is mainly devoted to the proof of a fundamental technical lemma. Here, we describe how to perform the pullback of a cocycle in presence of measurable boundary map associated to a Zimmer's cocycle (see Definition~\ref{def:pullback:map}).

In Section~\ref{sec:volume:totale} we study the volume of measurable cocycles. Its definition appears in Section~\ref{sec:def:vol} (see Definition~\ref{def:per:intro:volume}). Section~\ref{subsec:nuova:uniforme} is mainly devoted to the comparison between our invariant in the uniform case and the generalized version of the Euler number defined by Bader, Furman and Sauer~\cite{sauer:articolo}. We investigate the properties of our volume in Section~\ref{sec:representations:coupling:volume} and we prove that it extends the volume of representations (see Proposition~\ref{prop:rep:vol}).  The volume is invariant on the $\Isom^+(\mathbb{H}^n)$-cohomology classes of measurable cocycle, as shown in Proposition~\ref{gplus:cohomology}. Section~\ref{sec:rigidity:volume} is dedicated to the proof of our rigidity result, Theorem~\ref{mostow}. Here we also define reducible cocycles and we show that they have vanishing volume (see Example~\ref{es:reducible:cocycles}). Two crucial results are the Milnor-Wood type inequality contained in Proposition~\ref{prop:vol:estimate:wood} and the interpretation of the volume as a multiplicative constant given in Proposition \ref{prop:formula:volume}. Finally, we conclude with Corollary~\ref{cor:value:coupling} where we show that ergodic integrable self-couplings of uniform lattices are maximal, whence they are conjugated to the cocycle associated to the standard lattice embedding.

In Section~\ref{sec:max:cocycle:mapping:degree}, we relate the volume of measurable cocycles with the degree of continuous maps between closed hyperbolic manifolds. After having introduced the notion of pullback cocycle with respect to a continuous map, we prove Proposition~\ref{prop:moltiplicativita} which relates the volume of the pullback of maximal cocycle with the degree of the continuous map. This approach furnishes a different proof of the classic mapping degree theorem in Corollary~\ref{cor:mapping:degree}. Finally, in Proposition~\ref{prop:local:iso} we characterize maps homotopic to local isometries via maximal cocycles.

Then we move to Section~\ref{sec:Euler:totale} where we study the Euler number of measurable cocycles, defined in Section~\ref{sec:def:eulero:main:prop} (see Definition~\ref{euler}). In Remark~\ref{oss:nuova:uniforme} we discuss the link between our Euler number and the generalized version of the one introduced by Bader, Furman and Sauer~\cite{sauer:articolo}. Section~\ref{sec:eul:coc:vs:repre} is devoted to the study of the relation between our Euler number and the one associated to representations. More precisely in Proposition~\ref{eu:representations} we show that our invariant extends naturally the Euler number of representations. Following the line of the volume invariant, in Proposition~\ref{homs:cohomology} it is proved that the Euler number is constant along $\homs$-cohomology classes of measurable cocycles. Finally, we conclude with Section~\ref{sec:euler:rigidity} where we prove our main Theorem~\ref{matsumoto} using both the Milnor-Wood type inequality stated in Proposition~\ref{Milnor:wood:style} and the intepretation of the Euler number as multiplicative constant given in Proposition \ref{prop:proportionality:eulero}. 

\section{Zimmer's cocycles theory}\label{sec:zimmer:cocycle}

In this section we are going to introduce some basic definitions and results about Zimmer's cocycle theory. For a detailed discussion of this subject we refer the reader to both Furstenberg~\cite{furst:articolo} and to Zimmer~\cite{zimmer:libro}. 

\begin{deft}\label{cocycle}
Let $G$ and $H$ be two locally compact groups. Let $(X,\mu)$ be a standard Borel probability space on which $G$ acts in an essentially free way by preserving the measure $\mu$. In the sequel we will refer to it simply as \emph{standard Borel probability} $G$-space. Denote by $\textup{Meas}(X,H)$ the space of measurable functions from $X$ to $H$ endowed with the topology of convergence in measure.

A measurable function $\sigma: G \times X \rightarrow H$ is a \textit{measurable cocycle} (or, simply, \emph{cocycle}) if the map $$\sigma:G \rightarrow \textup{Meas}(X,H) \ , \hspace{10pt} g \mapsto \sigma(g, \cdot)$$ is continuous and $\sigma$ satisfies the following formula
\begin{equation}\label{cocycleq}
\sigma(g_1g_2,x)=\sigma(g_1,g_2x)\sigma(g_2,x)
\end{equation}
for every $g_1,g_2$ and almost every $x \in X$.
\end{deft}

\begin{oss}
We warn the reader that in general definition of measurable cocycle does not require the additional hypothesis of essentially freeness of the action of $G$ on $(X, \mu)$. However, this assumption is not very restrictive. Indeed, we can turn every probability measure-preserving action into an essentially-free action just by taking the product with an essentially-free action and the diagonal action.

On the other hand, it seems that working with essentially-free actions is rather common in the literature about simplicial volume and its ergodic version, called integral foliated simplicial volume, see e.g.~\cite{Sauer, Sthesis, FFM,LP,FLPS,FLF,Camp-Corro,FLMQ}. For this reason, we think that it is better to keep the same setting here, when we deal with ergodic theory applied to bounded cohomology, which is the dual theory of simplicial volume (notice that we are \emph{not} claiming, that our theory is the dual theory of integral foliated simplicial volume).

Moreover, it is worth noticing that having an essentially-free action has a nice implication on the standard Borel probability space $(X, \mu)$. Indeed, if we have an infinite torsion-free group $G$ acting on $(X, \mu)$ in a measure preserving way, then the measure $\mu$ cannot have atoms. This implies that the standard Borel probability space $(X, \mu)$ is measurably isomorphic to the interval $[0,1]$ endowed with the Lebesgue measure~\cite[Theorem~A.20]{kerr-li} (but, then the action is highly non-standard).
\end{oss}

At a first sight Equation (\ref{cocycleq}) may appear quite mysterious, but it is just a suitable extension of the ordinary chain rule for derivates to this more general context. We mention here another equivalent approach for defining measurable cocycles. Since $\sigma \in \, \textup{Meas}(G,\textup{Meas}(X,H))$, Equation (\ref{cocycleq}) is the characterization of Borel $1$-cocycles with values in the $G$-module $\textup{Meas}(X,H)$ in the sense of  Eilenberg-MacLane (see \cite{feldman:moore, Zimmer:preprint} for a similar description). Following this latter approach to measurable cocycles, we introduce the notion of cohomologous cocycles. 

\begin{deft}\label{def:coomologhi:cicli}
Let $\sigma_1,\sigma_2: G \times X \rightarrow H$ be two measurable cocycles and let $f:X \rightarrow H$ be a measurable map. We denote by
$$
\sigma_1^f \colon G \times X \rightarrow H
$$
the cocycle satisfying
$$
\sigma_1^f (g,x)=f(g x)^{-1} \sigma_1(g,x) f(x)\ ,
$$
for all $g \in G$ and almost every $x \in X$. We say that $\sigma_1$ and $\sigma_2$ are \textit{cohomologous} (or, \emph{equivalent}) if $\sigma_2=\sigma_1^f$, for some measurable map $f \colon X \rightarrow H$. 
\end{deft} 

In order to make the reader more familiar with the notion of measurable cocycles, we are going now to discuss in details a couple of families of them. We mention that beyond our explicit examples, measurable cocycles are quite ubiquitous in mathematics. For instance, they naturally appear in differential geometry (as differentiation cocycles) and in measure theory (as Radon-Nykodim cocycles) We refer the reader to~\cite[Examples~4.2.3 and 4.2.4]{zimmer:libro} for an overview of such examples. However, along this paper we will be primarily interested in cocycles arising from either representation or self-couplings.

We begin by introducing cocycles associated to \emph{representations}.

\begin{deft}\label{cocyclerep}
Let $\rho:G \rightarrow H$ be a continuous representation. Fix any standard Borel probability $G$-space $(X,\mu)$. The cocycle  \textit{associated to the representation} $\rho$ 
$$\sigma_\rho \colon G \times X \rightarrow H$$
is given by
\[
\sigma_\rho: G \times X \rightarrow H, \hspace{10pt} \sigma_\rho(g,x):=\rho(g) \ ,
\]
for every $g \in G$ and almost every $x \in X$. 
\end{deft}

We warn the reader that the cocycle $\sigma_\rho$ just introduced depends actually both on $\rho$ and on $X$. However, since the condition on $X$ is not significant, we drop the $X$ from the notation. Notice that when $G$ is discrete any representation is automatically continuous and hence it admits an associated cocycle.

The previous definition provides a large family of measurable cocycles and it shows that the theory of representations sits inside the much wider framework of Zimmer's cocycles theory.

We now introduce another example of cocycles which arises from the study of \emph{self-couplings}. We recall here the definition of coupling (compare with~\cite[Definition 1.1]{sauer:articolo}):

\begin{deft}\label{deft:coupling}
Let $G$ and $H$ be two locally compact second countable groups with Haar measures $m_G$ and $m_H$, respectively. We will assume additionally that both $G$ and $H$ are unimodular. A \textit{coupling} for $G$ and $H$ is the datum of a Lebesgue measure space $(\Omega,m_\Omega)$ with a $G \times H$-measurable, measure-preserving action such that there exist two finite measure spaces $(X,\mu)$ and $(Y,\nu)$ with measurable isomorphisms
\[
\imath:(G,m_G) \times (Y,\nu) \rightarrow (\Omega,m_\Omega), \hspace{10pt} \jmath:(H,m_H) \times (X,\mu) \rightarrow (\Omega,m_\Omega)\ ,
\]
where $\imath$ is $G$-equivariant and $\jmath$ is $H$-equivariant, both with respect to the natural actions on the first factor. If there exists a coupling $\Omega$ for two groups $G$ and $H$, we say that $G$ and $H$ are \textit{measure equivalent}. 
\end{deft}

Given a coupling $(\Omega,m_\Omega)$ of two groups $G$ and $H$ as above, one can construct two different cocycles associated to it. By the commutativity of the actions of $G$ and $H$ on $\Omega$, we have a well-defined action of $G$ on the space $X$. Indeed, $X$ can be naturally identified with the space of $H$-orbits in $\Omega$. Similarly, we obtain an action of $H$ on the space $Y$. Both actions are measure preserving. For any $g \in G$ and almost every $h \in H$ and $x \in X$, there must exist $h_1 \in H$ which depends on both $g$ and $x$ such that 
\[
g\jmath(h,x)=\jmath(hh_1^{-1},g.x) \ ,
\]
where $g.x$ denotes the action of $g$ on $X$ in order to distinguish it from the one on $\Omega$. Since $h_1$ only depends on both $g$ and $x$, the previous formula leads to the definition of cocycles associated to a coupling.

\begin{deft}\label{couplings}
Let $(G,m_G)$ and $(H,m_H)$ be two locally compact, unimodular, second countable groups with their respective Haar measures. Let $(\Omega,m_\Omega)$ be a coupling for $G$ and $H$. The \textit{right measure equivalence cocycle} associated to the coupling is defined by
\[
\alpha_\Omega: G \times X \rightarrow H, \hspace{5pt} \alpha_\Omega(g,x):=h_1 \ .
\]
Note that $\alpha_\Omega(g,x)$ satisfies 
\[
g\jmath(h,x)=\jmath(h\alpha_\Omega(g,x)^{-1},g.x)
\]
for every $g \in G$ and almost every $x \in X$ and $h \in H$. 
\end{deft}

By interchanging the role of $G$ and $H$, one can define the \textit{left measure equivalence cocycle} 
\[
\beta_\Omega:H \times Y \rightarrow G \ ,
\]
as
\[
h\imath(g,y)=\imath(\beta_\Omega(h,y)g,h.y) \ ,
\]
for every $h \in H$ and almost every $g \in G$ and $y \in Y$. The left (respectively, right) measurable cocycle encodes the information about the action of $G \times H$ on $(\Omega,m_\Omega)$ when we identify it with $(G,m_G) \times (Y,\nu)$ (respectively, $(H,m_H) \times (X,\mu)$). 

After having introduced some examples of cocycles, we now move to the fundamental notion of \emph{boundary map} in the sense of Furstenberg~\cite{furst:articolo}. Assume that both $G$ and $H$ admit a Furstenberg-Poisson boundary (see~\cite{furstenberg:annals} for a precise definition). For the convenience of the reader we mention here an example of Furstenberg-Poisson boundary that we will need in the sequel. Given a Lie group $G$ of non-compact type, the Furstenberg-Poisson boundary can be naturally identified with the quotient $G/P$, where $P$ is any minimal parabolic subgroup of $G$. More generally, any lattice $\Gamma < G$ admits the previous quotient $G/P$ as a natural Furstenberg-Poisson boundary. Denote by $B(G)$ (respectively, by $B(H)$) the Furstenberg-Poisson boundary associated to $G$ (respectively, $H$). Endow both boundaries with their natural Borel sigma algebras induced by the Haar sigma algebras on their respective groups. 

\begin{deft}\label{boundary}
Let $\sigma:G \times X \rightarrow H$ be a measurable cocycle. We say that a measurable map
$\phi:B(G) \times X \rightarrow B(H)$
is $\sigma$-\emph{equivariant} if 
\[
\phi(g\xi,gx)=\sigma(g,x)\phi(\xi,x) \ ,
\]
for every $g \in G$ and almost every $\xi \in B(G)$ and $x \in X$. 

A \textit{boundary map} for $\sigma$ is a $\sigma$-equivariant measurable map $\phi$.
\end{deft}

\begin{oss}\label{oss:boundary:genmap}
In the definition above we assumed that both $G$ and $H$ admit a Furstenberg-Poisson boundary. However in the sequel we will need a slight more general notion of boundary map.
Indeed we will work with the group $\homs$ which is not a Lie group when endowed with the discrete topology. To overcome this difficulty it will be sufficient to consider a suitable compact space on which the group acts measurably. More precisely, consider $G$ and $H$ two locally compact groups. Assume that $G$ admits a Furstenberg-Poisson boundary $B(G)$ and that $H$ acts measurably on a compact completely metrizable space $Y$. A \emph{generalized boundary map} is a measurable map $$\phi:B(G) \times X \rightarrow Y$$ which is a $\sigma$-equivariant, that is 
\[
\phi(g\xi,gx)=\sigma(g,x)\phi(\xi,x) \ ,
\]
for every $g \in G$ and almost every $\xi \in B(G)$ and $x \in X$. 
\end{oss}

The existence and the uniqueness of a boundary map are both strictly related with the properties of the cocycle $\sigma$. For instance, any \emph{proximal} cocycle admits an essentially unique boundary map. We refer the reader to~\cite{furst:articolo} for a discussion on that property. 

Since boundary maps are defined in terms of measurable cocycles, it is natural to describe how they vary along a cohomology class.

\begin{deft}\label{boundary:map:cohomology}
Let $\sigma: G \times X \rightarrow H$ be a measurable cocycle and assume it admits a boundary map $\phi:B(G) \times X \rightarrow B(H)$. Let $f: X \rightarrow H$ be a measurable map. 
The \emph{boundary map associated to the twisted cocycle} $\sigma^f$
$$
\phi^f: B(G) \times X \rightarrow B(H)
$$
is the measurable map defined by
$$
 \phi^f(\xi,x):=f(x)^{-1}\phi(\xi,x)
$$
for almost every $\xi \in B(G)$ and $x \in X$. 
\end{deft}

An easy computation shows that the map $\phi^f$ is indeed $\sigma^f$-equivariant with respect to the cocycle $\sigma^f$. The same definition will hold also in the case of generalized boundary maps. 

\section{Bounded cohomology and first results}\label{sec:bounded:cohomology}

\subsection{Continuous (bounded) cohomology}\label{subsec:continuous:bounded:cohomology}
Let $G$ be a locally compact group. In this section we recall the definition of the \emph{continuous bounded cohomology} of $G$ and in the next section will show how to compute it via Burger-Monod's functorial approach~\cite{burger2:articolo} (see also \cite{monod:libro}). It is worth mentioning that some equivalent results in the special case of discrete groups date back to Ivanov~\cite{Ivanov} (see also~\cite{miolibro}). We should also warn the reader that the next two sections are rather technical. Moreover, most of the technical definitions are only needed for stating the main classic results of Section~\ref{sec:bm} and we will not use them anymore in the sequel. Nevertheless, we decided to write this self-contained sections for sake of completeness in order to offer to the reader a complete toolbox for a better comprehension the paper. 

\begin{deft}
Let $G$ be a locally compact group. Let $E$ be a Banach space and let $\pi:G \rightarrow \Isom(E)$ be a representation. The pair $(E,\pi)$ is called \emph{Banach $G$-module}, and it can be alternatively thought of as an action of $G$ on $E$ via linear isometries
\begin{align*}
\theta_\pi \colon G \times &E \rightarrow E \\
\theta_\pi(g,v) &\coloneqq \pi(g)v \ .
\end{align*}
We say that the Banach $G$-module is \emph{continuous} if for all $v \in \, E$ the orbit map 
\begin{align*}
G &\rightarrow E \\
g &\mapsto \pi(g) v \ 
\end{align*}
is continuous at $e \in \, G$, where $e$ denotes the neutral element of $G$.

Given two (continuous) Banach $G$-modules $(E, \pi_E)$ and $(F, \pi_F)$, a $G$-\emph{morphism} $\varphi \colon E \rightarrow F$ is a linear 
$G$-equivariant map between Banach $G$-spaces, i.e. $$\varphi( \pi_E(g) v) = \pi_F(g) \varphi(v)$$ for all $g \in \, G$ and $v \in \, E$.

For every (continuous) Banach $G$-module $(E, \pi)$, we define the submodule of $G$-\emph{invariants} $E^G$ as
$$
E^G \coloneqq \{v \in \, E \, | \, \pi(g)v = v, \, \forall g \in \, G \} \ .
$$
\end{deft}

\begin{oss}\label{oss:caso:discreto:continuo:non:continuo:coincidono}
Notice that if we endow $G$ with the discrete topology, then there is no difference between Banach $G$-modules and \emph{continuous} Banach $G$-modules. In this paper the main example of this situation is given by lattices in $\textup{Isom}(\mathbb{H}^n)$, $n \geq 2$.
\end{oss}

Since not continuous actions may reveal quite annoying, sometimes it can be convenient to study the continuous submodule of a given Banach $G$-module:

\begin{deft}
Let $G$ be a locally compact group and let $(E, \pi)$ be a Banach $G$-module. We define the \emph{maximal continuous submodule} of $E$ as follows
$$
\mathcal{C}E:=\{v \in E \, | \, G \rightarrow E, \ , g \mapsto \pi(g)v \, \mbox{ is continuous at $e \in \, G$}\} \ ,
$$
where $e \in G$ is the neutral element.
\end{deft}

\begin{oss}
Note that any $G$-morphism between Banach $G$-modules restricts to the maximal continuous submodules~\cite[Lemma~1.2.6]{monod:libro}.
\end{oss}

For the convenience of the reader we list now some standard examples that will be needed in the sequel:
\begin{es}\label{es:Banach:G:modules}
Let $G$ be a locally compact group:
\begin{enumerate}
\item We say that $(\mathbb{R},\pi)$ has the \emph{trivial} Banach $G$-module structure if $\pi$ is \emph{trivial}, i.e.  if $\pi(g) = \textup{id}_\mathbb{R}$ for all $g \in \, G$.

\item Let $(E, \pi)$ be a Banach $G$-module. Let us consider the Banach space of \emph{continuous $E$-valued functions} on $G^{\bullet+1}$, that is
$$
\textup{C}^\bullet_c(G; E) \coloneqq \{f \colon G^{\bullet + 1} \rightarrow E \ | \ \textup{$f$ is continuous} \}
$$
endowed with the supremum norm
$$
\lVert f \rVert_\infty \coloneqq \sup_{g_0, \cdots, g_\bullet} \lVert f(g_0, \cdots, g_\bullet) \rVert_E \ ,
$$
where $\lVert \ \cdot \ \rVert_E$ denotes the norm on $E$. The pair $(\textup{C}^\bullet_c(G; E), \tau)$ is a Banach $G$-module with the action defined by
$$
(\tau(g) \cdot f)(g_0, \cdots, g_\bullet) \coloneqq \pi(g) f(g^{-1}g_0, \cdots, g^{-1}g_\bullet) \ ,
$$
for all $g, g_0, \cdots, g_\bullet \in \, G$ and $f \in \textup{C}^\bullet_c(G; E)$.

\item The previous Banach $G$-module contains an important Banach $G$-submodule, the space of \emph{continuous bounded $E$-valued functions} on $G^{\bullet+1}$, namely
$$
\textup{C}_{cb}^\bullet(G; E) \coloneqq \{f \in \, \textup{C}^\bullet_c(G; E) \, | \, \lVert f \rVert_\infty < +\infty \} \ ,
$$
endowed with the restriction of the previous $G$-action $\tau$. 

\item Given a Banach $G$-module $(E, \pi)$, suppose that $E$ is the dual of some Banach space. In this case $E$ admits a natural weak-$^*$ topology and an associated Borel weak-$^*$ structure. Consider a standard Borel probability space $(X,\mu)$ on which $G$ acts measurably by preserving the measure class of $\mu$. Let us consider the Banach space of \emph{bounded \mbox{weak-$^*$} measurable $E$-valued functions} on $X^{\bullet+1}$:
\begin{align*}
\mathcal{B}^\infty(X^{\bullet+1};E):=\{ f:X^{\bullet+1} \rightarrow E| &\textup{$f$ is weak-$^*$ measurable}, \\
\sup_{x_0,\cdots,x_\bullet \in X} \lVert &f(x_0,\cdots,x_\bullet) \rVert_E < \infty \} \ ,
\end{align*}
where $\lVert \ \cdot \ \rVert_E$ denotes the norm on $E$. 

We endow it with the structure of Banach $G$-module $(\mathcal{B}^\infty(X^{\bullet+1};E), \tau)$ given by
\[
(\tau(g) \cdot f)(x_0,\cdots,x_\bullet):=\pi(g)f(g^{-1}x_0,\cdots,g^{-1}x_\bullet) \ ,
\]
for every $g \in G$, for every $x_0, \cdots, x_\bullet \in X$ and for every $f \in \mathcal{B}^\infty(X^{\bullet+1}; E)$. 

\item In the same setting of the previous case, let us consider the Banach space of \emph{equivalence classes of bounded weak-$^*$ measurable $E$-valued functions} on $X^{\bullet+1}$:
$$
\textup{L}^\infty_{\textup{w}^*}(X^{\bullet+1};E) \coloneqq \{ [f]_\sim \ | \ f \in \, \mathcal{B}^\infty(X^{\bullet+1};E)\} \ ,
$$
where $f \sim g$ if and only if they agree $\mu$-almost everywhere and $[f]_\sim$ denotes the equivalence class of $f$ with respecto to $\sim$. 

We endow it with the structure of Banach $G$-module $(\textup{L}^\infty_{\textup{w}^*}(X^{\bullet+1};E), \tau)$ given by
$$
(\tau(g) \cdot f)(\xi_0,\cdots,\xi_\bullet):=\pi(g)f(g^{-1}\xi_0,\cdots,g^{-1}\xi_\bullet) \ ,
$$
for every $g \in \, G$, for every $x_0, \cdots, x_\bullet \in X$ and for every $f \in \textup{L}^\infty_{\textup{w}^*}(X^{\bullet+1}; E)$. Notice that we defined the action $\tau$ working directly with representatives of the equivalence classes, dropping the parenthesis to avoid a heavy notation. We will do the same thing every time that we will need to work with such spaces. 
\end{enumerate}
\end{es}

\begin{oss}\label{oss:cocatene:cont:uguale:non:cont:discreto}
As in Remark~\ref{oss:caso:discreto:continuo:non:continuo:coincidono}, we stress again the fact that if $G$ is a discrete group, then every function $f:G^{\bullet+1} \rightarrow E$ is automatically continuous and hence we can drop the subscript \emph{c} in $\textup{C}_c^\bullet(G;E)$. Indeed when $G$ is discrete we are going to write $\textup{C}^\bullet(G;E)$. We will do the same thing writing $\textup{C}^\bullet_b(G;E)$ instead of $\textup{C}^\bullet_{cb}(G;E)$.  
\end{oss}

We are now ready to introduce the definition of \emph{continuous (bounded) cohomology}. Let $G$ be a locally compact group and let $(E, \pi)$ be a Banach $G$-module. Then, the Banach $G$-modules of continuous (bounded) functions on $G$ fit inside the following cochain complex
$$
(\textup{C}_{c (b)}^\bullet(G; E),  \delta^\bullet) \ ,
$$
where $\delta^\bullet$ is the standard homogeneous coboundary operator defined by 
$$
\delta^\bullet:\textup{C}^\bullet_{c(b)}(G;E) \rightarrow \textup{C}^{\bullet+1}_{c(b)}(G;E) \ ,
$$
\begin{equation}\label{eq:coboundary:standard}
\delta^\bullet f(g_0, \cdots, g_{\bullet +1}) = \sum_{i = 0}^{\bullet+1} (-1)^i f(g_{0}, \cdots, g_{i -1}, g_{i+1}, \cdots, g_{\bullet+1}) \ ,
\end{equation}
for every $g_0, \cdots, g_{\bullet+1} \in \, G$. By linearity, it is easy to check that $\delta^\bullet$ sends bounded cochains to bounded cochains and the same holds for $G$-invariant cochains. This allows us to restrict the coboundary operator to the subcomplex of $G$-invariant cochains $(\textup{C}_{c (b)}^\bullet(G; E)^G,  \delta^\bullet)$.

\begin{deft}\label{def:cont:bdd:coho:sezione:isom}
The \emph{continuous (bounded) cohomology} of $G$ with coefficients in $E$, denoted by $\textup{H}^\bullet_{c (b)}(G; E)$, is the cohomology of the complex $(\textup{C}_{c (b)}^\bullet(G; E)^G;  \delta^\bullet)$.
\end{deft}

\begin{oss}\label{oss:cont:(bounded):coom:per:discreti:uguale}
Notice that if $G$ is discrete, then there is no difference between continuous and non-continuous (bounded) cohomology. For this reason, we will drop the $c$-subscript from the notation in the case of discrete groups (compare with Remarks~\ref{oss:caso:discreto:continuo:non:continuo:coincidono} and~\ref{oss:cocatene:cont:uguale:non:cont:discreto}).
\end{oss}

A key feature of continuous bounded cohomology is that it is endowed with a natural seminorm. Indeed, the supremum norm on continuous bounded cochains induces a seminorm in continuous bounded cohomology as follows
$$
\lVert \psi \rVert_\infty \coloneqq \inf \{ \lVert f \rVert_\infty \, | \, f \in \, \textup{C}_{cb}^\bullet(G; E)^G \mbox{ and } [f] = \psi \} \ ,
$$
where  $\psi \in \, \textup{H}_{cb}^\bullet(G; E)$.

The gap between the continuous and the continuous bounded cohomology groups may be studied via the comparison map. More precisely, the inclusion of cochain complexes
$$
\iota \colon\textup{C}_{c b}^\bullet(G; E)^G \rightarrow \textup{C}_{c}^\bullet(G; E)^G \ ,
$$
induces a natural map
$$
\comp^\bullet \colon \textup{H}_{c  b}^\bullet(G; E) \rightarrow \textup{H}_{c}^\bullet(G; E) \ ,
$$
which we call \emph{comparison map}. For the convenience of the reader, we will sometimes denote it by $\comp^\bullet_G$.

\subsection{Burger-Monod's theory of continuous bounded cohomology}\label{sec:bm}

Since computing continuous (bounded) cohomology by means of its definition is usually very hard, we explain now how to compute it in terms of \emph{strong resolutions by relatively injective modules}. We begin with the following:

\begin{deft}
Let $G$ be a locally compact group and let $E, F$ be a Banach $G$-modules. A $G$-morphism $i \colon E \rightarrow F$ is said to be \emph{admissible} if there is a morphism of Banach spaces $\sigma \colon F \rightarrow E$ with bounded operator norm $\lVert \sigma \rVert \leq 1$ and such that $i \circ \sigma \circ i = i$.

We say that a Banach $G$-module $U$ is \emph{relatively injective} if for every admissible $G$-morphism $i \colon E \rightarrow F$ of continuous Banach $G$-modules and every $G$-morphism $\psi \colon E \rightarrow U$, there exists a $G$-morphism $\varphi \colon F \rightarrow U$ satisfying $\varphi \circ i = \psi$ and $\lVert \varphi \rVert \leq \lVert \psi \rVert$.
\end{deft}
\begin{oss}
Notice that a Banach $G$-module is relatively injective if and only if $\mathcal{C}E$ is so, by~\cite[~4.1.5]{monod:libro}.
\end{oss}
\begin{es}\label{es:relatively:injective}
Let $G$ be a locally compact group. The following Banach $G$-modules introduced in Example~\ref{es:Banach:G:modules} are examples of relatively injective Banach $G$-modules:
\begin{enumerate}
\item For every Banach $G$-module $E$, the spaces of continuous bounded functions $C^\bullet_{cb}(G; E)$ are relatively injective~\cite[Proposition~4.4.1]{monod:libro}.

\item Assume that $G$ is a semisimple Lie group of non-compact type (or one of its lattices). Let $B(G)$ denote its Furstenberg-Poisson boundary. Then, for every Banach $G$-module $E$ which is dual of some Banach space, we have that  $\textup{L}^\infty_{\textup{w}^*}(B(G)^{\bullet+1};E)$ is relatively injective~\cite[Theorem~5.7.1]{monod:libro}. Indeed $B(G)$ is an \emph{amenable $G$-space} in the sense of Zimmer \cite[Definition 5.3.1]{monod:libro} since it can be realized as the quotient $G/P$, where $P$ is a \emph{minimal parabolic subgroup} (which is amenable being a compact extension of a solvable group). Notice that the amenability is preserved by restricting the action to any lattice of $G$.  
\end{enumerate} 
\end{es}

We now introduce the definition of \emph{strong resolution}.

\begin{deft}
Let $G$ be a locally compact group. A \emph{Banach $G$-complex} $(E^\bullet, \delta^\bullet)$ is a cochain complex, i.e. $\delta^n \circ \delta^{n-1}$ for all $n \in \, \mathbb{N}$, whose modules are Banach $G$-modules and whose differentials are bounded $G$-morphisms in every degree. Moreover, we assume that $E^n = 0$ for every $n \leq -1$. 

Given a Banach $G$-module $E$, an \emph{augmented Banach $G$-complex} $(E, E^\bullet, \delta^\bullet)$ with augmentation map $\varepsilon \colon E \rightarrow E^0$ is the following $G$-complex
$$
0 \rightarrow E \xrightarrow{\varepsilon} E^0 \xrightarrow{\delta^0} E^1 \xrightarrow{\delta^1} E^2 \rightarrow \cdots \rightarrow E^n \xrightarrow{\delta^n} \cdots \ ,
$$
where we ask that $\varepsilon$ is an isometric embedding.

A \emph{resolution} of a Banach $G$-module $E$ is an exact augmented Banach $G$-complex $(E, E^\bullet, \delta^\bullet)$. 

A \emph{contracting homotopy} for a resolution $(E^\bullet, \delta^\bullet)$ of $E$ is a collection of maps $h^n \colon E^n \rightarrow E^{n-1}$ for every $n \in \, \mathbb{N}$ such that $\lVert h^n \rVert \leq 1$ and $h^{n+1} \circ \delta^{n} + \delta^{n-1} \circ h^{n} = \textup{id}_{E^n}$ for $n \geq 0$ and $h^0 \circ \varepsilon = \textup{id}_{E^0}$.

A resolution $(E^\bullet, \delta^\bullet)$ of $E$ is \emph{strong} if the subcomplex $(\mathcal{C}E^\bullet, \delta^\bullet)$ admits a contracting homotopy.
\end{deft} 

\begin{es}
Let $G$ be a locally compact group. The followings are examples of strong resolutions by relatively injective modules, i.e.  strong resolutions $(E^\bullet, \delta^\bullet)$ such that each $E^\bullet$ is a relatively injective $G$-module.
\begin{itemize}
\item Let $E$ be a Banach $G$-module. The resolution $(C_{cb}^\bullet(G; E), \delta^\bullet)$ with augmention map $\varepsilon: E \rightarrow C^0(G; E)$ given by the inclusion of constant functions is a strong resolution of $E$ via relatively injective modules~\cite[Theorem~7.2.3]{monod:libro} (compare with Example~\ref{es:relatively:injective}). Here $\delta^\bullet$ denotes the coboundary operator introduced in Equation~($\ref{eq:coboundary:standard}$)

\item Assume that $G$ is a semisimple Lie group of non-compact type (or one of its lattices). Let $B(G)$ denote its Furstenberg-Poisson boundary and let $E$ be a Banach $G$-module dual to some Banach space. Then,
$$
0 \rightarrow E \xrightarrow{\varepsilon} \textup{L}^\infty_{\textup{w}^*}(B(G);E) \xrightarrow{\delta^0} \textup{L}^\infty_{\textup{w}^*}(B(G)^{2};E) \xrightarrow{\delta^1} \cdots 
$$
is a strong resolution of $E$ by relatively injective $G$-modules~\cite[Theorem~7.5.3]{monod:libro}. Here $\varepsilon \colon E \rightarrow \textup{L}^\infty_{\textup{w}^*}(B(G);E)$ denotes the inclusion of constant functions and
$$
\delta^\bullet \colon \textup{L}^\infty_{\textup{w}^*}(B(G)^{\bullet+1};E) \rightarrow \textup{L}^\infty_{\textup{w}^*}(B(G)^{\bullet+2};E) \ ,
$$
\begin{equation}\label{eq:coboundary:essentially:bounded}
\delta^\bullet(f)(\xi_0,\cdots,\xi_{\bullet+1})=\sum_{i=0}^{\bullet+1} (-1)^i f(\xi_0,\cdots,\xi_{i-1},\xi_{i+1},\cdots,\xi_{\bullet+1}) \ ,
\end{equation}
for every $f \in \, \textup{L}^\infty_{\textup{w}^*}(B(G)^{\bullet+1};E)$ and $\xi_0,\cdots,\xi_{\bullet+1} \in \, B(G)$.
\end{itemize}
\end{es}
The importance of strong resolutions by relatively injective modules is described by the following results:
\begin{teor}[{\cite[Theorem~7.2.1]{monod:libro}}]\label{teor:monod:resolution:solo:iso}
Let $G$ be a locally compact group and let $E$ be a Banach $G$-module. Then, for every strong resolution $(E^\bullet, \delta^\bullet)$ of $E$ by relatively injective modules, the cohomology of the complex of $G$-invariants $((E^\bullet)^G, \delta^\bullet)$ is isomorphic as topological vector space to the continuous bounded cohomology $\textup{H}_{cb}^n(G; E)$ for every $n \geq 0$.
\end{teor}
\begin{teor}[{\cite[Theorem~7.5.3]{monod:libro}}]\label{teor:monod:resolution:L:inf}
 Let $G$ be a semisimple Lie group of non-compact type and let $E$ be a Banach $G$-module dual to some Banach space. Let $B(G)$ denote the Furstenberg-Poisson boundary of $G$. Then, the cohomology of the complex of $G$-invariants $$(\textup{L}^\infty_{\textup{w}^*}(B(G)^{\bullet+1};E)^G, \delta^\bullet)$$ is \emph{isometrically isomorphic} to the continuous bounded cohomology $\textup{H}_{cb}^n(G; E)$ for every $n \geq 0$. The same holds also for any lattice $\Gamma \leq G$. 
\end{teor}
\begin{oss}\label{oss:alternating}
Note that Theorem~\ref{teor:monod:resolution:L:inf} is still true when we restrict to the subcomplex and to the subresolution of alternating functions~\cite[Theorem~7.5.3]{monod:libro}. We recall that $f \colon B(G)^{\bullet} \rightarrow E$ is said to be \emph{alternating} if for every permutation $\sigma$ of $n$ elements, we have
$$
\textup{sign}(\sigma) f(\xi_{\sigma(1)}, \cdots, \xi_{\sigma(\bullet)}) = f(\xi_1, \cdots, \xi_\bullet) \ ,
$$
where $\textup{sign}$ denotes the parity of the permutation and $\xi_1, \cdots, \xi_\bullet \in \, B(G)$. We will need this stronger version of Theorem~\ref{teor:monod:resolution:L:inf} at the end of the proof of Proposition~\ref{Milnor:wood:style}
\end{oss}

We conclude this section by recalling that even strong resolutions may encode useful information. For instance, let $(X,\mu)$ be a standard Borel probability space on which $G$ acts by preserving the measure class of $\mu$. Given a Banach $G$-module $E$ that is dual of some Banach space, we can consider the resolution of $E$ given by $(\mathcal{B}^\infty(X^{\bullet+1};E), \delta^\bullet)$, where the augmentation map $\varepsilon \colon E \rightarrow \mathcal{B}^\infty(X;E)$ is given by the inclusion of constant functions and the coboundary operator is given by Equation~\eqref{eq:coboundary:essentially:bounded}. It is proved by Burger and Iozzi~\cite[Proposition 2.1]{burger:articolo} that the complex of bounded measurable functions $(\mathcal{B}^\infty(X^{\bullet+1};E),\delta^\bullet)$ is in fact a strong resolution for $E$. Since by Burger and Monod~\cite[Proposition 1.5.2]{burger2:articolo} the cohomology of \emph{any} strong resolution of the Banach $G$-module $E$ maps naturally to the continuous bounded cohomology of $G$, we get a canonical map
\begin{equation}\label{eq:map:end:sec:bm}
\mathfrak{c}^\bullet:\textup{H}^\bullet(B^\infty(X^{\bullet+1};E)^G) \rightarrow \hcb^\bullet(G;E) \ .
\end{equation}
This means that any bounded measurable cocycle $f \in B^\infty(X^{\bullet+1};E)^G$ naturally determines a cohomology class in $\hcb^\bullet(G;E)$.

\subsection{Continuous bounded cohomology of $G = \Isom(\mathbb{H}^n)$}\label{sec:cbc}

Let $n \geq 3$ and let $G = \Isom(\mathbb{H}^n)$ be the group of isometries of the hyperbolic $n$-space. We are going to denote the same group also by $\po(n,1)$ (and $\po^\circ(n,1)$ will be the connected component of the identity, that is the subgroup of orientation preserving isometries). Our main reference about the continuous and the continuous bounded cohomologies of $G$ is Bucher, Burger and Iozzi's paper~\cite{bucher2:articolo}. Let $G^+ \leq G$ denote the subgroup of orientation-preserving isometries. Since $G^+$ has index $2$ as a subgroup of $G$, there exists a well-defined homomorphism
\begin{equation}\label{eq:homo:1:meno1}
\varepsilon \colon G \rightarrow G \slash G^+ \cong \{-1, 1\} .
\end{equation}
Using the homomorphism $\varepsilon$, we can endow $\mathbb{R}$ with the following structure of Banach $G$-module 
\begin{align*}
G \times \mathbb{R} \rightarrow \mathbb{R} \ , \hspace{10pt} (g, a) \mapsto \varepsilon(g) \cdot a  \ .
\end{align*}
We will denote the Banach $G$-module $\mathbb{R}$ either by $\mathbb{R}_\varepsilon$ if it has the previous $G$-structure or by $\mathbb{R}$ if it is endowed with  the trivial one (see Example~\ref{es:Banach:G:modules}.1).

More generally, given a lattice $\Gamma < G$ and a representation $\rho \colon \Gamma \rightarrow G$, we may consider $\mathbb{R}$ as a Banach $\Gamma$-module endowed either with the trivial action or the one given by
\begin{align*}
\Gamma \times \mathbb{R} \rightarrow \mathbb{R} , \hspace{10pt} (\gamma, a) \mapsto \varepsilon(\rho(\gamma)) \cdot a  \ ,
\end{align*}
where $\varepsilon$ is the homomorphism defined in Equation~(\ref{eq:homo:1:meno1}). Then, since $\textup{C}_{c (b)}^\bullet(\rho)$ is norm non-increasing, any representation induces pullback maps in continuous (bounded) cohomology:
$$
\textup{H}_{c (b)}^\bullet(\rho) \colon \textup{H}_{c (b)}^\bullet(G; \mathbb{R}) \rightarrow \textup{H}_{(b)}^\bullet(\Gamma; \mathbb{R})
$$
and
$$
\textup{H}_{c (b)}^\bullet(\rho) \colon \textup{H}_{c(b)}^\bullet(G; \mathbb{R}_\varepsilon) \rightarrow \textup{H}_{(b)}^\bullet(\Gamma; \mathbb{R}_\rho) \ ,
$$
where we dropped the subscript \emph{c} in the right hand side being $\Gamma$ discrete.

We now recall how to compute the continuous bounded cohomologies $\textup{H}_{cb}^\bullet(G; \mathbb{R}_{(\varepsilon)})$ and $\textup{H}_{b}^\bullet(\Gamma; \mathbb{R}_{(\varepsilon)})$. Since the boundary at infinity $\partial \hyp^n$ of the hyperbolic $n$-space $\mathbb{H}^n$ can be realized as quotient of either $G$ or $G^+$ by a minimal parabolic subgroup, it may be identified with their Furstenberg-Poisson boundary. Therefore, applying Theorem~\ref{teor:monod:resolution:L:inf}, we know that $\textup{H}_{cb}^\bullet(G; \mathbb{R}_{(\varepsilon)})$ is isometrically isomorphic to the cohomology of complex of $G$-invariants $(\textup{L}^{\infty}((\partial \mathbb{H}^n)^{\bullet+1}; \mathbb{R}_{(\varepsilon)})^{G} , \delta^\bullet)$. For ease of notation along the paper we will identify $\partial \mathbb{H}^n \cong \mathbb{S}^{n-1}$. The same strategy also works for any lattice $\Gamma < G$, i.e. $\textup{H}_{b}^\bullet(\Gamma; \mathbb{R}_{(\rho)})$ is isometrically isomorphic to the cohomology of the complex of $G$-invariants $(\textup{L}^{\infty}((\partial \mathbb{H}^n)^{\bullet+1}; \mathbb{R}_{(\rho)})^{\Gamma} , \delta^\bullet)$.

We introduce now a cocycle which will be a key tool in our paper.

\begin{deft}\label{def:vol:cociclo}
We define the \emph{volume cocycle}
$$
\Vol_n \colon (\mathbb{S}^{n-1})^{n+1} \rightarrow \mathbb{R}_\varepsilon \ 
$$
as
$$
(\xi_0, \cdots, \xi_n) \mapsto \mbox{signed volume of the convex hull of } \{\xi_0, \cdots, \xi_n\} \ .
$$
In the sequel we will see $\Vol_n$ as an element of both the spaces $\textup{L}^{\infty}((\mathbb{S}^{n-1})^{n+1}; \mathbb{R}_\varepsilon)^G$ and $\mathcal{B}^{\infty}((\mathbb{S}^{n-1})^{n+1}; \mathbb{R}_\varepsilon)^G$.
\end{deft}

One of the peculiar features of the volume cocycle is that it is in fact the unique representative of its cohomology class. Indeed, following~\cite[Lemma~1]{bucher2:articolo}, the absence of $(n-1)$-cocycles implies that cohomology groups $\textup{H}^n_{cb}(G; \mathbb{R}_\varepsilon)$ and $\textup{H}^n_c(G; \mathbb{R}_\varepsilon)$ coincide with the subspaces of the $n$-cocycles sitting inside $\textup{L}^{\infty}((\mathbb{S}^{n-1})^{n+1}; \mathbb{R}_{\varepsilon})^{G}$ and $\textup{C}_c((\mathbb{H}^n)^{n+1};\R_{(\varepsilon)})^G$, respectively (in the latter case we are using \cite[Chapter III, Proposition 2.3]{guichardet} and the fact that $\mathbb{H}^n$ is the quotient of $G$ or $G^+$ by a maximal compact subgroup). Note that $\textup{C}_c((\mathbb{H}^n)^{n+1};\R_{\varepsilon})^G$ denotes the space of continuous real-valued $G$-invariant functions on $(n+1)$-tuples of points of $\hyp^n$. Since $\textup{H}^n_c(G; \mathbb{R}_\varepsilon) \cong \mathbb{R}$ and it is generated by the volume form, which admits $\Vol_n$ as a bounded representative, we have that the comparison map
$$
\comp^n \colon \textup{H}^n_{cb}(G; \mathbb{R}_\varepsilon) \rightarrow \textup{H}^n_c(G; \mathbb{R}_\varepsilon)
$$
is an isometric isomorphism in top dimensional degree (see~\cite[Proposition~2]{bucher2:articolo}).

\subsection{The orientation cocycle and the Euler class} \label{sec:euler:class}

In this section we are going to introduce the notion of Euler class. We will use the orientation cocycle to get a preferred representative for this class. We refer the reader to~\cite{iozzi02:articolo} for a more detailed description about these notions. 

Fix an orientation on $\bbS^1$ and denote by $\homs$ the group of orientation-preserving homeomorphisms endowed with the discrete topology. Consider the group $\textup{Homeo}_\Z(\R)$ of the homeomorphisms of the real line commuting with the integer translation $T \colon \R \rightarrow \R$ given by $T(x)=x+1$, for every $n \in \, \mathbb{Z}$. This group can be thought of as the universal covering of $\homs$ and the projection $p \colon \textup{Homeo}_\Z(\R) \rightarrow \homs$ fits in the following central extension:
$$
0 \rightarrow \mathbb{Z} \xrightarrow{\iota} \textup{Homeo}_\Z(\R) \xrightarrow{p} \homs \rightarrow 0 \ ,
$$
where $\iota(n) \coloneqq T^n$ for every $n \in \, \mathbb{Z}$. We want now to construct a cocycle associated to the previous central extension. Notice that given an element $f \in \homs$ there exists a unique lift $\tilde f \in \textup{Homeo}_\Z(\R)$ which satisfies $\tilde f(0) \in [0,1)$. Hence, the previous construction allows us to define a section as follows
$$
s\colon \homs \rightarrow \textup{Homeo}_\Z(\R), \hspace{10pt} s(f):=\tilde f \ .
$$

Then, one can define an integer valued $2$-cocycle $\mathcal{E} \colon (\homs)^2 \rightarrow \Z$ associated to the previous central extension via the section $s$ (see for instance \cite[Section~2.2, Section~10.2]{miolibro}):
$$
\widetilde{f \circ g} \circ T^{\mathcal{E}(f,g)} = \tilde f \circ \tilde g \ ,
$$
for every $f,g \in \homs$. The cocycle $\mathcal{E}$ determines an integral bounded cohomology class which is mapped to a real bounded cohomology class $\textup{e}_b \in \uph^2_b(\homs,\R)$ under the change of coefficients homomorphism. Notice that the construction of $\textup{e}_b$ is independent of the chosen section $s$ in the definition (see again~\cite[Section~2.2]{miolibro}). 
\begin{deft}
The previous cohomology class $\textup{e}_b$ is called \emph{bounded Euler class}.
\end{deft}
Since we will need to represent explicitly $\textup{e}_b$ via a measurable cocycle defined on the space $(\bbS^1)^3$, we introduce the \emph{orientation cocycle}
\begin{equation}\label{eq:orien:cocycle}
\orien: (\bbS^1)^3 \rightarrow \R, \hspace{10pt} \orien(\xi_0,\xi_1,\xi_2):=
\begin{cases*}
+1  &\text{if $(\xi_0,\xi_1,\xi_2)$ are positively oriented ,}\\
-1   &\text{if $(\xi_0,\xi_1,\xi_2)$ are negatively oriented ,}\\
0    &\text{otherwise .} 
\end{cases*}
\end{equation}

Notice that the function $\orien$ is a bounded $\homs$-invariant Borel cocycle on $\bbS^1$ and so it canonically determines a bounded cohomology class in $\textup{H}_b^2(\homs; \mathbb{R})$, as explained at the end of Section~\ref{sec:bm} (see Equation~\eqref{eq:map:end:sec:bm}). Since the orientation cocycle can be thought of as the area of an ideal triangle divided by $\pi$, one can deduce that it is a cocycle as a consequence of Stokes Theorem. Moreover, using the previous definition of $\mathcal{E}$, Iozzi~\cite[Lemma 2.1]{iozzi02:articolo} shows that 
\begin{equation}\label{orientation:eulero}
-2\textup{e}_b=[\orien] \ .
\end{equation}

Hence, we get the following:

\begin{deft}
We define the \emph{Euler cocycle} as the representative of the Euler class given by
$$
\epsilon = -\frac{\orien}{2} \in \mathcal{B}^\infty((\bbS^1)^3;\R)^{\homs} \ .
$$
\end{deft}

\subsection{Relative cohomology}\label{sec:rel:coom}

As we mentioned in the introduction our definition of volume of measurable cocycles will apply for both uniform and non-uniform torsion-free lattices $\Gamma < G^+$. However, following~\cite{bucher2:articolo}, when we deal with the non-uniform case, it is necessary to work with the \emph{relative} bounded cohomology of spaces. For the convenience of the reader we recall here its definition. 

Let $X$ be a topological space and let $A \subset X$ be a subspace. Given a singular simplex $\sigma \colon \Delta^\bullet \rightarrow X$, we denote by $\textup{Im}(\sigma)$ its image. Recall that the relative cohomology $\textup{H}^\bullet(X, A; \mathbb{R})$ is computed via the following cochain complex
$$
(\textup{C}^\bullet(X, A; \mathbb{R}),\delta^\bullet) \coloneqq  (\{ f \in \, \textup{C}^\bullet(X; \mathbb{R}) \, | \, f(\sigma) = 0 \mbox{ if } \textup{Im}(\sigma) \subset A \} , \delta^\bullet) \ ,
$$
where $\textup{C}^\bullet(X; \mathbb{R})$ denotes the space of singular cochains in $X$ and $\delta^\bullet$ is the usual coboundary operator. We define the \emph{bounded cochain complex} as 
$$
(\textup{C}^\bullet_b(X, A; \mathbb{R}),\delta^\bullet) \coloneqq (\{ f \in \, \textup{C}^\bullet(X, A; \mathbb{R}) \, | \, \lVert f \rVert_\infty < +\infty \} , \delta^\bullet) \ ,
$$
where $\lVert f \rVert_\infty \coloneqq \sup \{| f(\sigma) | \, | \, \sigma \colon \Delta^\bullet \rightarrow X \mbox{ continuous } \}$.

\begin{deft}
The \emph{relative bounded cohomology} of the topological pair $(X, A)$, denoted by $\textup{H}^\bullet_b(X, A ; \mathbb{R})$, is the cohomology of the bounded complex $(\textup{C}^\bullet_b(X, A; \mathbb{R})  , \delta^\bullet)$.
\end{deft}

\begin{oss}
The previous notion restricts to the bounded cohomology of a space $X$ when we consider the pair $(X, \emptyset)$.
\end{oss}

The deep connection between bounded cohomology of spaces and bounded cohomology of groups is nicely expressed via Gromov's mapping theorem~\cite{Grom82} (see~\cite[Corollary~4.4.5]{FM:Grom} for a topological proof by Frigerio and the first author following Gromov's approach and Ivanov's proof~\cite[Theorem~4.1]{Ivanov} using homological algebra). More precisely, let $X$ be a topological space and let $\pi_1(X)$ denote its fundamental group. Then, there exists an isometric isomorphism in bounded cohomology
\begin{equation}\label{eq:gromov:map}
\xymatrix{
g_X \colon\textup{H}_b^\bullet(\pi_1(X); \mathbb{R}) \ar[r]^-{\cong}  & \textup{H}_b^\bullet(X; \mathbb{R}) \ .
}
\end{equation} 

Let $n \geq 3$. The argument that follows actually holds also for $n=2$, but we will not need this fact. Let now $\Gamma < G^+$ be a torsion-free non-uniform lattice and let $M = \Gamma \backslash \mathbb{H}^n$. It is well known that $M$ is a complete finite-volume hyperbolic manifold. Let $N$ be a compact core of $M$, that is a compact subset of $M$ such that $M \setminus N$ consists of the disjoint union of finitely many horocyclic neighbourhoods of cusps. Let us call them $E_j$ with $j = 1, \cdots, k$. Note that all compact cores of $M$ are homotopy equivalent.

 We are going to explain now how to evaluate an  element $\alpha \in \, \textup{H}^n_b(\Gamma; \mathbb{R})$ on the fundamental class $[N, \partial N] \in \, \textup{H}_n(N, \partial N; \mathbb{R})$ of $N$ via the Kronecker product. Recall that the Kronecker product is the following bilinear form
$$
\langle \cdot , \cdot \rangle \colon \textup{H}^n(N, \partial N; \mathbb{R}) \times \textup{H}_n(N, \partial N; \mathbb{R}) \rightarrow \mathbb{R} \ .
$$

To that end, we need to  describe a map from $\textup{H}^n_b(\Gamma; \mathbb{R})$ to $\textup{H}^n(N, \partial N; \mathbb{R})$. First, let us consider the isometric isomorphism $g_M$ described in Gromov's mapping theorem (see Equation~(\ref{eq:gromov:map})). This allows us to map $\alpha$ to $g_M(\alpha) \in \, \textup{H}^n_b(M; \mathbb{R})$. Hence, we are reduced to describe a map from $\textup{H}^n_b(M; \mathbb{R})$ to $\textup{H}^n(N, \partial N; \mathbb{R})$. Recall that the short exact sequence
$$
0 \rightarrow \textup{C}_b^\bullet(M, M \setminus N; \mathbb{R}) \rightarrow \textup{C}_b^\bullet(M; \mathbb{R}) \rightarrow\textup{C}_b^\bullet(M \setminus N ; \mathbb{R}) \rightarrow 0
$$
induces a long exact sequence
$$
\cdots \rightarrow \textup{H}_b^{\bullet - 1}(M \setminus N; \mathbb{R}) \rightarrow \textup{H}^\bullet_b(M, M \setminus N; \mathbb{R}) \rightarrow \textup{H}^\bullet_b(M; \mathbb{R}) \rightarrow \textup{H}_b^\bullet(M \setminus N; \mathbb{R}) \rightarrow \cdots \ .
$$

Since $M \setminus N = \sqcup_{j = 1}^k E_j$ is the disjoint union of spaces with virtually Abelian fundamental groups, whence with amenable fundamental groups, Gromov's mapping theorem implies that $\textup{H}_b^{\bullet \geq 1}(M \setminus N; \mathbb{R}) = 0$. This shows that the inclusion $j \colon (M, \emptyset) \rightarrow (M, M \setminus N)$ induces an isomorphism in bounded cohomology in degree greater than or equal to $2$. We mention here (but, we will not use it later) that $\textup{H}^\bullet_b(j) \colon \textup{H}^\bullet_b(M, M \setminus N; \mathbb{R})  \rightarrow \textup{H}^\bullet_b(M; \mathbb{R})$ is in fact an isometric isomorphism in degree greater than or equal to $2$, as proved by Bucher, Burger, Frigerio, Iozzi, Pagliantini and Pozzetti~\cite{BBFIPP}. 

Therefore, since $n \geq 3$ and $(M, M \setminus N) \simeq (N, \partial N)$ via a homotopy of pairs $h$, the composition
$$
\xymatrix{
\textup{H}_b^n(\Gamma; \mathbb{R}) \ar[r]^-{g_M^n} & \textup{H}^n_b(M; \mathbb{R}) \ar[r]^-{\textup{H}^n_b(j)^{-1}} & \textup{H}^n_b(M, M \setminus N; \mathbb{R}) \ar[r]^-{\textup{H}^n_b(h)} & \textup{H}^n_b(N, \partial N; \mathbb{R}) , 
}
$$
is an isometric isomorphism. We denote the previous composition by $\textup{J}^n$.

We are now able to show how to evaluate an element $\alpha \in \, \textup{H}^n_b(\Gamma; \mathbb{R})$ on the fundamental class $[N, \partial N] \in \, \textup{H}_n(N, \partial N; \mathbb{R})$ of $N$. Indeed, it is now sufficient to take the following Kronecker product
$$
\langle \comp^n \circ\textup{J}^n (\alpha), [N, \partial N] \rangle \in \, \mathbb{R} \ ,
$$
where $\comp^n \colon \textup{H}^n_b(N, \partial N; \mathbb{R}) \rightarrow \textup{H}^n(N, \partial N; \mathbb{R})$ denotes the comparison map.

\subsection{Transfer maps}\label{sec:trans}

Let $n \geq 2$ and let $G$ be as usual $\Isom(\mathbb{H}^n)$. Let $G^+ < G$ be the subgroup of orientation-preserving isometries. Consider $\Gamma < G^+$ a torsion-free non-uniform lattice. Let $M = \Gamma \backslash \mathbb{H}^n$ and let $N$ be its compact core. In this section we are going to introduce two standard maps which allow to \emph{transfer} the information contained in the bounded cohomology of $\Gamma$ and in the cohomology of $(N, \partial N)$ to the  continuous bounded and the continuous cohomologies of $G$, respectively. More precisely, following~\cite{bucher2:articolo} (see also~\cite{BIW1, MichelleKim}), we define two natural \emph{transfer maps}
$$
\trans_\Gamma^\bullet \colon \textup{H}^\bullet_b(\Gamma; \mathbb{R}) \rightarrow \textup{H}_{cb}^\bullet(G; \mathbb{R}_\varepsilon) 
$$
and
$$
\tau_{\text{dR}}^\bullet \colon \textup{H}^\bullet(N, \partial N; \R) \rightarrow \textup{H}^\bullet_c(G; \mathbb{R}_\varepsilon) \ .
$$

\begin{oss}
Note that we are going to define the transfer maps in the case of non-uniform lattices. However, the same techniques apply \emph{verbatim} to the uniform case.
\end{oss}

Let us begin by recalling the definition of $\trans_\Gamma$.  

\begin{deft}\label{def:trans:map}
Let $\Gamma$ and $G$ as above. Let $\widehat{\trans}_\Gamma$ be the following cochain map
$$
\widehat{\trans}^\bullet_\Gamma \colon \textup{L}^\infty((\mathbb{S}^{n-1})^{\bullet+1}; \mathbb{R})^\Gamma \rightarrow \textup{L}^\infty((\mathbb{S}^{n-1})^{\bullet + 1}; \mathbb{R}_\varepsilon)^G \ ,
$$
$$
\widehat{\trans}^\bullet_\Gamma (\psi) (\xi_0, \cdots, \xi_{\bullet}) \coloneqq \int_{\Gamma \backslash G} \varepsilon(\overline{g}^{-1}) \cdot \psi(\overline{g} \cdot \xi_0, \cdots, \overline{g} \cdot \xi_{\bullet}) d \mu (\overline{g})
$$
where $\psi \in \, \textup{L}^\infty((\mathbb{S}^{n-1})^{\bullet + 1}; \mathbb{R})^\Gamma$, $\xi_0, \cdots, \xi_{\bullet} \in \, \mathbb{S}^{n-1}$ and $\mu$ is the normalized invariant probability measure on $\Gamma \backslash G$ induced by the Haar measure. 
\end{deft}
Notice that by linearity it is immediate to check that $\widehat{\trans}^\bullet_\Gamma$ commutes with the coboundary operator. Moreover, the following computation shows that if $\psi$ is $\Gamma$-invariant, then $\widehat{\trans}^\bullet_\Gamma (\psi)$ is $G$-invariant. For every $g_0 \in G$, we have

\begin{align*}
g_0 \cdot (\widehat{\trans}^\bullet_\Gamma (\psi) (\xi_0, \cdots, \xi_{\bullet})) &=\varepsilon(g_0)\cdot \widehat{\trans}^\bullet_\Gamma(g_0^{-1} \cdot \xi_0,\ldots,g_0^{-1} \cdot \xi_n)=\\
&=\int_{\Gamma \backslash G} \varepsilon(\overline{g_0g^{-1}})\psi(\overline{g g_0^{-1}} \cdot \xi_0,\ldots \overline{g g_0^{-1}} \cdot \xi_n)d\mu(\overline{g})=\\
&=\int_{\Gamma \backslash G} \varepsilon(\overline{g}^{-1})\psi(\overline{g} \cdot \xi_0,\ldots \overline{g} \cdot \xi_n)d\mu(\overline{g})= \\
&=(\widehat{\trans}^\bullet_\Gamma(\psi)(\xi_0,\ldots,\xi_n)) \ ,
\end{align*}
where to move from the first line to the second one we used the fact that $\varepsilon$ is a homomorphism and to move from the second line to the third one we used the unimodularity of $G$ and the $G$-invariance of the measure $\mu$. As a consequence of the previous computation we can give the following

\begin{deft}
We define the \emph{transfer map} $\trans_\Gamma$ as the induced map in cohomology by $\widehat{\trans}_\Gamma$:
$$
\trans_\Gamma^\bullet \colon \textup{H}^\bullet_{b}(\Gamma; \mathbb{R}) \rightarrow \textup{H}^\bullet_{cb}(G; \mathbb{R}_\varepsilon) \ .
$$
\end{deft}

Unfortunately, the definition of the transfer map $\tau_{\text{dR}}$ needs some more work. We define
$$
\tau_{\text{dR}}^\bullet \colon \textup{H}^\bullet(N, \partial N; \mathbb{R}) \rightarrow \textup{H}^\bullet_c(G; \mathbb{R}_\varepsilon)
$$
as the composition of two maps. Recall that the relative de Rham cohomology $\textup{H}^\bullet_{\text{dR}}(M, M \setminus N; \mathbb{R})$ is defined as the cohomology of the complex of differential forms $\Omega^\bullet(M, M \setminus N; \mathbb{R})$ defined on $M$ and whose restriction to $M \setminus N$ vanishes. Via the relative de Rham isomorphism $\Psi_{\text{dR}}$ we have
$$
\xymatrix{
\textup{H}^\bullet_\textup{dR}(M, M \setminus N ; \mathbb{R}) \ar[r]^-{\Psi_\text{dR}^\bullet}_-{\cong} &\textup{H}^\bullet(M, M \setminus N; \mathbb{R}) \ar[r]^-{\textup{H}^\bullet(h)}_-{\cong} & \textup{H}^\bullet(N, \partial N; \mathbb{R}) \ ,
}
$$
where $\textup{H}^\bullet(h)$ is the isomorphism induced by  the homotopy equivalence of pairs $h \colon (M, M \setminus N) \rightarrow (N, \partial N)$. Let us denote by $\Psi^\bullet = \textup{H}^\bullet(h) \circ \Psi_{\text{dR}}^\bullet$ the composition of the previous two isomorphism. Since we want to define $\tau_{\text{dR}}^\bullet$ as the composition of $(\Psi^\bullet)^{-1}$ with another map, we are reduced to introduce a map 
$$
\trans_{\text{dR}}^\bullet \colon \textup{H}^\bullet_\textup{dR}(M, M \setminus N ; \mathbb{R}) \rightarrow \textup{H}^\bullet_c(G; \mathbb{R}_\varepsilon) \ .
$$
Following~\cite{bucher2:articolo}, let us denote by $U = \pi^{-1}(M \setminus N)$ the preimage under the universal covering projection $\pi \colon \mathbb{H}^n \rightarrow M$ of the finite disjoint union of horocyclic neighbourhoods of cusps. Since we can identify the \emph{relative} differential forms $\Omega^\bullet(M, M \setminus N ; \mathbb{R}) $ with the $\Gamma$-invariant relative differential forms $\Omega(\mathbb{H}^n, U ; \mathbb{R})^\Gamma$, we get the following

\begin{deft}
Let $\Gamma$ and $G$ be as above. We call $\widehat{\trans}_{\text{dR}}$ the cochain map
$$
\widehat{\trans}_{\text{dR}}^\bullet \colon \Omega^\bullet(\mathbb{H}^n, U ; \mathbb{R})^\Gamma \rightarrow \Omega^\bullet(\mathbb{H}^n; \mathbb{R}_\varepsilon)^G
$$
defined by
$$
\widehat{\trans}_{\text{dR}}^\bullet (\psi) \coloneqq \int_{\Gamma \backslash G} \varepsilon(\overline{g}^{-1}) (\overline{g}^* \psi) d\mu(\overline{g}) \ ,
$$
where $\psi \in \, \Omega^\bullet(\mathbb{H}^n, U ; \mathbb{R})^\Gamma$ and $\mu$ is the normalized invariant probability measure on $\Gamma \backslash G$.
\end{deft}

Similar arguments that we used in the case of $\trans_\Gamma$ show that $\widehat{\trans}_{\text{dR}}$ is indeed a cochain map and that if $\psi$ is $\Gamma$-invariant, then $\widehat{\trans}_{\text{dR}}(\psi)$ is $G$-invariant.

\begin{deft}
We denote by $\textup{H}^\bullet(\widehat{\trans}_{\text{dR}}^\bullet)$ the induced map in coohomology:

$$\textup{H}^\bullet(\widehat{\trans}_{\text{dR}}^\bullet): \textup{H}^\bullet_{\text{dR}}(M, M \setminus N; \mathbb{R}) \rightarrow \textup{H}^\bullet(\Omega^\bullet(\mathbb{H}^n; \mathbb{R}_\varepsilon)^G) \ .$$
\end{deft}

Since any $G$-invariant differential form on $\mathbb{H}^n$ is closed by Cartan's Lemma, we may identify  $\textup{H}^\bullet(\Omega^\bullet(\mathbb{H}^n; \mathbb{R}_\varepsilon)^G)$ with $\Omega^\bullet(\mathbb{H}^n; \mathbb{R}_\varepsilon)^G$.  Therefore, in order to define $\tau_{\text{dR}}$, it only remains to consider the Van Est isomorphism (see \cite[Corollary~7.2]{guichardet})
$$
\mathcal{VE}^\bullet \colon \Omega^\bullet(\mathbb{H}^n; \mathbb{R}_\varepsilon)^G \rightarrow \textup{H}_c^\bullet(G; \mathbb{R}_\varepsilon) \ .
$$

\begin{deft}\label{def:tau:dr}
We define the transfer map 
$$\tau_{\text{dR}}^\bullet \colon \textup{H}^\bullet(N, \partial N; \mathbb{R}) \rightarrow \textup{H}^\bullet_c(G; \mathbb{R}_\varepsilon)$$
as follows
$$
\tau_{\text{dR}}^\bullet \coloneqq \mathcal{VE}^\bullet \circ \textup{H}^\bullet(\widehat{\trans}_{\text{dR}}^\bullet) \circ (\Psi^\bullet)^{-1}   \ ,
$$
where we use again the identification $\textup{H}^\bullet(\Omega^\bullet(\mathbb{H}^n; \mathbb{R}_\varepsilon)^G) \cong \Omega^\bullet(\mathbb{H}^n; \mathbb{R}_\varepsilon)^G$.
\end{deft}

\begin{oss}\label{tau:iso:grado:massimo}
Note that $\tau_{\text{dR}}$ is an isomorphism in top degree as mentioned by Bucher, Burger and Iozzi~\cite{bucher2:articolo}. 
\end{oss}

We conclude this section by a remarkable fact proved by Bucher, Burger and Iozzi~\cite[Proposition~3]{bucher2:articolo}. The transfer maps
$$
\trans_\Gamma^\bullet \colon \textup{H}^\bullet_b(\Gamma; \mathbb{R}) \rightarrow \textup{H}_{cb}^\bullet(G; \mathbb{R}_\varepsilon) 
$$
and
$$
\tau_{\text{dR}}^\bullet \colon \textup{H}^\bullet(N, \partial N; \R) \rightarrow \textup{H}^\bullet_c(G; \mathbb{R}_\varepsilon)
$$
fits in the following commutative diagram
\begin{equation}\label{eq:diagramma:commutativo:trasfer:maps}
\xymatrix{
\textup{H}^\bullet_b(\Gamma; \mathbb{R}) \ar[d]_-{\textup{J}^\bullet}^-{\cong} \ar[rd]^-{\trans_\Gamma^\bullet} \\
\textup{H}^\bullet_b(N, \partial N; \mathbb{R})  \ar[d]_-{\comp^\bullet} & \textup{H}^\bullet_{cb}(G; \mathbb{R}_\varepsilon) \ar[d]^-{\comp^\bullet} \\
\textup{H}^\bullet(N, \partial N; \mathbb{R}) \ar[r]_-{\tau_{\text{dR}}^\bullet} & \textup{H}^\bullet_c(G; \mathbb{R}_\varepsilon) \ .
}
\end{equation}

\begin{oss}
Note that when $G=\textup{PSL}(2,\R)$, we will simply use $\mathbb{R}$ as coefficients in the cohomology groups appearing in Diagram (\ref{eq:diagramma:commutativo:trasfer:maps}). Indeed, in this special case the the trivial action on $\R$ agrees with the one twisted by the sign.
\end{oss}

\section{A technical lemma}\label{sec:lemma:tecnico}

The aim of this section is to prove a fundamental technical lemma that we need in the sequel. Let $H$ be a locally compact group and let $\Gamma$ be a discrete group. We will suppose that $\Gamma$ admits a Furstenberg-Poisson boundary $B(\Gamma)$ and that $H$ acts measurably on compact completely metrizable space $Y$.

Suppose that there exists a sign homomorphism $\varepsilon \colon H \rightarrow \{-1, 1\}$ and denote by $H^+=\ker \varepsilon$. For instance this is the case when either $H=\Isom(\hyp^n)$ or $H=\homs$. In the latter case the sign homomorphism is trivial and $H=H^+$. 

Let $(X, \mu_X)$ be a standard Borel probability $\Gamma$-space. Consider the measurable cocycle $\sigma \colon \Gamma \times X \rightarrow H^+$ and suppose that there exists an essentially unique boundary map $\phi \colon B(\Gamma) \times X \rightarrow Y$. Here we are using the generalized version of boundary map introduced in Remark \ref{oss:boundary:genmap}.

We will be interested in the sequel on the following problem: given a cocycle $\psi \in \, \mathcal{B}^\infty(Y^{\bullet+1}; \mathbb{R}_{(\varepsilon)})^H$ we want to \emph{pullback} it along $\phi$, whence $\sigma$, in order to obtain a cocycle in $\textup{L}^\infty(B(\Gamma)^{\bullet+1}; \mathbb{R})^\Gamma$. Recall that $\mathbb{R}_\varepsilon$ is endowed with the $H$-module structure induced by $\varepsilon$ and $\R$ has trivial $\Gamma$-module structure. Our approach will follow some ideas that date back to Bader, Furman and Sauer's paper~\cite[Proposition~4.2]{sauer:articolo} (these ideas are also discussed by the second author for the definition of the Borel invariant of measurable cocycles~\cite[Proposition~3.1]{savini3:articolo}). However, we deal here with a more general version which will be suitable for our later computations. Recall that $\textup{L}^\infty(X)$ is a $\Gamma$-module endowed with the following action
$$
\gamma \cdot \big( f (x) \big) = \gamma \cdot f (\gamma^{-1} \cdot x) =  f (\gamma^{-1} \cdot x) \ ,
$$
where in the last equality we consider $\mathbb{R}$ as a $\Gamma$-module endowed with the trivial action. 

\begin{deft}\label{def:pullback:via:phi}
Under the previous assumptions we define the \emph{pullback along} $\phi$ as the following cochain map
$$
\textup{C}^\bullet(\phi) \colon \mathcal{B}^\infty(Y^{\bullet+1}; \mathbb{R}_{(\varepsilon)})^H \rightarrow \textup{L}^\infty_{\text{w}^*}(B(\Gamma)^{\bullet + 1}; \textup{L}^\infty(X))^\Gamma 
$$
$$
\textup{C}^\bullet(\phi)(\psi)(\xi_0, \cdots, \xi_\bullet) = \big(x \mapsto  \psi(\phi(\xi_0, x), \cdots, \phi(\xi_\bullet, x))\big) \ ,
$$
where $\psi \in \, \mathcal{B}^\infty(Y^{\bullet+1}; \mathbb{R}_{(\varepsilon)})^H$, $\xi_0, \cdots, \xi_\bullet \in \, B(\Gamma)$ and $x \in \, X$.
\end{deft}

The fact that the previous map is a well-defined cochain map is the content of the following lemma.

\begin{lem}\label{lemma:tecnico:pullback}
The pullback map $\textup{C}^\bullet(\phi)$ along $\phi$ is a well-defined norm non-increasing cochain map.
\end{lem}
\begin{proof}
It is immediate to check that $\textup{C}^\bullet(\phi)$ is a norm non-increasing cochain map. Let us prove that it is well-defined, that is  $\textup{C}^\bullet(\phi)(\psi)$ is $\Gamma$-invariant. To that end, we identify $$\textup{L}^\infty_{\text{w}^*}(B(\Gamma)^{\bullet + 1}; \textup{L}^\infty(X)) \cong \textup{L}^\infty(B(\Gamma)^{\bullet + 1} \times X)$$ and we endow it with the diagonal $\Gamma$-action given by
$$
\gamma \cdot f(\xi_0, \cdots, \xi_\bullet, x) = f(\gamma^{-1} \cdot \xi_0, \cdots, \gamma^{-1} \cdot \xi_\bullet, \gamma^{-1} \cdot x) \ ,
$$
where $\gamma \in \, \Gamma$, $\xi_0, \cdots, \xi_\bullet \in \, B(\Gamma)$ and $x \in \, X$, respectively. Since $\textup{C}^\bullet(\phi)(\psi)$ lies in the $\Gamma$-module $\textup{L}^\infty(B(\Gamma)^{\bullet + 1} \times X)$, we only have to prove that it is $\Gamma$-invariant with respect to the previous action. Given $\gamma \in \, \Gamma$, $\xi_0, \cdots, \xi_\bullet \in \, B(\Gamma)$ and $x \in \, X$, we have
\begin{align*}
\gamma \cdot \textup{C}^\bullet(\phi) (\psi)(\xi_0, \cdots, \xi_\bullet) (x) &= \gamma \cdot  \textup{C}^\bullet(\phi)(\psi) (\xi_0, \cdots, \xi_\bullet, x ) \\
&= \textup{C}^\bullet(\phi) (\psi) (\gamma^{-1} \cdot \xi_0, \cdots, \gamma^{-1} \cdot \xi_\bullet, \gamma^{-1} \cdot x ) \\
&= \psi \big(\phi(\gamma^{-1} \cdot \xi_0, \gamma^{-1} \cdot x), \cdots, \phi(\gamma^{-1} \cdot \xi_\bullet, \gamma^{-1} \cdot x) \big) \\
&= \psi \big(\sigma(\gamma^{-1}, x) \phi(\xi_0, x), \cdots , \sigma(\gamma^{-1}, x) \phi(\xi_\bullet, x) \big) \\
&= \varepsilon(\sigma(\gamma^{-1}, x)^{-1}) \cdot \textup{C}^\bullet(\phi) (\psi) (\xi_0, \cdots, \xi_\bullet, x ) \\ &= \textup{C}^\bullet(\phi)(\psi) (\xi_0, \cdots, \xi_\bullet) (x ) \ ,
\end{align*}
where in the fourth step we used the $\sigma$-equivariance of $\phi$ and in the following one the $H$-invariance of $\psi$. Moreover, the fact that $\sigma$ takes value in $H^+$ allows us to conclude the chain of equalities. This proves that the pullback along $\phi$ is a well-defined, whence the thesis.
\end{proof}

As explained before we want to pullback cocycles from the space $\mathcal{B}^\infty(Y^{\bullet+1}; \mathbb{R}_{(\varepsilon)})^H$ to $\textup{L}^\infty(B(\Gamma)^{\bullet+1}; \mathbb{R})^\Gamma$ via $\sigma$ or, equivalently, $\phi$. The pullback along $\phi$ just introduced allows us to go from $\mathcal{B}^\infty(Y^{\bullet+1}; \mathbb{R}_{(\varepsilon)})^H$ to $\textup{L}^\infty_{\text{w}^*}(B(\Gamma)^{\bullet + 1}; \textup{L}^\infty(X))^\Gamma$. The next ingredient in our construction consists in the following integration (compare with~\cite{sauer:articolo, savini3:articolo}):

\begin{deft}\label{def:integral:chain:map}
We define the \emph{integration map} $\textup{I}^\bullet_X$ as follows
$$
\textup{I}_X^\bullet \colon \textup{L}^\infty_{\text{w}^*}(B(\Gamma)^{\bullet + 1}; \textup{L}^\infty(X))^\Gamma \rightarrow \textup{L}^\infty(B(\Gamma)^{\bullet + 1}; \mathbb{R})^\Gamma
$$
$$
\textup{I}_X^\bullet (\psi) (\xi_0, \cdots , \xi_\bullet) = \int_X \psi (\xi_0, \cdots , \xi_\bullet) (x) d\mu_X(x) \ ,
$$
where $\psi \in \, \textup{L}^\infty_{\text{w}^*}(B(\Gamma)^{\bullet + 1}; \textup{L}^\infty(X))^\Gamma$, $\xi_0, \cdots , \xi_\bullet \in \, B(\Gamma)$ and $\mu_X$ denotes the probability measure of the standard Borel probability $\Gamma$-space $X$.
\end{deft}

It is immediate to check that $\textup{I}^\bullet_X$ is a cochain map. The following lemma shows that it is well-defined and norm non-increasing:

\begin{lem}\label{lemma:int:ben:def}
The integration cochain map $\textup{I}^\bullet_X$ is well-defined and norm non-increasing.
\end{lem}
\begin{proof}
First we have to prove that for every $\psi \in \textup{L}^\infty_{\text{w}^*}(B(\Gamma)^{\bullet + 1}; \textup{L}^\infty(X))^\Gamma$, the cocycle $\textup{I}_X^\bullet (\psi)$ is $\Gamma$-invariant.  Given $\gamma \in \, \Gamma$ and $\xi_0, \cdots, \xi_\bullet \in \, B(\Gamma)$, we have
\begin{align*}
\gamma \cdot \textup{I}_X^\bullet(\psi) (\xi_0, \cdots, \xi_\bullet) &= \textup{I}_X^\bullet(\psi) (\gamma^{-1} \cdot \xi_0, \cdots, \gamma^{-1} \cdot \xi_\bullet) \\
&=\int_X \psi(\gamma^{-1} \cdot \xi_0, \cdots, \gamma^{-1} \cdot \xi_\bullet) (x) d\mu_X(x) \\
&= \int_X \psi( \xi_0, \cdots, \cdot \xi_\bullet) (\gamma \cdot x) d\mu_X(x) \\
&=\int_X \psi( \xi_0, \cdots, \cdot \xi_\bullet) (x) d\mu_X(x) \\ &= \textup{I}_X^\bullet(\psi) (\xi_0, \cdots, \xi_\bullet)  \ ,
\end{align*}
where in the fourth step we used the fact that $\Gamma$ acts in a measure-preserving way. This proves that the integration map is well-defined.

Let $\psi \in \textup{L}^\infty_{\text{w}^*}(B(\Gamma)^{\bullet + 1}; \textup{L}^\infty(X))^\Gamma$ and let us check that $\textup{I}^\bullet_X$ is norm non-increasing. The following computation 
\begin{align*}
\lVert \textup{I}^\bullet_X (\psi) \rVert_\infty &= \big\lVert \int_X \psi( \cdot ) (x) d\mu_X(x) \big\rVert_\infty \\
&\leq \int_X \lVert \psi \rVert_\infty  (x) d\mu_X(x) \\
&=\lVert \psi \rVert_\infty \cdot \mu_X(X) \\ &= \lVert \psi \rVert_\infty \ .
\end{align*}
concludes the proof.
\end{proof}

\begin{oss}\label{oss:only:bounded}
Note that the integration map is well-defined only on bounded cochains. There is no such a map for unbounded ones. This fact will be stressed later in the definition of volume of cocycles of uniform lattices (see Remark~\ref{oss:uniforme:diagramma}). 
\end{oss}

We are now able to construct our desired \emph{pullback}.

\begin{deft}\label{def:pullback:map}
Let $\Gamma$ and $H$ be two groups as at the beginning of this section. Fix a standard Borel probability $\Gamma$-space $(X,\mu_X)$. Consider a cocycle $\sigma:\Gamma \times X \rightarrow H^+$ with essentially unique boundary map $\phi:B(\Gamma) \times X \rightarrow Y$. The \emph{pullback map along $\phi$} is the map
$$
\textup{C}^\bullet (\Phi_X) \colon \mathcal{B}^\infty(Y^{\bullet+1}; \mathbb{R}_{(\varepsilon)})^H  \rightarrow \textup{L}^\infty(B(\Gamma)^{\bullet + 1}; \mathbb{R})^\Gamma
$$
defined as the composition 
$$
\textup{C}^\bullet (\Phi_X) = \textup{I}^\bullet_X \circ \textup{C}^\bullet(\phi) \ .
$$
\end{deft}

We will need in the proof of the Milnor-Wood inequalities (Propositions~\ref{prop:vol:estimate:wood} and~\ref{Milnor:wood:style}) that $\textup{C}^\bullet (\Phi_X)$ is norm non-increasing. Lemmas~\ref{lemma:tecnico:pullback} and \ref{lemma:int:ben:def} imply the following:

\begin{lem}\label{lemma:composizione:norm:non:increasing}
The cochain map $\textup{C}^\bullet (\Phi_X)$ is norm non-increasing.
\end{lem}

\begin{oss}\label{oss:personal:communications}
In personal communications of the second author with Maria Beatrice Pozzetti, she suggested that one could actually get rid of the boundary map to define the pullback. Indeed given a measurable cocycle $\sigma:\Gamma \times X \rightarrow H^+$ it is possible to define the map 
$$
\textup{C}^\bullet_{b}(\sigma):\textup{C}^\bullet_{cb}(H;\R_{(\varepsilon)})^H \rightarrow \textup{C}^\bullet_b(\Gamma;\R)^\Gamma \ ,
$$
$$
\psi \mapsto \textup{C}^\bullet_b(\sigma)(\psi)(\gamma_0,\ldots,\gamma_\bullet):=\int_X \psi(\sigma(\gamma_0^{-1},x)^{-1},\ldots,\sigma(\gamma_\bullet^{-1},x)^{-1})d\mu_X(x) \ .
$$
The previous definition may appear quite strange, but it is actually compatible with the formula appearing in \cite[Theorem 5.6]{sauer:companion} and it is inspired by the cohomological induction defined by Monod and Shalom \cite{MonShal} for measurable cocycles associated to couplings. 

The fact that $\textup{C}^\bullet(\sigma)$ preserves invariant cochains relies on both Equation \ref{cocycleq} and the $\Gamma$-invariance of $\mu_\Omega$. One can check that $\textup{C}^\bullet_b(\sigma)$ is also a cochain map and hence it descends naturally to a map in continuous bounded cohomology
$$
\textup{H}^\bullet_b(\sigma):\textup{H}^\bullet_{cb}(H;\R_{(\varepsilon)}) \rightarrow \textup{H}^\bullet_b(\Gamma,\R) \ , \ \ \textup{H}^\bullet_b(\sigma)([\psi]):=\left[ \textup{C}^\bullet_b(\sigma)(\psi) \right] \ .
$$

Suppose now that $\sigma$ admits an essentially unique boundary map $\phi:B(\Gamma) \times X \rightarrow Y$. Given an element $\psi \in \mathcal{B}^\infty(Y^{\bullet+1};\R_{(\varepsilon)})^H$, by applying \cite[Corollary 2.7]{burger:articolo} we get that $\textup{H}^\bullet_b(\sigma)([\psi])$ admits as a canonical representative $\textup{C}^\bullet(\Phi_X)(\psi)$. In this way we get back our initial approach to pullback via boundary maps. 

Since in the sequel we will need the existence of a boundary map to prove our rigidity results, we prefer to keep using boundary map to define a pullback map in continuous bounded cohomology.
\end{oss}

\section{Volume of Zimmer cocycles in $G^+ = \Isom(\mathbb{H}^n)$}\label{sec:volume:totale}

In this section we extend the classic notion of volume of representations to the much general setting of Zimmer's cocycles. Let $\Gamma <G^+ = \Isom^+(\mathbb{H}^n)$ be a torsion-free non-uniform lattice, where $n \geq 3$ (this dimensional assumption will be assumed all along this section). Let $M = \Gamma \backslash \mathbb{H}^n$ be the complete finite-volume hyperbolic manifold with fundamental group $\Gamma$ and let $N$ be its compact core.  Consider $(X, \mu_X)$ a standard Borel probability $\Gamma$-space. Let $\sigma \colon \Gamma \times X \rightarrow G^+$ be a measurable cocycle which admits an essentially unique boundary map $\phi \colon \mathbb{S}^{n-1} \times X \rightarrow \mathbb{S}^{n-1}$, that is a $\sigma$-equivariant measurable map (see Definition~\ref{boundary}). 

Note that our setup only involves non-uniform lattices. However, all the definitions and results that we will discuss in this section still hold when we deal with uniform ones. For the convenience of the reader, we will stress the slight appropriate modifications needed to adapt our framework to uniform lattices.

\subsection{Definition of volume}\label{sec:def:vol}

In this paragraph we define the notion of \emph{volume} of cocycles. As proved in Section~\ref{sec:lemma:tecnico}, there exists a natural way to pullback a cocycle $c \in \, \mathcal{B}^{\infty}((\mathbb{S}^{n-1})^{\bullet + 1}; \mathbb{R}_\varepsilon)^G$ to $\textup{L}^{\infty}((\mathbb{S}^{n-1})^{\bullet + 1}; \mathbb{R})^{\Gamma}$ via the boundary map $\phi$. Recall that the volume cocycle $\Vol_n$, given in Definition~\ref{def:vol:cociclo}, is an element of $\mathcal{B}^{\infty}((\mathbb{S}^{n-1})^{n + 1}; \mathbb{R}_\varepsilon)^G$. Therefore, we can pullback it from $\mathcal{B}^{\infty}((\mathbb{S}^{n-1})^{\bullet + 1}; \mathbb{R}_\varepsilon)^G$ to $\textup{L}^{\infty}((\mathbb{S}^{n-1})^{\bullet + 1}; \mathbb{R})^{\Gamma}$ via the cochain map $\textup{C}^n(\Phi_X)$ as explained in Section~\ref{sec:lemma:tecnico}. We define the \emph{volume class} associated to $\sigma$ as
$$
\left[\textup{C}^n(\Phi_X)(\Vol_n)\right] \, \in \textup{H}_b^n(\Gamma; \mathbb{R}) \ .
$$

Using the techniques introduced in Section~\ref{sec:rel:coom}, we can map the above element to $\textup{H}_b^n(N, \partial N; \mathbb{R})$. Hence, the composition with the comparison map leads to our definition of volume of cocycles.
\begin{deft}\label{def:per:intro:volume}
Let $\Gamma < G^+$ be a torsion-free non-uniform lattice and let $(X, \mu_X)$ be a standard Borel probability $\Gamma$-space. Let $M = \Gamma \backslash \mathbb{H}^n$ and let $N$ be its compact core. Consider a cocycle $\sigma \colon \Gamma \times X \rightarrow G^+$ which admits an essentially unique boundary map $\phi$. The \emph{volume of} $\sigma$ is defined as follows
$$
\Vol(\sigma) \coloneqq \langle \comp^n \circ \textup{J}^n \left[\textup{C}^n(\Phi_X)(\Vol_n)\right] , [N, \partial N] \ \rangle \in \, \mathbb{R} \ ,
$$
where $\langle \cdot, \cdot \rangle$ is the Kronecker product, $\textup{J}^n$ is the map introduced in Section~\ref{sec:rel:coom} and $[N, \partial N]$ denotes the relative fundamental class of $N$.
\end{deft}

\begin{oss}\label{oss:existence:boundary:map}
Since we defined the volume of a measurable cocycle $\sigma$ using a boundary map, one could naturally ask under which condition such a map actually exists. Monod and Shalom \cite[Proposition 3.3]{MonShal0} proved that it is sufficient to assume that $\sigma$ is \emph{non-elementary}, that is $\sigma$ is not cohomologous to a cocycle taking values into an elementary subgroup of $\textup{PO}^\circ(n,1)$. 
\end{oss}

\subsection{The uniform case}\label{subsec:nuova:uniforme}
As we mentioned in the introduction, our invariant is inspired by the Euler number of self-couplings introduced by Bader, Furman and Sauer~\cite{sauer:articolo}. Indeed, when we adapt the definition of our volume to uniform lattices and we restrict our attention to self-couplings, we recover Bader, Furman and Sauer's Euler number. Moreover, as explained at the end of this section, a careful reading of \cite[Lemma~4.10]{sauer:articolo} shows that one could define a \emph{generalized} Euler number for arbitrary measurable cocycles, dropping both the assumptions on the target and on the measurable space. Then, this latter invariant agrees with our volume in the uniform case. We refer the reader to the end of this section for a discussion about this topic.

Let us explain now how the definition of volume in the non-uniform case can be suitably modified for dealing with uniform lattices. Since $M = \Gamma \backslash \mathbb{H}^n$ is now closed and aspherical, Poincar\'{e} Duality shows that the top dimensional cohomology group $\textup{H}^n(M; \mathbb{R}) \cong \textup{H}^n(\Gamma; \mathbb{R})$ is isomorphic to $\mathbb{R}$ via the isomorphism obtained by the evaluation on the fundamental class $[M]$ of $M$.  Therefore, we get the following definition:
\begin{deft}\label{def:volume:uniforme}
Let $\Gamma < G^+$ be a torsion-free uniform lattice and let $(X, \mu_X)$ be a standard Borel probability $\Gamma$-space. Let $M = \Gamma \backslash \mathbb{H}^n$. Consider a cocycle $\sigma \colon \Gamma \times X \rightarrow G^+$ which admits an essentially unique boundary map $\phi$. Then, we define the \emph{volume of} $\sigma$ to be
$$
\Vol(\sigma) \coloneqq \langle \comp^n \circ g_M^n\left[\textup{C}^n(\Phi_X)(\Vol_n)\right] , [M] \ \rangle \in \, \mathbb{R} \ ,
$$
where $\langle \cdot, \cdot \rangle$ is the Kronecker product, $g^n_M$ the isomorphism of Gromov's mapping theorem and $[M]$ denotes the fundamental class of $M$.
\end{deft}

\begin{oss}\label{oss:uniforme:diagramma}
Note that also in the uniform case our definition of volume of cocycles defined in terms of uniform lattices involves the comparison map. At first glance, the presence of the comparison map may appear rather misleading with respect to the classic definition of volume of representations of uniform lattices (see~\cite{bucher2:articolo}). However, when we deal with measurable cocycles, we need to work with bounded cohomology. Indeed, any representation $\rho \colon \Gamma \rightarrow G^+$ induces pullback maps in both continuous and continuous bounded cohomologies, which fit in the following commutative diagram:
\begin{equation}\label{eq:diagr:comm:unif:repr}
\xymatrix{
\textup{H}_{cb}^n(G; \mathbb{R}_\varepsilon) \ar[rr]^-{\textup{H}_{cb}^n(\rho)} \ar[d]_-{\comp_G^n} && \textup{H}^n_b(\Gamma; \mathbb{R}) \ar[d]_-{\comp^n_\Gamma} \ar@{-->}[rr]^-{g^n_M} && \textup{H}^n_b(M; \mathbb{R}) \ar@{-->}[d]^-{\comp^n} \\
\textup{H}_{c}^n(G; \mathbb{R}_\varepsilon) \ar[rr]_-{\textup{H}^n_c(\rho)} && \textup{H}^n(\Gamma; \mathbb{R}) \ar[rr]_-{\cong} && \textup{H}^n(M; \mathbb{R}) \ ,
}
\end{equation}
Unfortunately, as discussed in Remark~\ref{oss:only:bounded} given a cocycle $\sigma \colon \Gamma \times X \rightarrow G^+$ we are only able to pullback the volume cocycle $\Vol_n$ in bounded cohomology. Indeed, our construction is defined at the level of cochains via the following map
$$
\textup{C}^n (\Phi_X) \colon \mathcal{B}^\infty((\mathbb{S}^{n-1})^{n+1}; \mathbb{R}_{\varepsilon})^G  \rightarrow \textup{L}^\infty((\mathbb{S}^{n-1})^{n + 1}; \mathbb{R})^\Gamma
$$
which is the composition
$$
\textup{C}^n (\Phi_X) = \textup{I}^n_X \circ \textup{C}^n(\phi) \ .
$$
Since the integration cochain map $\textup{I}^n_X$ is well-defined only for bounded cocycles, our pullback provides a class in bounded cohomology.
\end{oss}

As explained above (and in the introduction) one of the source of inspiration of our volume of measurable cocycles is the Euler number of self-couplings introduced by Bader, Furman and Sauer~\cite{sauer:articolo}. However, a careful reading of \cite[Lemma~4.10]{sauer:articolo} shows that one can easily extends Bader-Furman-Sauer's Euler number and their results to arbitrary measurable cocycles. In the sequel, we will refer to the extensions of both their invariant and their results, as \emph{generalized} Euler number and \emph{generalized} Bader-Furman-Sauer's results, respectively. 

For the convenience of the reader, we explain here how to show that our volume of measurable cocycles agrees with the generalized Euler number in the case of uniform lattices. Before introducing the formal definition of the Euler number introduced by Bader, Furman and Sauer~\cite{sauer:articolo}, it is convenient to recall the existence of an \emph{isometric isomorphism} in bounded cohomology due to Monod and Shalom~\cite[Proposition~4.6]{MonShal}. Let $\Gamma$ and $\Lambda$ be two countable discrete groups. Recall by Definition~\ref{couplings} that given a $(\Gamma, \Lambda)$-coupling $(\Omega, m_\Omega)$, there exists an associated right measure equivalence cocycle $\alpha_\Omega \colon \Gamma \times \Lambda \backslash \Omega \rightarrow \Lambda$. It is proved by Monod and Shalom~\cite{MonShal} that $\alpha_\Omega$ induces an isometric isomorphism 
$$
\textup{H}_b^\bullet(\Omega) \colon \textup{H}_b^\bullet(\Lambda; \textup{L}^\infty(\Gamma \backslash \Omega)) \rightarrow \textup{H}^\bullet_b(\Gamma; \textup{L}^\infty(\Lambda \backslash \Omega))
$$
which depends only on the coupling $\Omega$, as suggested by the notation $\uph^\bullet_b(\Omega)$. 

Bader, Furman and Sauer use the isometric isomorphism $\textup{H}_b^\bullet(\Omega)$ in order to define the Euler number associated to self-couplings~\cite{sauer:articolo}. More precisely, given a torsion-free \emph{uniform} lattice $\Gamma < G^+$ and a $(\Gamma, \Gamma)$-coupling $(\Omega, m_\Omega)$, the Euler number associated to $\Omega$ is defined as follows:
$$
\eu(\Omega) = \langle \comp^n \circ g_{M}^n \circ \textup{I}^n_{\Gamma \backslash \Omega} \circ \textup{H}^n_b(\Omega) \circ \textup{H}^n_b(\kappa) \circ \textup{H}^n_{cb}(i)([\Vol_n]) , [M] \rangle \in \, \mathbb{R} \ ,
$$ 
where $\textup{H}^n_b(\kappa) \colon \textup{H}_b^n(\Gamma; \mathbb{R}) \rightarrow \textup{H}_b^n(\Gamma; \textup{L}^\infty(\Gamma \backslash \Omega))$ is the map induced by the change of coefficients, $\textup{H}_{cb}^n(i) \colon \textup{H}_{cb}^n(G^+; \mathbb{R}_\varepsilon) \rightarrow \textup{H}_{b}^n(\Gamma; \mathbb{R})$ is the map induced by the lattice embedding $i \colon \Gamma \rightarrow G^+ < G$ and $M = \Gamma \backslash \mathbb{H}^n$.

Since $\Gamma$ is a torsion-free uniform lattice in $G^+$ also the right measure equivalence cocycle $\alpha_\Omega \colon \Gamma \times \Gamma \backslash \Omega \rightarrow \Gamma$ takes values in $G^+$. Therefore, we can compute its volume via Definition~\ref{def:volume:uniforme}. We show now how our invariant in the uniform case agrees with the \emph{generalized} version of Bader-Furman-Sauer's Euler number. For the convenience of the reader, we will denote by $\Gamma_\ell$ and $\Gamma_r$ the left and the right copy of $\Gamma$, respectively.

Recall that in this setting, the existence of an essentially unique boundary map $\phi \colon \mathbb{S}^{n-1} \times \Gamma_\ell \backslash \Omega \rightarrow \mathbb{S}^{n-1}$ for $\alpha_\Omega$ comes from the general theory of Furstenberg boundaries (see~\cite{furst:articolo73,burger:mozes,MonShal0}). Therefore, there is a well-defined notion of volume of the cocycle $\alpha_\Omega$. Then, we can slightly modify the commutative diagram of \cite[Lemma~4.10]{sauer:articolo} in order to obtain the following one:
\begin{equation}\label{diagr:numero:euler:bfs}
\xymatrix{
\textup{H}^n(\mathcal{B}^\infty((\mathbb{S}^{n-1})^{\bullet+1}; \mathbb{R}_\varepsilon)^{G^+}) \ar[r]^-{\textup{H}^n(\phi)} \ar[d] & \textup{H}^n_b(\Gamma_r; \textup{L}^\infty(\Gamma_\ell \backslash \Omega)) \ar[r]^-{\textup{I}^n_{\Gamma \backslash \Omega}} &  \textup{H}^n_b(\Gamma_r; \mathbb{R}) \ar[rr]^-{\comp^n \circ g^n_M} && \textup{H}^n(M; \mathbb{R}) \\
\textup{H}^n_b(\Gamma_\ell; \mathbb{R}) \ar[r]_-{\textup{H}_b^n(\kappa)} & \textup{H}^n_b(\Gamma_\ell; \textup{L}^\infty(\Gamma_r \backslash \Omega)) \ar[u]_-{\textup{H}^n_b(\Omega)}  \ , &
}
\end{equation}
where the arrow on the left is the composition of the map $\mathfrak{c}^n$ defined in Equation~\eqref{eq:map:end:sec:bm} with the restriction to $\Gamma_\ell$-invariant cochains $\textup{H}^n_{cb}(i)$. Using the previous diagram, one can check that our volume is in fact the generalized Bader-Furman-Sauer's Euler number:
\begin{align*}
\eu(\Omega) &= \langle \comp^n \circ g^n_M \circ \textup{I}^n_{\Gamma \backslash \Omega} \circ \textup{H}^n_b(\Omega) \circ \textup{H}^n_b(\kappa) \circ \textup{H}^n_{cb}(i)([\Vol_n]) , [M] \rangle \\
&=\langle \comp^n \circ g^n_M\left[\textup{C}^n(\Phi_{\Gamma \backslash \Omega})(\Vol_n)\right], [M] \rangle = \Vol(\alpha_\Omega) \ .
\end{align*}

\subsection{Volume of cocycles vs. volume of representations}\label{sec:representations:coupling:volume}

As we mentioned in the introduction, our aim is to define a notion of volume of measurable cocycles which extends the classic volume of representations introduced by Bucher, Burger and Iozzi~\cite{bucher2:articolo}. We formalize here this philosophical approach as follows:

\begin{prop}\label{prop:rep:vol}
Assume $n \geq 3$. Let $\Gamma < G^+=\Isom^+(\hyp^n)$ be a torsion-free non-uniform lattice. Let $\rho \colon \Gamma \rightarrow G^+$ be a non-elementary representation with measurable boundary map $\varphi \colon \mathbb{S}^{n-1} \rightarrow \mathbb{S}^{n-1}$. For any $(X, \mu_X)$ standard Borel probability $\Gamma$-space consider the measurable cocycle 
$
\sigma_\rho \colon \Gamma \times X \rightarrow G^+
$
associated to $\rho$. Then, 
$$
\Vol(\sigma_\rho) = \Vol(\rho) \ .
$$
\end{prop}

\begin{proof}
Recall that the measurable boundary map $\varphi$ of $\rho$ is essentially unique because of the doubly ergodic action of $\Gamma$ on $\mathbb{S}^{n-1}$. Therefore, we construct an essentially unique boundary map $\phi$ for $\sigma_\rho$ as follows:
$$
\phi \colon \mathbb{S}^{n-1} \times X \rightarrow \mathbb{S}^{n-1} , \hspace{5pt} \phi(\xi, x) \coloneqq \varphi(\xi)\ ,
$$
where $\xi \in \, \mathbb{S}^{n-1}$ and $x \in \, X$. Since the boundary map $\phi$ does not depend on the second variable $x \in \, X$, one can check that the following diagram commutes
$$
\xymatrix{
\mathcal{B}^\infty((\mathbb{S}^{n-1})^{n + 1}; \mathbb{R}_\varepsilon)^G \ar[rr]^{\textup{C}^n(\phi)} \ar[dr]_-{\textup{C}^n(\varphi)} &&  \textup{L}^\infty_{\text{w}^*}((\mathbb{S}^{n-1})^{n + 1}; \textup{L}^\infty(X))^\Gamma \ar[dl]^-{\textup{I}^n_X}\\
& \textup{L}^\infty((\mathbb{S}^{n-1})^{n + 1}; \mathbb{R})^\Gamma \ .
}
$$
Recall that by Definition \ref{def:pullback:map} we have $\textup{C}^n(\Phi_X) = \textup{I}^n_X \circ \textup{C}^n(\phi)$ and so the commutativity of the diagram above implies
$$
\textup{C}^n(\Phi_X)(\Vol_n) =\textup{C}^n(\varphi)(\Vol_n) \ .
$$
Since Buger and Iozzi~\cite{burger:articolo} proved that $\textup{C}^n(\varphi)(\Vol_n)$ is a natural representative of the cohomology class $\textup{H}^n_b(\rho)(\left[\Vol\right])$, we get the following
\begin{align*}
\Vol(\sigma_\rho) &= \langle \comp^n \circ \textup{J}^n \left[\textup{C}^n(\Phi_X)(\Vol_n)\right], [N, \partial N] \rangle \\
&= \langle \comp^n \circ \textup{J}^n \left[\textup{C}^n(\varphi)(\Vol_n)\right], [N, \partial N] \rangle \\
&= \langle \comp^n \circ \textup{J}^n \circ \textup{H}^n_b(\rho)(\left[\Vol_n \right]), [N, \partial N] \rangle  = \Vol(\rho) \ .
\end{align*}
This concludes the proof.
\end{proof}
\begin{oss}\label{oss:everything:uniform:case}
All the results regarding non-uniform lattices can be easily translated to the setting of uniform ones. For convenience of the reader we show here how to extend the previous result to that setting. Keeping the same notation of Proposition~\ref{prop:rep:vol}, we assume now that $\Gamma < G^+$ is a torsion-free uniform lattice. Note that the volume cocycle $\Vol_n$ can be though of as an element of both $\textup{C}^n_{cb}(G; \mathbb{R}_\varepsilon)$ and $\textup{C}^n_c(G; \mathbb{R}_\varepsilon)$. In order to distinguish these two cases we will denote $\Vol_n^b$ and $\Vol_n$ the continuous bounded and the continuous cocycles, respectively. Then, the proof in the presence of non-uniform lattices adapts to the uniform ones as follows:
\begin{align*}
\Vol(\rho) &= \langle \textup{H}^n_{c}(\rho)([\Vol_n]), [M] \rangle \\
                &=\langle \comp^n \circ g_M^n \circ \hcb^n(\rho)([\Vol_n^b]),[M] \rangle \\
                &=\langle \comp^n \circ g_M^n [\textup{C}^n(\varphi)(\Vol_n^b)], [M] \rangle \\
                &=\langle \comp^n \circ g_M^n [\textup{C}^n(\Phi_X)(\Vol^b_n)],[M] \rangle= \Vol(\sigma_\rho) \ ,
\end{align*}
where we used the commutativity of Diagram (\ref{eq:diagr:comm:unif:repr}) completed with the dotted arrows. 

\end{oss}
\begin{oss}
We have just proved that the volume of a non-elementary representation $\rho$ coincides with the volume of the associated cocycle $\sigma_\rho$. On the other hand, it is well-known that elementary representations have vanishing volume. Therefore, it could be interesting to investigate which cocycles play the same role. As we will see later in Definition~\ref{def:lementary:cocycle}, \emph{reducible cocycles} will provide an example of such a family.
\end{oss}

Recall that the classic volume of representations satisfies the following property:
$$
\Vol(g \rho g^{-1}) = \varepsilon(g) \cdot \Vol(\rho) 
$$
for every $g \in \, G$ and every representation $\rho \colon \Gamma \rightarrow G^+$. In particular, when $g \in \, G^+$ the volume of $\rho$ is constant along the conjugacy class of the representation $\rho$. 

Following the analogy between volume of cocycles and volume of representations, we prove now a similar property in our setting. Consider a measurable function $f \colon X \rightarrow G$, where $(X, \mu_X)$ is a probability space. We define the \emph{sign of} $f$ as the function $\varepsilon(f) \colon X \rightarrow \{-1, 1\}$ given by  $\varepsilon(f)(x) = \varepsilon(f(x))$. In the case in which $\varepsilon(f)$ is almost everywhere constant, we will simply denote the real number identified with the essential image of $\varepsilon(f)$ by $\varepsilon(f)$ itself.

Then, we have the following:

\begin{prop}\label{gplus:cohomology}
Assume $n \geq 3$. Let $\Gamma < G^+=\Isom^+(\hyp^n)$ be a torsion-free non-uniform lattice and let $(X, \mu_X)$ be a standard Borel probability $\Gamma$-space. Consider a measurable function $f \colon X \rightarrow G$ such that its sign $\varepsilon(f)$ is almost everywhere constant. Let $\sigma \colon \Gamma \times X \rightarrow G^+$ be a measurable cocycle with an essentially unique boundary map $\phi \colon \mathbb{S}^{n-1} \times X \rightarrow \mathbb{S}^{n-1}$. Then, we have
$$
\Vol(\sigma^f) = \varepsilon(f) \cdot \Vol(\sigma) \ .
$$
In particular, if $f \colon X \rightarrow G^+$, we have that the volume is constant along the $G^+$-cohomology class of $\sigma$.
\end{prop}
\begin{proof}
Recall by Definitions~\ref{def:coomologhi:cicli} and~\ref{boundary:map:cohomology} that $\sigma^f$ is defined as
$$
\sigma^f(\gamma, x) = f(\gamma x)^{-1} \sigma(\gamma, x) f(x) \ ,
$$
for all $\gamma \in \Gamma$ and for almost every $x \in  X$, and its essentially unique boundary map associated is given by
$$
\phi^f \colon \mathbb{S}^{n-1} \times X \rightarrow \mathbb{S}^{n-1} \ , \hspace{10pt} \phi^f(\xi, x) = f(x)^{-1} \phi(\xi, x) \ ,
$$
for almost every $\xi \in \, \mathbb{S}^{n-1}$ and $x \in \, X$. Let us denote by $\Phi_X^f$ the composition of the pullback along by $\phi^f$ with the integration. Then, the volume of $\sigma^f$ can be computed as follows
\begin{align*}
\textup{C}^n(\Phi^f_X)(\Vol_n)(\xi_0, \cdots, \xi_n) &= \int_X \Vol_n(\phi^f(\xi_0, x), \cdots, \phi^f(\xi_n, x)) d \mu_X \\
&= \int_X \Vol_n (f(x)^{-1} \phi(\xi_0, x), \cdots, f(x)^{-1}\phi(\xi_n, x)) d\mu_X \\
&= \int_X \varepsilon(f(x)) \cdot \Vol_n(\phi(\xi_0, x), \cdots, \phi(\xi_n, x)) d \mu_X \\
&= \varepsilon(f) \int_X \Vol_n(\phi(\xi_0, x), \cdots, \phi(\xi_n, x)) d \mu_X \\
&= \varepsilon(f) \cdot \textup{C}^n(\Phi_X)(\Vol)(\xi_0, \cdots, \xi_n) \ .
\end{align*}

Note that in the step from the third and the fourth line, we used the hypothesis on the sign $\varepsilon(f)$ which is almost everywhere constant. Using the previous computation and the linearity of the Kronecker product we get the desired equality
\begin{align*}
\Vol(\sigma^f) &= \langle \comp^n \circ \textup{J}^n\left[ \textup{C}^n(\Phi_X^f)(\Vol_n)\right], [N, \partial N] \rangle \\
&= \varepsilon(f) \cdot \langle \comp^n \circ \textup{J}^n\left[\textup{C}^n(\Phi_X)(\Vol_n)\right], [N, \partial N] \rangle \\
&= \varepsilon(f) \cdot \Vol(\sigma) \ .
\end{align*}

\end{proof}

Using both Proposition~\ref{prop:rep:vol} and Proposition~\ref{gplus:cohomology} we obtain the following: 

\begin{cor}\label{cor:max:vol}
Assume $n \geq 3$. Let $\Gamma < G^+=\Isom^+(\hyp^n)$ be a torsion-free non-uniform lattice. Fix $(X,\mu_X)$ a standard Borel probability $\Gamma$-space. If we denote by $i:\Gamma \rightarrow G^+$ the standard lattice embedding, we have
$$
\Vol(\sigma)=\Vol(\Gamma \backslash \hyp^n)
$$
for every cocycle in the $G^+$-cohomology class of $\sigma_i$. 
\end{cor}

\begin{proof}
By~\cite[Lemma 2]{bucher2:articolo} we know that the standard lattice embedding statisfies $\Vol(i)=\Vol(\Gamma \backslash \hyp^n)$. Since the volume is constant along the $G^+$-cohomology class by Proposition~\ref{gplus:cohomology}, the thesis follows. 
\end{proof}

\begin{oss}\label{oss:commento:estensione:mostow}
At first sight it may seem rather unsatisfactory having the invariance of the volume only over a $G^+$-cohomology class of the standard embedding lattice. However, if $\sigma_i$ denotes the cocycle associated to the standard lattice embedding, Proposition~\ref{gplus:cohomology} ensures that
$$
\lvert \Vol(\sigma^f_i) \rvert = \Vol(\Gamma \slash \mathbb{H}^n) \ ,
$$
for every measurable function $f:X \rightarrow G$ with almost everywhere constant sign $\varepsilon(f)$. Therefore, it is immediate to check that, when we restrict to cocycles associated to representations, Proposition~\ref{gplus:cohomology} and Corollary~\ref{cor:max:vol} show the same behaviour as the one discussed by Bucher, Burger and Iozzi~\cite{bucher2:articolo}. This shows that Corollary~\ref{cor:max:vol} may be interpreted as a natural extension of classic result on representations to the much wider setting of measurable cocycles.
\end{oss}

\subsection{Volume rigidity for Zimmer's cocycles}\label{sec:rigidity:volume}
In this section we establish and study a Milnor-Wood type inequality for volume of cocycles. Remarkably, we will show that the maximal value is attained at cocycles which are cohomologous to the one associated to the standard lattice embedding $i \colon \Gamma \rightarrow G^+$ via a measurable function with essentially constant sign. 

Note that our result in the non-uniform setting extends the Milnor-Wood type inequality for volume of representations proved by Bucher, Burger and Iozzi~\cite[Theorem~1]{bucher2:articolo}. Moreover, recall by Section~\ref{subsec:nuova:uniforme} that in the uniform case our volume of measurable cocycles agrees with the generalized version of the Euler number introduced by Bader, Furman and Sauer~\cite{sauer:articolo}. Therefore, here we provide an alternative proof of the generalized version of the Milnor-Wood type inequality~\cite[Corollary~4.9]{sauer:articolo} for measurable cocycles.

\begin{prop}\label{prop:vol:estimate:wood}
Assume $n \geq 3$. Let $\Gamma < G^+=\Isom^+(\hyp^n)$ be a torsion-free non-uniform lattice and let $(X, \mu_X)$ be a standard probability $\Gamma$-space. Let $\sigma \colon \Gamma \times X \rightarrow G^+$ be a measurable cocycle with essentially unique boundary map $\phi \colon \mathbb{S}^{n-1} \times X \rightarrow \mathbb{S}^{n-1}$. If $M = \Gamma \backslash \mathbb{H}^n$, then we have
$$
| \Vol(\sigma)| \leq \Vol(M) \ .
$$
\end{prop}

\begin{proof}
Let us fix an arbitrary compact core $N$ of $M$. Recall by Diagram (\ref{eq:diagramma:commutativo:trasfer:maps}) that there exists the following commutative square:
\begin{equation}\label{eq:diagram:miln:wood:vol}
\xymatrix{
\textup{H}^n(\mathcal{B}^\infty((\mathbb{S}^{n-1})^{\bullet +1}; \mathbb{R}_\varepsilon)^G) \ar[d]_-{\textup{H}^n(\Phi_X)} \\
\textup{H}^n_b(\Gamma; \mathbb{R})  \ar[d]_-{\textup{J}^n}  \ar[rr]^-{\trans_\Gamma^n} && \textup{H}^n_{cb}(G; \mathbb{R}_\varepsilon) \ar[dd]^-{\comp^n} \\
\textup{H}^n_b(N, \partial N; \mathbb{R}) \ar[d]_-{\comp^n} \\
\textup{H}^n(N; \partial N) \ar[rr]_-{\tau^n_{\text{dR}}} && \textup{H}^n_c(G; \mathbb{R}_\varepsilon) \ . 
}
\end{equation}
Here the transfer maps are the one introduced in Section~\ref{sec:trans}. For ease of notation, we set $\omega_n^b = [\Vol_n] \in \, \textup{H}^n_{cb}(G; \mathbb{R}_\varepsilon)$ and $\omega_n = [\Vol_n] \in \, H_c^n(G; \mathbb{R}_\varepsilon)$. Since by~\cite[Proposition~2]{bucher2:articolo} we have $\textup{H}^n_{cb}(G; \mathbb{R}_\varepsilon) \cong \mathbb{R}\omega_n^b$, there must exist a $\lambda \in \, \mathbb{R}$ such that 
\begin{equation}\label{equation:trans:lambda}
\trans_\Gamma^n \left[ \textup{C}^n(\Phi_X) (\Vol_n) \right] = \lambda \omega_n^b \ .
\end{equation}
If we now apply the comparison map to both sides, we get
$$
\comp^n \circ \trans^n_\Gamma \left[\textup{C}^n(\Phi_X) (\Vol_n) \right] = \lambda \omega_n \ .
$$

We denote by $\omega_{N, \partial N}$ the relative volume form in the relative de Rham complex, that is the unique form such that $\langle \omega_{N, \partial N}, [N, \partial N] \rangle = \Vol(M) \ .$ Then, by construction, we have $\tau^n_{\text{dR}}([\omega_{N, \partial N}]) = \omega_n$. Since Diagram (\ref{eq:diagram:miln:wood:vol}) commutes, the following chain of equalities holds:
\begin{align*}
\tau^n_{\text{dR}} \big( \comp^n \circ \textup{J}^n \left[ \textup{C}^n(\Phi_X) (\Vol_n)\right] \big) &= \comp^n \circ \trans_\Gamma^n \left[ \textup{C}^n (\Phi_X)(\Vol_n)\right] \\ &= \lambda \omega_n \\ &= \tau_{\text{dR}}^n(\lambda [\omega_{N, \partial N}]) \ .
\end{align*}
Since $\tau_{\text{dR}}^n$ is injective as recalled in Remark~\ref{tau:iso:grado:massimo}, we have 
$$
\comp^n \circ \textup{J}^n \left[ \textup{C}^n(\Phi_X) (\Vol_n)\right] = \lambda [\omega_{N, \partial N}] \ .
$$
Using the previous equality, it is immediate to check that the volume of $\sigma$ may be expressed as follows
\begin{align*}
\Vol(\sigma) &= \langle \comp^n \circ  \textup{J}^n \left[ \textup{C}^n(\Phi_X) (\Vol_n)\right], [N, \partial N] \rangle \\
&= \lambda \langle [\omega_{N, \partial N}], [N, \partial N] \rangle = \lambda \Vol(M) \ .
\end{align*}
As a consequence of Equation (\ref{equation:trans:lambda}), taking the absolute value on both sides, we get
$$
\frac{| \Vol(\sigma) |}{| \Vol(M) |}  = |\lambda | =\frac{\lVert \trans_\Gamma^n \left[\textup{C}^n(\Phi_X)(\Vol_n) \right]\lVert_\infty}{\lVert \omega^b_n \rVert_\infty} \ .
$$

Therefore, we reduce ourselves to prove that 
$$
\frac{\lVert \trans_\Gamma^n \left[ \textup{C}^n(\Phi_X)(\Vol_n) \right] \rVert_{\infty}}{\lVert \omega_n^b\rVert_{\infty}} \leq 1 \ ,
$$
Recall that $\lVert \omega_n^b \rVert_\infty = \lVert \Vol_n \rVert_\infty$ because there are no coboundaries in $\textup{H}_{cb}^n(G; \mathbb{R}_\varepsilon)$, as proved by Bucher, Burger and Iozzi~\cite[Proposition~2]{bucher2:articolo}. Moreover, note that $\trans_\Gamma^n$ is norm non-increasing by definition. Therefore, the claim follows if we show that
$$
\lVert \left[ \textup{C}^n(\Phi_X)(\Vol_n) \right] \rVert_{\infty} \leq \lVert \Vol_n \rVert_\infty \ .
$$
Note that, by the definition of $\lVert \cdot \rVert_\infty$, we have
$$
\lVert \left[ \textup{C}^n(\Phi_X)(\Vol_n) \right] \rVert_{\infty} \leq \lVert \textup{C}^n(\Phi_X)(\Vol_n) \rVert_\infty  \ .
$$
Since we have already proved in Lemma~\ref{lemma:composizione:norm:non:increasing} that $\textup{C}^n(\Phi_X)$ is norm non-increasing , we get
$$
\lVert \textup{C}^n(\Phi_X)(\Vol_n) \rVert_\infty \leq \lVert \Vol_n \rVert_\infty \ ,
$$
whence the thesis.
\end{proof}

Having introduced a Milnor-Wood type inequality, we are interested now in investigating \emph{maximal cocycles}, that are cocycles of maximal volume. By Corollary~\ref{cor:max:vol} and Remark~\ref{oss:commento:estensione:mostow}, we already know that all the cocycles which are cohomologous to the one associated to the standard lattice embedding via a measurable function with essentially constant sign are maximal. We spend the rest of this section in proving that in fact they are the only ones. This remarkable result provides our desired rigidity result for measurable cocycles.

To that end, we need to prove the formula reported in Proposition \ref{prop:formula:volume} which allows us to express the volume as a multiplicative constant. Our result will be a generalization of Bucher, Burger and Iozzi's formula~\cite[Theorem 1.2]{bucher2:articolo}. Additionally, in the uniform case, our approach can be interpreted as an alternative proof of the generalized version Bader, Furman and Sauer's formula~\cite[Theorem~4.8]{sauer:articolo} (see Section~\ref{subsec:nuova:uniforme}).

\begin{proof}[Proof of Proposition \ref{prop:formula:volume}]
As already proved in Proposition~\ref{prop:vol:estimate:wood}, we know that
$$\trans_\Gamma^n \left[ \textup{C}^n(\Phi_X) (\Vol_n) \right] = \frac{\Vol(\sigma)}{\Vol(M)} \omega_n^b \ ,$$
where $\omega_n^b = \left[ \Vol_n\right] \in \, \textup{H}^n_{cb}(G; \mathbb{R}_\varepsilon)$.
A priori this equality holds at the level of cohomology classes but, as already explained in Section~\ref{sec:cbc}, we may identify $\textup{H}^n_{cb}(G; \mathbb{R}_\varepsilon)$ with the space of bounded measurable cocycles on $\mathbb{S}^{n - 1}$. Hence, the above equality can be restated in terms of cocycles as follows:
$$
\trans_\Gamma^n \circ\textup{C}^n(\Phi_X) (\Vol_n) = \frac{\Vol(\sigma)}{\Vol(M)} \Vol_n \ .
$$
It is immediate to check that this provides precisely the desired formula.
\end{proof}

Before going into the details of our rigidity Theorem~\ref{mostow}, we  define now a family of cocycles with vanishing volume.

\begin{deft}\label{def:lementary:cocycle}
Let $n \geq 3$. Let $\Gamma < G^+ =\Isom^+(\mathbb{H}^n)$  be a torsion-free non-uniform lattice and let $(X, \mu_X)$ be a standard Borel probability $\Gamma$-space. Let $\sigma \colon \Gamma \times X \rightarrow G^+$ be a measurable cocycle. Let $k < n$, we say that the cocycle $\sigma$ is \emph{reducible} if it is cohomologous to a cocycle $\sigma_\textup{red}: \Gamma \times X \rightarrow \Isom^+(\hyp^k)$ through a measurable map $f:X \rightarrow \Isom(\hyp^n)$ with essentially constant sign. Here $\Isom^+(\hyp^k)$ is thought of as a subgroup of $\Isom^+(\hyp^n)$ via the upper left corner injection, that is
$$
\Isom^+(\hyp^k) \rightarrow \Isom^+(\hyp^n), \hspace{10pt} g \mapsto 
\left(
\begin{array}{cc}
g & 0\\
0 & \textup{Id}_{n-k}\\
\end{array}
\right) \ ,
$$
where $\textup{Id}_{n-k}$ is the identity matrix of order $(n-k)$. 
\end{deft}

In the following example we show that reducible cocycles have volume equal to zero.

\begin{es}\label{es:reducible:cocycles}
Let $\sigma$ be a reducible cocycle. By Proposition~\ref{gplus:cohomology}, without loss of generality we can work directly with a cocycle
$$
\sigma:\Gamma \times X \rightarrow \Isom^+(\hyp^k)<\Isom^+(\hyp^n) \ .
$$
If we now assume that $\sigma$ admits a measurable map $\phi:\bbS^{n-1} \times X \rightarrow \bbS^{n-1}$, it should be clear that this map admits actually as target the $(k-1)$-dimensional sphere $\bbS^{k-1}$ stabilized by $\Isom^+(\hyp^k)$. Hence for almost every $x \in X$ we have a map $\phi_x:\bbS^{n-1} \rightarrow \bbS^{k-1} \ , \phi_x(\xi):=\phi(\xi,x)$ whose image lies entirely in $\bbS^{k-1}$. Moreover, since $X$ is a standard Borel space, by~\cite[Lemma 2.6]{fisher:morris:whyte} the map $\phi_x$ is measurable for almost every $x \in X$. 

Since the $\phi_x$-image of any tetrahedron lies in $\bbS^{k-1}$, we have

$$\int_{\Gamma \backslash G} \int_X \varepsilon(\overline{g}^{-1}) \cdot \Vol_n(\phi_x( \overline{g} \cdot \xi_0),\cdots,\phi_x(\overline{g} \cdot \xi_n))d\mu_X(x) d\mu (\overline{g})= 0 \ .
$$
By applying Proposition \ref{prop:formula:volume}, this shows that $\Vol(\sigma) = 0$.
\end{es}

We are now ready to discuss the proof of our main rigidity Theorem~\ref{mostow}.

\begin{proof}[Proof of Theorem~\ref{mostow}]
Let $\sigma_i$ denote the cocycle associated to the standard lattice embedding. We have already proved in Corollary~\ref{cor:max:vol} and Remark~\ref{oss:commento:estensione:mostow} that cocycles cohomologous to $\sigma_i$ via a measurable function with essentially constant sign have maximal volume. It remains to prove the converse. 

Let $\sigma \colon \Gamma \times X \rightarrow G^+$ be a maximal measurable cocycle with associated  essentially unique boundary map $\phi: \bbS^{n - 1} \times X \rightarrow \bbS^{n -1 }$. Up to conjugacy by a suitable element $g \in G$, we may assume that the volume of $\sigma$ is positive (see Proposition~\ref{gplus:cohomology}).

Fix a positive regular ideal tetrahedron with vertices $\xi_0, \cdots, \xi_n \in \, \mathbb{S}^{n-1}$ and denote its volume by $\nu_n$.

Define the measurable map $\phi_x \colon \bbS^{n-1} \rightarrow \bbS^{n-1}$ by $\phi_x(\xi):=\phi(\xi, x)$. Notice that the hypothesis on $X$ to be a standard Borel space implies that the map $\phi_x$ is measurable for almost every $x \in X$, again by~\cite[Lemma 2.6]{fisher:morris:whyte}. 

Following~\cite[Proposition 5]{bucher2:articolo}, we know that the equality stated in Proposition~\ref{prop:formula:volume} actually holds for every $\xi_0, \cdots, \xi_n \in \, \mathbb{S}^{n-1}$. Hence, we get 
$$\int_{\Gamma \backslash G} \int_X \varepsilon(\overline{g}^{-1}) \cdot \Vol_n(\phi_x( \overline{g} \cdot \xi_0),\cdots,\phi_x(\overline{g} \cdot \xi_n))d\mu_X(x) d\mu (\overline{g})= \nu_n \ , \nonumber 
$$
where $\mu$ is the normalized probability measure on the quotient $\Gamma \backslash G$. Since the argument of the integral is bounded from above by $\nu_n$, the previous formula implies that  
$$
\varepsilon(\overline{g}^{-1}) \cdot \Vol_n(\phi_x( \overline{g} \cdot \xi_0),\cdots,\phi_x(\overline{g} \cdot \xi_n)) = \nu_n
$$
for almost every $\overline{g} \in\Gamma \backslash G$ and almost every $x \in X$. Since $\phi$ is $\sigma$-equivariant,  the previous formula holds in fact for almost every $g \in G$ and for almost every $x \in X$. Following~\cite[Proposition 6]{bucher2:articolo}, this implies that $\phi_x$ is almost everywhere equal to an orientation-preserving isometry $f(x) \in \, G^+$. We now conclude the proof as describe by Bader, Furman and Sauer~\cite[Proposition~3.2]{sauer:articolo}. For the convenience of the reader, we recall here the procedure for obtaining the desired conjugation. The previous construction allows us to define a map $f \colon X \rightarrow G^+$. Note that the map $\hat \phi: X \rightarrow \textup{Meas}(\bbS^{n-1},\bbS^{n-1}), \hspace{5pt} \hat \phi(x):=\phi_x$ is measurable and its image lies entirely in the isometry group $\po^\circ(n,1) \subset \textup{Meas}(\bbS^{n-1},\bbS^{n-1})$ (see also~\cite[Proposition 3.2]{sauer:articolo}). Since by assumption $X$ is a standard Borel space, the measurability of the map $f$ follows now by~\cite[Lemma 2.6]{fisher:morris:whyte}.

Therefore,  given $\gamma \in \, \Gamma$, on the one hand we have
\[
\phi(i(\gamma) \xi, \gamma x)=\sigma(\gamma,x)\phi(\xi,x)=\sigma(\gamma,x)f(x)(\xi),
\]
and on the other hand
\[
\phi(i(\gamma) \xi, \gamma x)=f(\gamma x)(i(\gamma)\xi).
\]
The previous computations show that  
\[
i(\gamma)=f(\gamma x)^{-1} \sigma(\gamma,x) f(x) \ ,
\]
whence $\sigma$ is cohomologous to the cocycle associated to the standard lattice embedding $i:\Gamma \rightarrow G^+$.

\end{proof}

We already know that among maximal cocycles we may find the ones associated to maximal representations. Since we already mentioned that all our results also hold in the uniform case, it is remarkable that in this situation we can describe other families of maximal cocycles. We anticipate here the family arising from ergodic \emph{integrable} self-couplings and the ones coming from ergodic couplings of $\Isom(\mathbb{H}^n)$. We postpone to Section~\ref{sec:max:cocycle:mapping:degree} the discussion of the ones arising as pullback of maximal cocycles along maps homotopic to local isometries (see Proposition~\ref{prop:local:iso}). 

We recall first the definition of \emph{integrable} self-coupling (see for instance~\cite{shalom:annals, sauer:companion, sauer:articolo} for a discussion on this property).

Let $\Gamma$ be torsion-free uniform lattice in $G^+$ and let $\ell \colon \Gamma \rightarrow \mathbb{N}$ be the length function associated to some word-metric on $\Gamma$.  Keeping the notation of Section~\ref{sec:representations:coupling:volume}, consider a $(\Gamma_r, \Gamma_\ell)$-coupling $(\Omega, m_\Omega)$ with associated right measure equivalence cocycle $\alpha_\Omega \colon \Gamma_r \times \Gamma_\ell \backslash \Omega \rightarrow \Gamma_\ell$. We say that $(\Omega, m_\Omega)$ is an \emph{integrable} $(\Gamma_r, \Gamma_\ell)$-\emph{coupling} if for every $\gamma \in \, \Gamma$
$$
\int_{\Gamma_\ell \backslash \Omega} \ell(\alpha_\Omega(\gamma, x)) dm_{\Gamma_\ell \backslash \Omega}(x) < +\infty \ ,
$$
where $\Gamma_\ell \backslash \Omega$ is the normalized probability measure on $\Gamma_r \backslash \Omega$. We need this definition in order to apply a resuly by Bader, Furman and Sauer~\cite[Corollary~4.12]{sauer:articolo}.

As describe by Shalom~\cite[Theorem~3.6]{shalom:annals} (see also~\cite[Remark~5.5]{sauer:companion}), if $n \geq 3$, the group $\Isom(\mathbb{H}^n)$ endowed with its Haar measure $m_{\mathcal{H}}$ is an integrable $(\Gamma_r, \Gamma_\ell)$-coupling for any of its uniform lattices.

We are now ready to show that ergodic integrable self-couplings are maximal (compare with~\cite[Corollary~4.12]{sauer:articolo}).

\begin{cor}\label{cor:value:coupling}
Assume $n \geq 3$. Let $\Gamma < \textup{Isom}^+(\mathbb{H}^n)$ be a torsion-free uniform lattice and set $M=\Gamma \backslash \mathbb{H}^n$. Let $(\Omega,m_\Omega)$ be an ergodic integrable $(\Gamma_r,\Gamma_\ell)$-coupling with associated right measure equivalence cocycle $\alpha_\Omega:\Gamma_r \times \Gamma_\ell \backslash \Omega \rightarrow \Gamma_\ell$. Then it holds 
$$
|\Vol(\alpha_\Omega)| = \Vol(M) \ . 
$$
\end{cor}
\begin{proof}
Since we know by~\cite[Corollary~4.12]{sauer:articolo} that integrable self couplings satisfy $\eu(\Omega) = \pm \Vol(M)$, by Section~\ref{subsec:nuova:uniforme} we get 
$$
|\Vol(\alpha_\Omega)| = |\eu(\Omega)|=\Vol(M) \ ,
$$
and this concludes the proof.
\end{proof}

\begin{oss}
In the uniform case, the previous corollary together with Theorem~\ref{mostow} show that measurable cocycles associated to ergodic integrable self-couplings are conjugated with the ones associated to maximal representations. 
\end{oss}

\section{Maximal cocycles and mapping degree}\label{sec:max:cocycle:mapping:degree}

Let $M_1$ and $M_2$ be two closed hyperbolic manifold of the same dimension $n \geq 3$. We know by \cite[Theorem 6.4]{Thurston} (compare also with \cite[Corollary 1.3]{bucher2:articolo}), that given a continuous map $f \colon M_1 \rightarrow M_2$, we have 
$$
| \deg(f) | \leq \frac{\Vol(M_1)}{\Vol(M_2)}
$$
and the equality holds if and only if $f$ is homotopic to a local isometry.

We show here a result that characterizes maps homotopic to local isometries via the previous theorem using maximal cocycles (i.e. cocycles of maximal volume).

Let $f \colon M_1 \rightarrow M_2$ be a continuous map. Denote by $\Gamma_1$ and $\Gamma_2$ the fundamental groups of $M_1$ and $M_2$, respectively. Let $\pi_1(f) \colon \Gamma_1 \rightarrow \Gamma_2$ be the induced map by $f$ at the level of fundamental groups. Let us consider a measurable cocycle $\sigma \colon \Gamma_2 \times X \rightarrow G^+$, where $(X, \mu_X)$ is a standard Borel probability $\Gamma_2$-space. Note that $(X, \mu_X)$ can be also viewed as a standard Borel probability $\Gamma_1$-space via the action induced by $\pi_1(f)$. We define the \emph{pullback cocycle along $f$} as
$$
f^*\sigma \colon \Gamma_1 \times X \rightarrow G^+ , \hspace{10pt} f^*\sigma(\gamma, x) = \sigma(\pi_1(f)(\gamma), x) \ ,
$$
for all $\gamma \in \, \Gamma_1$ and $x \in \, X$.
\begin{lem}
The map $f^*\sigma$ is a measurable cocycle.
\end{lem}
\begin{proof}
Let $\gamma_1, \gamma_2 \in \, \Gamma_1$ and $x \in \, X$. Then, we have
\begin{align*}
f^*\sigma(\gamma_1 \gamma_2, x) &= \sigma(\pi_1(f)(\gamma_1 \gamma_2), x) \\
&=\sigma(\pi_1(f)(\gamma_1) \pi_1(f)(\gamma_2), x) \\
&= \sigma(\pi1(f)(\gamma_1), \pi_1(f)(\gamma_2) \cdot x) \sigma(\pi_1(f)(\gamma_2), x) \\
&= f^*\sigma(\gamma_1, \pi_1(f)(\gamma_2) \cdot x) f^*\sigma(\gamma_2, x) \ .
\end{align*}
The proof of the measurability and of the continuity of the map
$$
\Gamma_1 \rightarrow \textup{Meas}(X,G^+), \hspace{5pt} \gamma \mapsto \sigma(\pi_1(f)(\gamma),\cdot)
$$
are straightforward and we leave their proof to the reader.
\end{proof}

Given a continuous map $f:M_1 \rightarrow M_2$ of non-vanishing degree, thanks to the work of \cite{burger:mozes,franc09} there exists an essentially unique measurable map $\widetilde{f} \colon \partial \mathbb{H}^n \rightarrow \partial \mathbb{H}^n$ which is $\pi_1(f)$-equivariant. Now if we assume that $\sigma$ admits a boundary map $\phi \colon \mathbb{S}^{n-1} \times X \rightarrow \mathbb{S}^{n-1}$, the pullback cocycle along $f$ admits the following boundary map:
$$
f^*\phi \colon \mathbb{S}^{n-1} \times X \rightarrow \mathbb{S}^{n-1} , \hspace{10pt} f^*\phi(\xi,x) \coloneqq \phi(\widetilde{f}(\xi), x)  \ ,
$$
for all $\xi \in \, \mathbb{S}^{n-1}$ and $x \in \, X$.

Having introduced all the elements appearing in Proposition~\ref{prop:moltiplicativita}, we are now ready to prove it.

\begin{proof}[Proof of Proposition \ref{prop:moltiplicativita}]
Since $\sigma$ is maximal, by the proof of Theorem \ref{mostow} we know that the slices of the associated boundary map $\phi:\mathbb{S}^{n-1} \times X \rightarrow \mathbb{S}^{n-1}$ are isometries. More precisely $\phi_x:\mathbb{S}^{n-1} \rightarrow \mathbb{S}^{n-1}, \hspace{5pt} \phi_x(\xi):=\phi(\xi,x)$ is an isometry for almost every $x \in X$. Hence there exists a measurable map $F:X \rightarrow G^+$ such that 
$$
\phi_x(\xi)=F(x)(\xi) \ ,
$$
for almost every $\xi \in \mathbb{S}^{n-1}$. If we now consider the volume cocycle associated to $\sigma$, for almost every $\xi_0,\ldots,\xi_n \in \mathbb{S}^{n-1}$ we have that 
\begin{align*}
\textup{C}^n(\Phi_X)(\textup{Vol}_n)(\xi_0,\ldots,\xi_n)&=\int_X \textup{Vol}_n(\phi_x(\xi_0),\ldots,\phi_x(\xi_n))d\mu_X(x)=\\
&=\int_X \textup{Vol}_n(F(x)(\xi_0),\ldots,F(x)(\xi_n))d\mu_X(x)=\\
&=\int_X \textup{Vol}_n(\xi_0,\ldots,\xi_n)d\mu_X(x)=\textup{Vol}_n(\xi_0,\ldots,\xi_n) \ ,
\end{align*}
or equivalently it holds
\begin{equation}\label{equation:max:vol}
\textup{C}^n(\Phi_X)(\textup{Vol}_n)=\textup{Vol}_n \ .
\end{equation}
We have the following chain of equalities:

\begin{align}\label{equation:degree:estimate}
\textup{Vol}(f^\ast \sigma)&=\langle \textup{comp}^n \circ g^n_{M_1} [\textup{C}^n(f^\ast \Phi_X)(\textup{Vol}_n)], [M_1] \rangle= \\
&=\langle \textup{comp}^n \circ g^n_{M_1} [\textup{C}^n(\widetilde{f}) \circ \textup{C}^n(\Phi_X)(\textup{Vol}_n)], [M_1] \rangle= \nonumber\\
&=\langle \textup{comp}^n \circ g^n_{M_1} [\textup{C}^n(\widetilde{f})(\textup{Vol}_n)], [M_1] \rangle= \nonumber\\
&=\langle \textup{comp}^n \circ g^n_{M_1} \circ \textup{H}^n_{b}(\pi_1(f))[\textup{Vol}_n], [M_1] \rangle= \nonumber\\
&=\langle \textup{H}^n(f) \circ \textup{comp}^n \circ g^n_{M_1}[\textup{Vol}_n],[M_1] \rangle= \nonumber\\
&=\langle \textup{comp}^n \circ g^n_{M_1}[\textup{Vol}_n],\textup{H}_n(f)[M_1] \rangle= \nonumber\\
&=\langle \textup{comp}^n \circ g^n_{M_1}[\textup{Vol}_n],\textup{deg}(f) \cdot [M_2] \rangle=\textup{deg}(f)\textup{Vol}(M_2) \nonumber \ ,
\end{align}
where we used \cite{burger:articolo} to implement the class $\textup{H}^n_{b}(\pi_1(f))[\textup{Vol}_n]$ using the preferred representative $\textup{C}^n(\widetilde{f})(\textup{Vol}_n)$ to pass from the third to the fourth line. Additionally we exploited the fact that comparison maps commute with the maps induced by representations to move from the fourth line to the fifth. This concludes the proof.
\end{proof}

The previous proposition easily implies the mapping degree theorem, as shown in the following:

\begin{cor}\label{cor:mapping:degree} 
Let $f \colon M_1 \rightarrow M_2$ be a continuous map with $\deg(f) \neq 0$ between closed hyperbolic manifolds of the same dimension $n \geq 3$. Then,
$$
|\deg(f)| \leq \frac{\Vol(M_1)}{\Vol(M_2)}
$$
\end{cor}

\begin{proof}
Let $n \geq 3$. By Proposition~\ref{prop:moltiplicativita}, on the one hand we know that given a maximal cocycle $\sigma \colon \Gamma_2 \times X \rightarrow \po^\circ(n, 1)$, we get
$$
\Vol(f^*\sigma) = \deg(f) \cdot \Vol(M_2) \ .
$$
On the other, by Proposition~\ref{prop:vol:estimate:wood}  we have that
$$
\Vol(M_2) \cdot |\deg(f)| = |\Vol(f^*\sigma)| \leq \Vol(M_1) \ ,
$$
whence the thesis.
\end{proof}

Another remarkable application of Proposition \ref{prop:moltiplicativita} is the possibility to use the language of maximal cocycles in order to characterize continuous maps between closed hyperbolic manifolds that are homotopic to local isometries. This is the content of Proposition \ref{prop:local:iso}, whose proof is reported below. 

\begin{proof}[Proof of Proposition \ref{prop:local:iso}]
Suppose that $f$ is homotopic to a local isometry. On the one hand, Thurston's strict version of the mapping degree theorem~\cite[Theorem~6.4]{Thurston} implies that 
$$
\Vol(M_1) = |\textup{deg}(f)| \cdot \Vol(M_2) \ .
$$
On the other hand, Proposition~\ref{prop:moltiplicativita} implies
$$
|\Vol(f^*\sigma)| = |\deg(f)| \cdot \Vol(M_2) \ .
$$
This shows that $f^*\sigma$ is maximal.

Suppose now the converse, that is $f^*\sigma$ is maximal. Then, Proposition~\ref{prop:moltiplicativita} implies that
$$
\Vol(M_1) = |\Vol(f^*\sigma)| = |\deg(f)| \cdot \Vol(M_2) \ .
$$
The thesis now follows as a consequence of  \cite[Theorem 6.4]{Thurston}.
\end{proof}

\section{Euler number of Zimmer's cocycles of surface groups}\label{sec:Euler:totale}

Let $\Sigma_g$ a closed hyperbolic surface of genus $g \geq 2$ and denote by $\Gamma_g:=\pi_1(\Sigma_g)$ its fundamental group. Fix a standard Borel probability $\Gamma_g$-space $(X,\mu_X)$. In this section we want to define the Euler number associated to a measurable cocycle~$\sigma: \Gamma_g \times X \rightarrow \homs$ and study its rigidity property. This provides an extension of the study of maximal representations to the settings of measurable cocycles. 

Since $\Gamma_g$ is a uniform lattice, we have already discussed in Section~\ref{subsec:nuova:uniforme} how our approach leads to the generalized version of Bader-Furman-Sauer's Euler number. Here, we show in Remark~\ref{oss:nuova:uniforme} that our Euler number of measurable cocycles differs from the latter by a multiplicative constant. However, as mentioned above this section is mainly devoted to the extension of the Euler number of representations. Remarkably, the techniques that we develop here allows us to provide alternative proofs of the \emph{generalized} version of some results by Bader, Furman and Sauer~\cite{sauer:articolo} for measurable cocycles.

Assume from now on that  $\sigma$ admits an essentially unique boundary map $\phi: \bbS^1 \times X \rightarrow \bbS^1$. It is important to underline that a priori there is no well-defined action of $\Gamma_g$ on the circle $\bbS^1$, hence we first need to fix a hyperbolization $\pi_0:\Gamma_g \rightarrow \textup{PSL}(2,\R)$ to specify this action. We will assume that $\homs$ has the discrete topology. Note that $\bbS^1$ is a compact completely metrizable space on which $\homs$ acts in a measurable way. This means that for all along the section we consider generalized boundary maps as described in Remark \ref{oss:boundary:genmap}.

\subsection{Definition of Euler number}\label{sec:def:eulero:main:prop}

Recall by Section~\ref{sec:euler:class} that the Euler cocycle $$\epsilon=-\orien/2 \in \mathcal{B}^\infty((\bbS^1)^3;\R)^{\homs}$$ naturally determine a cohomology class in $$[\textup{C}^2(\Phi_X)(\epsilon)] \in \uph_b^2(\Gamma_g;\R) \ ,$$ and we call it \emph{the Euler class associated to $\sigma$.} Let $[\Sigma_g] \in \uph_2(\Sigma_g,\R)$ be the fundamental class of $\Sigma_g$. Similarly to what we have done in the previous section, we denote by
\[
\comp^2:\uph_b^2(\Sigma_g;\R) \rightarrow \uph^2(\Sigma_g;\R)
\]
the comparison map (which is non-trivial since $\Sigma_g$ is compact). We are now ready to define the Euler number associated to a cocycle.

\begin{deft}\label{euler} 
Let $\Sigma_g$ be a closed surface of genus $g \geq 2$ and let $\Gamma_g=\pi_1(\Sigma_g)$. Let $(X,\mu_X)$ be a standard Borel probability $\Gamma_g$-space. Fix a hyperbolization $\pi_0:\Gamma_g \rightarrow \textup{PSL}(2,\R)$ and assume that $\Gamma_g$ acts on $\bbS^1$ via $\pi_0$. Consider a cocycle $\sigma:\Gamma_g \times X \rightarrow \homs$ with essentially unique boundary map $\phi:\bbS^1 \times X \rightarrow \bbS^1$. The \textit{Euler number} $\eu(\sigma)$ associated to the cocycle $\sigma$ is given by
\[
\eu(\sigma):=\langle \comp^2 \circ g_{\Sigma_g}^2 [\textup{C}^2(\Phi_X)(\epsilon)], [\Sigma_g] \rangle \ ,
\]
where $\langle \cdot,\cdot \rangle$ is the Kronecker pairing and $\comp^2, g_{\Sigma_g}^2$  denote the comparison map and the isomorphism appearing in Gromov's mapping theorem, respectively. 
\end{deft}

\begin{oss}
As already noticed in Remark \ref{oss:existence:boundary:map}, also in this case it is sufficient to suppose that $\sigma$ is \emph{non-elementary} to obtain the existence of an essentially unique boundary map by \cite[Proposition 3.3]{MonShal0}.  
\end{oss}

\begin{oss}
As discussed in Remark~\ref{oss:uniforme:diagramma}, recall that the comparison map also appears in our definition of volume of cocycles defined in terms of \emph{uniform} lattices. For the same reason, we are forced to introduce the comparison map also when we define the Euler number of measurable cocycles.
\end{oss}

\begin{oss}\label{oss:nuova:uniforme}
We have already discussed in Section~\ref{subsec:nuova:uniforme} how the results proved by Bader, Furman and Sauer~\cite{sauer:articolo} about $(\Gamma, \Gamma)$-couplings can be  extended to their generalized versions (i.e. rephrased for arbitrarily measurable cocycles). Here, we show how our Euler number of measurable cocycle differs by a multiplicative constant from the \emph{generalized} version of Bader-Furman-Sauer's Euler number. Via this remark, we will see in the sequel how some of our results provide alternative proofs of the generalized version of some results by Bader, Furman and Sauer~\cite{sauer:articolo}.

Let $\Sigma_g$ be a closed surface of genus $g \geq 2$ and let $\Gamma_g=\pi_1(\Sigma_g)$. Fix a hyperbolization $\pi_0:\Gamma_g \rightarrow \textup{PSL}(2,\R)$ and suppose that $\Gamma_g$ acts on $\bbS^1$ via $\pi_0$. Let $(\Omega,m_\Omega)$ be a $(\Gamma_g,\Gamma_g)$-coupling with associated right measure equivalence cocycle $\alpha_\Omega:\Gamma_g \times \Gamma_g \backslash \Omega \rightarrow \Gamma_g$. Then it holds
$$
\eu(\alpha_\Omega)=-\frac{\eu(\Omega)}{2 \pi} \ .
$$
Indeed, Bader, Furman and Sauer~\cite{sauer:articolo} consider the pullback of the Volume cocycle rather that of the Euler one. These cocycles are related by the following equation
$$
-2\pi \epsilon = \Vol_2 \ .
$$
Moreover, by standard result about Furstenberg boundaries theory, there exists a boundary map $\phi: \bbS^1 \times \Gamma_g \backslash \Omega \rightarrow \bbS^1$ associated to $\alpha_\Omega$. As showed in Section~\ref{subsec:nuova:uniforme}, by a suitable modification of the commutative Diagram (\ref{diagr:numero:euler:bfs}) to this context, we get

\begin{align*}
\eu(\Omega) &= \langle \comp^2 \circ g_{\Sigma_g}^2 \circ \textup{I}^2_{\Gamma_g \backslash \Omega} \circ \textup{H}^2_b(\Omega) \circ \textup{H}^2_b(\kappa) \circ \textup{H}^2_b(i)([\Vol_2]) , [\Sigma_g] \rangle \\
&=\langle \comp^2 \circ g_{\Sigma_g}^2 \left[\textup{C}^2(\Phi_{\Gamma_g \backslash \Omega})(\Vol_2)\right], [\Sigma_g] \rangle \\
&=\langle \comp^2 \circ g_{\Sigma_g}^2 \left[\textup{C}^2(\Phi_{\Gamma_g \backslash \Omega})(-2\pi \epsilon) \right], [\Sigma_g] \rangle=-2 \pi \eu(\alpha_\Omega) \ .
\end{align*}
\end{oss}

\subsection{Euler number of cocycles vs. Euler number of representations}\label{sec:eul:coc:vs:repre}

Recall by Definition~\ref{cocyclerep} that, given any standard Borel probability $\Gamma_g$-space $(X,\mu_X)$, every representation $\rho:\Gamma_g \rightarrow \homs$ induces naturally a cocycle $\sigma_\rho:\Gamma_g \times X \rightarrow \homs$. 

We are going to prove, under an additional hypothesis on $\rho$, that the Euler number associated to $\sigma_\rho$ coincides with the classic Euler number associated to $\rho$. This shows that Definition~\ref{euler} extends the classic notion of Euler number of representations to the wider theory of Zimmer's cocycles.

\begin{prop}\label{eu:representations} 
Let $\Sigma_g$ be a closed surface of genus $g \geq 2$ and let $\Gamma_g=\pi_1(\Sigma_g)$. Let $(X,\mu_X)$ be a standard Borel probability $\Gamma_g$-space. Fix a hyperbolization $\pi_0:\Gamma_g \rightarrow \textup{PSL}(2,\R)$ and assume that $\Gamma_g$ acts on $\bbS^1$ via $\pi_0$. Consider a representation $\rho:\Gamma_g \rightarrow \homs$ with associated cocycle $\sigma_\rho$. Assume that $\rho$ admits an essentially unique measurable map $\varphi: \bbS^1 \rightarrow \bbS^1$ which is  equivariant with respect to the actions determined by $\pi_0$ and $\rho$, respectively. Then, we have
\[
\eu(\sigma_\rho)=\eu(\rho) \ .
\]
\end{prop}

\begin{proof}
Thanks to the assumption about the existence of the essentially unique map $\varphi$, we can define a $\sigma_\rho$-equivariant boundary map $\phi: \bbS^1 \times X \rightarrow \bbS^1$ of $\sigma_\rho$ as follows: 
\[
\phi:\bbS^1 \times X \rightarrow \bbS^1, \hspace{10pt} \phi(\xi,x):=\varphi(\xi) \,
\]
for almost every $\xi \in \bbS^1$ and $x \in X$ (see Section~\ref{sec:zimmer:cocycle}). Since the boundary map $\phi$ actually does not depend on the second variable $x \in X$, it is immediate to check that the two following pullback maps agree $$\textup{C}^2(\Phi_X)(\epsilon)=\textup{C}^2(\varphi)(\epsilon) \ .$$

Recalling that the pullback $\uph_b^2(\rho)(\text{e}_b)$ of the Euler class along $\rho$ admits as a representative $\textup{C}^2(\varphi)(\epsilon)$ as described by Burger and Iozzi~\cite{burger:articolo}, we get that 

\begin{align*}
\eu(\sigma_\rho)&=\langle \comp^2 \circ g_{\Sigma_g}^2 \left[\textup{C}^2(\Phi_X)(\epsilon)\right], [\Sigma_g] \rangle=\\
	               &=\langle \comp^2 \circ g_{\Sigma_g}^2 \left[\textup{C}^2(\varphi)(\epsilon)\right],[\Sigma_g] \rangle=\\
		    &=\langle \comp^2 \circ g_{\Sigma_g}^2 \circ \uph_b^2(\rho)(\text{e}_b),[\Sigma_g] \rangle=\eu(\rho) \ ,
\end{align*}
and the statement now follows. 
\end{proof}

It is a standard fact that the Euler number is constant along the semiconjugacy class of a representation $\rho:\Gamma_g \rightarrow \homs$ (see for instance~\cite{iozzi02:articolo}). We show now that a similar result still holds in the more general theory of Zimmer's cocycles. More precisely, the following proposition proves that the Euler number is constant on the cohomology class of a cocycle $\sigma:\Gamma_g \times X \rightarrow \homs$. 

\begin{prop}\label{homs:cohomology}
Let $\Sigma_g$ be a closed surface of genus $g \geq 2$ and let $\Gamma_g=\pi_1(\Sigma_g)$. Let $(X,\mu_X)$ be a standard Borel probability $\Gamma_g$-space. Fix a hyperbolization $\pi_0:\Gamma_g \rightarrow \textup{PSL}(2,\R)$ and assume that $\Gamma_g$ acts on $\bbS^1$ via $\pi_0$. Consider a cocycle $\sigma:\Gamma_g \times X \rightarrow \homs$ with essentially unique boundary map $\phi:\bbS^1 \times X \rightarrow \bbS^1$. Then we have
\[
\eu(\sigma^f)=\eu(\sigma) \ ,
\]
for every measurable map $f:X \rightarrow \homs$. In particular $\eu$ is constant along the $\homs$-cohomology classes. 
\end{prop}

\begin{proof}
We know by Definitions~\ref{def:coomologhi:cicli} and~\ref{boundary:map:cohomology} that the twisted cocycle $$\sigma^f(\gamma,x):=f(\gamma x)^{-1} \sigma(\gamma,x) f(x)$$ admits an essentially unique boundary map given by
\[
\phi^f:\bbS^1 \times X \rightarrow \homs, \hspace{5pt} \phi^f(\xi,x):=f(x)^{-1}\phi(\xi,x) \ .
\]
Since the cocycle $\epsilon$ is $\homs$-invariant, the same strategy followed in the proof of Proposition~\ref{gplus:cohomology} shows that $\textup{C}^2(\Phi^f_X)(\epsilon)=\textup{C}^2(\Phi_X)(\epsilon)$. This provides the desired equality
$$
\eu(\sigma^f) =\langle \comp^2 \circ g_{\Sigma_g}^2 \left[\textup{C}^2(\Phi^f_X)(\epsilon)\right],[\Sigma_g] \rangle = \langle \comp^2 \circ g_{\Sigma_g}^2 [\textup{C}^2(\Phi_X)(\epsilon)],[\Sigma_g] \rangle=\eu(\sigma) \ .
$$
\end{proof}

We conclude this paragraph with a nice application of both Proposition~\ref{eu:representations} and Proposition~\ref{homs:cohomology}:
\begin{cor}\label{eu:hyperbolization}
Let $\Sigma_g$ be a closed surface of genus $g \geq 2$ and let $\Gamma_g=\pi_1(\Sigma_g)$. Consider a hyperbolization $\pi_0:\Gamma_g \rightarrow \textup{PSL}(2,\R)$ and with associated cocycle $\sigma_{\pi_0}$. Then, we have
\[
\eu(\sigma)=\chi(\Sigma_g)
\]
for every cocycle $\sigma$ lying in the cohomology class of $\sigma_{\pi_0}$. 
\end{cor}

\begin{proof}
We first note that we can consider the identity $\textup{id}_{\bbS^1}:\bbS^1 \rightarrow \bbS^1$ as a measurable equivariant map with respect to the action of the hyperbolization $\pi_0$ on both the domain and the target space. Then, recall that if $\pi_0:\Gamma_g \rightarrow \textup{PSL}(2,\R)$ is a hyperbolization then $\eu(\pi_0) = \chi(\Sigma_g)$, as shown by Iozzi~\cite{iozzi02:articolo}. The result now follows from Propositions~\ref{eu:representations} and~\ref{homs:cohomology}. 
\end{proof}

\subsection{Rigidity and Euler number}\label{sec:euler:rigidity}

Given a closed surface $\Sigma_g$, we proved in Corollary~\ref{eu:hyperbolization} that the Euler number of a cocycle associated to a hyperbolization is equal to $\chi(\Sigma_g)$. In fact, a stronger result holds: the absolute value of the  Euler number of any cocycle defined in terms of $\pi_1(\Sigma_g)$ is always bounded by $|\chi(\Sigma_g)|$. We formalize this result in the following proposition, which can be interpreted as a generalized Milnor-Wood inequality (compare with~\cite{milnor:articolo,wood}). As explained in Remark~\ref{oss:nuova:uniforme}, up to a multiplicative constant this result provides an alternative proof of the generalized version of a result by Bader, Furman and Sauer~\cite[Corollary~4.9]{sauer:articolo}.

\begin{prop}\label{Milnor:wood:style}
Let $\Sigma_g$ be a closed hyperbolic surface of genus $g \geq 2$ and let $\Gamma_g=\pi_1(\Sigma_g)$. Let $(X,\mu_X)$ be a standard Borel probability $\Gamma_g$-space. Fix a hyperbolization $\pi_0:\Gamma_g \rightarrow \textup{PSL}(2,\R)$ and assume that $\Gamma_g$ acts on $\bbS^1$ via $\pi_0$. For every cocycle $\sigma:\Gamma_g \times X \rightarrow \homs$ with essentially unique boundary map $\phi:\bbS^1 \times X \rightarrow \bbS^1$, we have
\[
|\eu(\sigma)| \leq |\chi(\Sigma_g)| \ .
\]
\end{prop}

\begin{proof}
Using the same notation of Section~\ref{sec:trans} and suitably adapting Diagram (\ref{eq:diagramma:commutativo:trasfer:maps}), we can produce the following commutative diagram
\begin{equation}\label{diagr:miln:wood:eulero}
\xymatrix{
\uph^2((\mathcal{B}^\infty(\bbS^1)^{\bullet+1},\R)^{\homs}) \ar[d]_-{\uph^2(\Phi_X)} \\
\uph^2_b(\Gamma_g;\R)  \ar[r]^{\hspace{-15pt}\textup{trans}^2} \ar[d]_-{g_{\Sigma_g}^2} & \hcb^2(\textup{PSL}(2,\R);\R) \ar[dd]^{\comp^2}\\
\uph^2_b(\Sigma_g;\R) \ar[d]_-{\comp^2} \\
\uph^2(\Sigma_g;\R) \cong \uph_{\textup{dR}}^2(\Sigma_g;\R) \ar[r]_-{\tau_{\textup{dR}}^2} & \uph^2_c(\textup{PSL}(2,\R);\R) \cong \Omega^2(\hyp^2)^{\textup{PSL}(2,\R)} \ .
}
\end{equation}
This diagram together with Gauss-Bonnet theorem will be the main ingredients in the proof. Recall that $\epsilon=-\orien \slash 2$ is the cocycle representing the Euler class on $\bbS^1$. Since the cohomology group $\hcb^2(\textup{PSL}(2,\R);\R)$ is one dimensional and generated by the Euler class $\textup{e}_b$, as stated by Iozzi~\cite{iozzi02:articolo}, we have
\[
\textup{trans}^2 [\textup{C}^2(\Phi_X)(\epsilon)]=\lambda \textup{e}_b \ .
\]

We want now to determine this constant explicitly. By applying to both sides of the equation the comparison map we get
\[
\comp^2 \circ \textup{trans}^2 [\textup{C}^2(\Phi_X)(\epsilon)]=\comp^2 (\lambda \textup{e}_b)=\lambda \comp^2(\textup{e}_b)= \lambda \textup{e} \ ,
\]
where $\textup{e} \in \uph_c^2(\textup{PSL}(2,\R);\R)$ is the Euler class seen in the continuous cohomology group. By the Van Est isomorphism we know that each continuous class corresponds bijectively to a $\textup{PSL}(2,\R)$-invariant differential form on the associated symmetric space, that means in this case the hyperbolic plane $\hyp^2$. By construction the class of the orientation cocycle $\orien$ is mapped to the standard volume form on $\hyp^2$ normalized by $\pi$, that is $\omega/\pi$. This means that the Euler class $\textup{e}=[-\orien/2]$ will be represented by $-\omega/2\pi$.

Now the volume form $\omega$ defines a natural volume form $\omega_0 $ on the surface $\Sigma_g$ endowed with the hyperbolic structure determined by the hyperbolization $\pi_0$, that is $\pi_0(\Gamma_g) \backslash \hyp^2$. Moreover the transfer map $\tau_{\textup{dR}}$ is injective in degree $2$ and maps $\omega_0$ in $\omega$ (see Remark~\ref{tau:iso:grado:massimo}). 

By the commutativity of Diagram (\ref{diagr:miln:wood:eulero}), we have 
\[
\tau_{\textup{dR}}^2 (\comp^2 \circ g_{\Sigma_g}^2 [\textup{C}^2(\Phi_X)( \epsilon)] = \comp^2 \circ \textup{trans}^2[\textup{C}^2(\Phi_X)(\epsilon)]=\lambda \textup{e}=\lambda \tau_{\textup{dR}}^2 \left( \left[-\frac{\omega_0}{2\pi}\right] \right)\ ,
\]
and by injectivity of the transfer map we get 
\[
\comp^2 \circ g_{\Sigma_g}^2 \left[\textup{C}^2(\Phi_X)(\epsilon)\right]=\lambda \left[ -\frac{\omega_0}{2 \pi} \right] \ .
\]

If we now we evaluate both sides on the fundamental class $[\Sigma_g] \in \uph^2(\Sigma_g;\R)$ we obtain
\begin{align*}
\eu(\sigma)&=\langle \comp^2 \circ g_{\Sigma_g}^2 \left[\textup{C}^2(\Phi_X)(\epsilon)\right], [\Sigma_g] \rangle=\\
	        &=\langle \lambda \left[-\frac{\omega_0}{2 \pi}\right], [\Sigma_g] \rangle\\ &=-\frac{\lambda}{2\pi} \int_{\Sigma_g}\omega_0\\
                   &= -\frac{\lambda}{2\pi} (4\pi(g-1))=\lambda \chi(\Sigma_g) \ ,\\ 
\end{align*}
where we used the Gauss-Bonnet theorem to pass from the third to the fourth line. In particular we get that 
\[
\lambda=\frac{\eu(\sigma)}{\chi(\Sigma_g)} \ . 
\]
If we now take the absolute value on both sides, using Lemma~\ref{lemma:composizione:norm:non:increasing}, it follows that 
\[
|\lambda|=\frac{|| \trans^2 [\textup{C}^2(\Phi_X)(\epsilon)]||_\infty}{||\textup{e}_b||_\infty} \leq 1 \ ,
\]
by an argument similar to the one exposed in Proposition~\ref{prop:vol:estimate:wood}. Note that $||\textup{e}_b||_\infty = ||\epsilon||_\infty$ because of the double ergodicity of the action of $\Gamma_g$ induced by $\pi_0$ and the fact that $\epsilon$ is alternating. Indeed, double ergodicity implies that essentially bounded $\Gamma$-invariant functions $f \colon (\mathbb{S}^1)^2 \rightarrow \mathbb{R}$ are constant and so if we compute bounded cohomology via the subresolution of alternating functions, as explained in Remark~\ref{oss:alternating}, we have that $\textup{L}^\infty_{\textup{alt}}((\mathbb{S}^1)^2;\mathbb{R})^\Gamma=0$.

\end{proof}

We move forward to reach the proof of the generalized Matsumoto theorem for cocycles. In order to do this, we need first to prove Proposition \ref{prop:proportionality:eulero} where we express the Euler number as a multiplicative constant between cocycles. The result we are going to prove will generalize~\cite[Proposition 1.7]{iozzi02:articolo} (compare this result with the generalized version of Bader, Furman and Sauer's formula~\cite[Lemma~4.10]{sauer:articolo}).

\begin{proof}[Proof of Proposition \ref{prop:proportionality:eulero}]
We have already showed in the proof of Proposition~\ref{Milnor:wood:style} that $$\textup{trans}^2 [\textup{C}^2(\Phi_X)(\epsilon)]=\frac{\eu(\sigma)}{\chi(\Sigma_g)}\textup{e}_b,$$ and hence by linearity we argue that $$\textup{trans}^2 [\textup{C}^2(\Phi_X)(\orien)]=\frac{\eu(\sigma)}{\chi(\Sigma_g)}[\orien].$$ Since there are no essentially bounded $\textup{PSL}(2,\R)$-invariant cocycles on $\bbS^1$ by the doubly ergodic action of $\Gamma_g$ induced by $\pi_0$ and the fact that the orientation cocycle is alternating, the previous equality holds actually at the level of cocycles, that means $$\trans^2 \circ \textup{C}^2(\Phi_X)(\orien)=\frac{\eu(\sigma)}{\chi(\Sigma_g)}\orien,$$ and the statement follows.  
\end{proof}

Thanks to Proposition~\ref{prop:proportionality:eulero}, we are now ready to prove Theorem~\ref{matsumoto}.

\begin{proof}[Proof of Theorem~\ref{matsumoto}]
By Corollary~\ref{eu:hyperbolization} we already know that the cohomology class of a cocycle associated to a hyperbolization has maximal Euler number. We are going to prove the converse. 
Let $\sigma: \Gamma_g \times X \rightarrow \homs$ be a cocycle with essentially unique boundary map $\phi: \bbS^1 \times X \rightarrow \bbS^1$. Assume $|\eu(\sigma)|=\chi(\Sigma_g)$. Up to composing with a suitable homeomorphism which reverses the orientation of $\bbS^1$ we can assume that the equation holds even dropping the absolute value. Consider now a triple of points $(\xi_0,\xi_1,\xi_2) \in (\bbS^1)^3$, on which the orientation cocycle is equal to $+1$. 

Define the map $\phi_x: \bbS^1 \rightarrow \bbS^1$ by $\phi_x(\xi):=\phi(\xi,x)$, which is measurable for almost every $x \in X$ by~\cite[Lemma 2.6]{fisher:morris:whyte}. As a consequence of Proposition~\ref{prop:proportionality:eulero} we obtain
\[
\int_{\pi_0(\Gamma_g)\backslash \textup{PSL}(2,\R)} \int_X \orien(\phi_x(\overline{g}\cdot\xi_0),\phi_x(\overline{g}\cdot \xi_1),\phi_x(\overline{g}^{-1}\cdot \xi_2))d\mu_X(x)d\mu_0(\bar g)=1,
\]
for almost every $(\xi_0,\xi_1,\xi_2) \in (\bbS^1)^3$. We remind the reader that $\mu_0$ is the induced normalized probability measure by the Haar measure of $\textup{PSL}(2,\R)$ on the quotient $\pi_0(\Gamma) \backslash \textup{PSL}(2,\mathbb{R})$. The previous equality implies that 
\[
 \orien(\phi_x(\overline{g}\cdot \xi_0),\phi_x(\overline{g}\cdot \xi_1),\phi_x(\overline{g}\cdot \xi_2))=1,
\]
for almost every $\overline{g} \in \pi_0(\Gamma_g) \backslash \textup{PSL}(2,\R)$ and almost every $x \in X$. By the equivariance of the map $\phi$ with respect to the cocycle $\sigma$ the equality above actually holds for almost every $g \in \textup{PSL}(2,\R)$ and for almost every $x \in X$. This means that the map $\phi_x:\bbS^1 \rightarrow \bbS^1$ is order-preserving for almost every $x \in X$. As proved by Iozzi~\cite{iozzi02:articolo}, this implies that the map $\phi_x$ must agree almost everywhere with a orientation-preserving homeomorphism $f(x) \in \homs$. In this way we get a function $f:X \rightarrow \homs$. Since $X$ is standard Borel, also in this case the measurability of $f$ is guaranteed by the fact that $\hat \phi:X \rightarrow \textup{Meas}(\bbS^1,\bbS^1), \hspace{5pt} \hat \phi(x):=\phi_x$ is measurable by~\cite[Lemma 2.6]{fisher:morris:whyte} and its image lies in $\homs \subset \textup{Meas}(\bbS^1,\bbS^1)$ (see also~\cite[Proposition~3.2]{sauer:articolo}). Then, by applying~\cite[Proposition~3.2]{sauer:articolo}, the thesis follows as in Theorem~\ref{mostow}.
\end{proof}

As in the case of the volume, by Proposition~\ref{eu:representations} we already know that among the maximal cocycles we may find the ones associated to maximal representations. Moreover, since in the case of $(\Gamma_g, \Gamma_g)$-couplings our Euler number only differs by a multiplicative constant from Bader-Furman-Sauer's one, we have the following corollary:

\begin{cor}\label{cor:massimalita:coupling:eulero}
Let $\Sigma_g$ be a closed surface of genus $g \geq 2$ and let $\Gamma_g=\pi_1(\Sigma_g)$. Fix a hyperbolization $\pi_0:\Gamma_g \rightarrow \textup{PSL}(2,\R)$ and suppose that $\Gamma_g$ acts on $\bbS^1$ via $\pi_0$. Let $(\Omega,m_\Omega)$ be an ergodic integrable $(\Gamma_g,\Gamma_g)$-coupling with associated right measure equivalence cocycle $\alpha_\Omega:\Gamma_g \times \Gamma_g \backslash \Omega \rightarrow \Gamma_g$. Then it holds 
$$
|\eu(\alpha_\Omega)| = |\chi(\Sigma_g)|
$$
\end{cor}
\begin{proof}
Since we know by~\cite[Corollary~4.12]{sauer:articolo} that integrable self couplings satisfy $\eu(\Omega) = \pm \Vol(\pi_0(\Gamma_g) \backslash \hyp^2)$, by Remark~\ref{oss:nuova:uniforme}, we get 
$$
|\eu(\alpha_\Omega)| = \left|\frac{\eu(\Omega)}{2\pi}\right| = \frac{\Vol (\pi_0(\Gamma_g) \backslash \mathbb{H}^2)}{2\pi}= |\chi(\Sigma_g)| \ ,
$$
where in the last equality we used Gauss-Bonnet theorem. This concludes the proof.
\end{proof}



\bibliographystyle{amsalpha}
\bibliography{biblionote}
\end{document}